\documentclass{article}
\usepackage{arxiv}

\usepackage[normalem]{ulem} 

\usepackage[utf8]{inputenc} 
\usepackage[T1]{fontenc}    
\usepackage{url}            
\usepackage{booktabs}       
\usepackage{amsfonts}       
\usepackage{microtype}      
\usepackage{color}
\usepackage{lipsum}
\usepackage{amsthm}
\usepackage{indentfirst}
\usepackage{graphicx}
\usepackage{epstopdf}
\usepackage{fourier}
\usepackage{bm}
\usepackage{bbm} 
\usepackage{mathtools}
\usepackage{latexsym,enumerate}
\usepackage{ulem}

\newcommand{\comment}[1]{}

\newcommand{\BEA}{\begin{eqnarray}}
\newcommand{\EEA}{\end{eqnarray}}

\newcommand{\BR}{\mathbb{R}}

\newtheorem{lemma}{Lemma}[section]
\newtheorem{theorem}{Theorem}[section]
\newtheorem{proposition}{Proposition}[section]
\newtheorem{definition}{Definition}[section]
\newtheorem{corollary}{Corollary}[section]

\newtheorem{remark}{Remark}[section]

\title{Spectral Convergence of Symmetrized Graph Laplacian on manifolds with boundary}

\author{
  J. Wilson Peoples \\
  Department of Mathematics \\
  The Pennsylvania State University, University Park, PA 16802, USA\\
  \texttt{peoplesjwilson@gmail.com} \\
  \And
  John Harlim \\
  Department of Mathematics, Department of Meteorology and Atmospheric Science, \\ Institute for Computational and Data Sciences \\
  The Pennsylvania State University, University Park, PA 16802, USA\\
  \texttt{jharlim@psu.edu} \\
}

\begin{document}

\maketitle

\begin{abstract}
 We study the spectral convergence of a symmetrized graph Laplacian matrix induced by a Gaussian kernel evaluated on pairs of embedded data, sampled from a manifold with boundary, a sub-manifold of $\mathbb{R}^m$. Specifically, we deduce the convergence rates for eigenpairs of the discrete graph-Laplacian matrix to the eigensolutions of the Laplace-Beltrami operator that are well-defined on manifolds with boundary, including the homogeneous Neumann and Dirichlet boundary conditions. For the Dirichlet problem, we deduce the convergence of the \emph{truncated graph Laplacian}, whose convergence is recently observed numerically in applications, and provide a detailed numerical investigation on simple manifolds. Our method of proof relies on a min-max argument over a compact and symmetric integral operator, leveraging the RKHS theory for spectral convergence of integral operator and a recent asymptotic result of a Gaussian kernel integral operator on manifolds with boundary.
\end{abstract}


\section{Introduction}

The graph Laplacian is a popular tool for unsupervised learning tasks, such as dimensionality reduction \cite{belkin2003laplacian,cl:2006}, clustering \cite{ng2002spectral,von2008consistency}, and semi-supervised learning tasks \cite{belkin2002semi,calder2020properly}. Given a finite set of sample points $X=\{x_1,\ldots, x_n\}$, the graph Laplacian matrix is constructed by characterizing the affinity between any two points in this set. Under a manifold assumption of $X\subset \mathcal{M}$, where $\mathcal{M}$ is a $d$-dimensional Riemannian sub-manifold of $\mathbb{R}^m$, a common method of constructing a graph Laplacian matrix is through kernels that evaluate the distance between pairs of points in $X$ by an ambient topological metric. A motivation for such an approach is that when the kernel is local \cite{cl:2006,berry2016local}, including the case of compactly supported \cite{trillos2020error,ct:2022}, the geodesic distance (a natural metric for nonlinear structure) between pairs of points on the manifold can be well approximated by the Euclidean distance of the embedded data in $\mathbb{R}^m$.  This fact established the graph Laplacian as a foundational manifold learning algorithm. Particularly, it allows one to access the intrinsic geometry of the data using only the information of embedded data when the graph Laplacian matrix is consistent with the continuous counterpart, the Laplace-Beltrami operator, in the limit of large data.

In the past two decades, many convergence results have been established to justify this premise. To the best of our knowledge, there are three types of convergence results that have been reported. The first mode of convergence results is in the pointwise sense, as documented in \cite{belkin2005towards,hein2005graphs,s:2006,hein2006uniform,ting2011analysis,bh:2016}. In a nutshell, this mode of convergence shows that in high probability, the graph Laplacian  acting on a smooth function evaluated on a data point on the manifold converges to the corresponding smooth differential operator, the Laplace-Beltrami (or its variant weighted Laplace operator), applied to the corresponding test function evaluated on the same data point, as the kernel bandwidth parameter $\epsilon \to 0$ after $n\to \infty$. The second mode of convergence is in the weak sense \cite{hein2006uniform,vaughn2024diffusion,v:2020} which characterizes the convergence of an energy-type functional induced by the graph Laplacian to the corresponding Dirichlet energy of the Laplace-Beltrami operator defined on appropriate Hilbert space. The third convergence mode is spectral consistency. The first work that shows such consistency is \cite{belkin2007convergence} without reporting any error rate. Subsequently, a convergence rate for the graph Laplacian constructed under a deterministic setting was shown in \cite{burago2014graph} using Dirichlet energy argument. With the language of optimal transport, the same approach is extended to probabilistic setting in \cite{trillos2020error}, where the authors reported spectral convergence rates of $\mathcal{O} \left(n^{-1/2d} \right)$. In \cite{ct:2022}, the above rate was improved in the $\epsilon$-graph setting by a logarithmic factor. Moreover, for a particular scaling of $\epsilon$, the authors obtained convergence rates of $\mathcal{O} \left(n^{-1/(d+4)} \right)$, a significant improvement for higher dimensions. The above results rely heavily on the connection between eigenvalues and pointwise estimates, which arises due to the min-max theorem. In \cite{wormell2021spectral}, the authors are able to obtain convergence rates in the more general setting of a weighted graph Laplacian, when the manifold is a higher dimensional torus. To achieve this result, a different approach was taken, using spectral stability results of perturbed PDEs, as well as Hardy space embedding results. The rates they obtain are $\mathcal{O}\left( n^{-2/(d + 8) + o(1)} \right)$. Recently, in \cite{dunson2021spectral}, convergence of eigenvectors was proved with rates in the $L^\infty$ sense. This rate is obtained by first deducing an $L^2$ error bound on the eigenvectors, followed by Sobolev embedding arguments. While the rates are similar to those found in \cite{ct:2022}, their method of proof is different. In \cite{calder2022lipschitz}, eigenvector convergence in the $L^\infty$ sense is also found, but it is proved as a consequence of more general Lipschitz regularity results. The rate found in this sense coincides with the $L^2$ convergence rate found in \cite{ct:2022}. Recently, $\Gamma$-convergence, which is a similar idea to weak convergence, has gained popularity, and was used in \cite{trillos2018variational} to prove spectral convergence. 

We should point out that the spectral convergence results mentioned above are only valid on closed manifolds, i.e. compact manifolds without boundary (\cite{l:2003}, pg. 27). For compact manifolds with boundary, it is empirically observed that eigenvectors of the graph Laplacian approximate eigenfunctions of Neumann Laplacian \cite{cl:2006}. Theoretical studies for such a case are available only recently. Particularly, a spectral consistency result was reported in \cite{singer2017spectral} with no convergence rate. In a deterministic setup as in \cite{burago2014graph}, spectral convergence rate with a specific kernel function that depends on the ``reflected geodesics'' was reported in \cite{lu2020graph}. In a probabilistic setup, a spectral convergence rate with a kernel induced by the point integral method, designed for solving PDEs on manifolds, was reported in \cite{ts:2020}. 
 
While these results are encouraging, from a practical standpoint the natural Neumann boundary condition is too restrictive.  Beyond the traditional data science applications, such as dimensionality reduction and clustering, there is a growing interest in using the eigenfunctions of the Laplace-Beltrami operator to represent functions (and operators) defined on manifolds to overcome the curse of ambient dimension with the traditional scientific computational tools. For example, in the Bayesian inverse problem of estimating parameters in PDEs defined on unknown manifolds \cite{harlim2019Kernel,hs:2022}, it was shown that an effective estimation of the distribution of the parameters can be achieved with a prior that is represented by a linear superposition of the eigenvectors of the corresponding graph Laplacian matrix induced by the available data. When the PDE solutions satisfy the Dirichlet boundary condition, as considered in \cite{hs:2022}, it is numerically observed that eigenvectors of an appropriately truncated graph Laplacian matrix approximate eigenfunctions of the Dirichlet Laplacian. The idea of using a truncated graph Laplacian matrix was first proposed by \cite{tg:2019} for solving nonhomogeneous Dirichlet boundary value problems whose solutions correspond to statistical quantities such as the mean first-passage time and committor function that are useful for characterizing chemical reaction applications. These empirical findings motivate the theoretical study in this paper.

In this paper, we study the spectral convergence of the graph Laplacian matrix constructed using the Gaussian kernel on data that lie on a compact manifold with boundary. As we already mentioned, one of the main goals here is to show convergence of the Dirichlet Laplacian, which to our knowledge has not been documented. We will show that the general strategy considered in this work also provides an alternative proof for the closed manifold setting as well as the compact manifold with Neumann boundary conditions. Much of the heavy machinery needed for our arguments correspond to convergence results in the Reproducing Kernel Hilbert Space (RKHS) setting, which are well known and proved in the literature {\color{black}\cite{rosasco2010learning}.} Particularly, to deduce the difference between eigenvalues of the Laplace-Beltrami operator, $\Delta=-\text{div}\circ\nabla$, and the graph Laplacian matrix $\tilde{L}_{\epsilon,n}$, induced by the data in $X$, we apply a min-max argument over the following identity,
\BEA
\quad \quad\|\nabla f\|_{L^2(\mathcal{M})}^{\color{black}^2} - \lambda(\tilde{L}_{\epsilon,n}) = \underbrace{ \|\nabla f\|_{L^2(\mathcal{M})}^{\color{black}^2} - \langle L_\epsilon f, f\rangle_{L^2(\mathcal{M})} }_{\text{approximation error}} 
+  \underbrace{\langle L_\epsilon f, f\rangle_{L^2(\mathcal{M})}-\lambda(\tilde{L}_{\epsilon,n}).}_{\text{discretization error}} \label{generalequality} 
\EEA
where $\lambda(\tilde{L}_{\epsilon,n})$ denotes the eigenvalue of the matrix $\tilde{L}_{\epsilon,n}$ and $L_\epsilon$ denotes an integral operator induced by the discrete Graph Laplacian construction in the limit of large data. {\color{black}Here, the Dirichlet energy is identical to the weak approximation to the Laplace-Beltrami operator, $\|\nabla f\|^2_{L^2(\mathcal{M})} = \langle\Delta f,f\rangle_{L^2(\mathcal{M})}$, which from the integration by parts formula with a vanishing boundary integral term, which occurs both when the manifold has no boundary, and when the manifold has boundary but the integrated function satisfies either homogeneous Neumann or Dirichlet boundary conditions.} In this paper, we consider an integral operator of the form $L_\epsilon = c( I - \hat{K}_\epsilon)$, for some constant $c>0$ and a compact, self-adjoint operator $ \hat{K}_\epsilon$ on $L^2(\mathcal{M})$, which simplifies the min-max arguments for the proof relative to that in unnormalized Graph Laplacian formulations as shown in \cite{burago2014graph,trillos2020error,ct:2022}. Instead, the difficulty in our proof is transferred to the RKHS setting, where \cite{rosasco2010learning} have proved the necessary convergence results between discrete estimator and integral operator. The symmetry requirement in our formulation is because our arguments for bounding the discretization error in Equation \eqref{generalequality} rely on a well-established spectral consistency between an integral operator and its discrete Graph Laplacian matrix in the RKHS setting \cite{rosasco2010learning}. One of the motivations for using the result in \cite{rosasco2010learning} in comparing the spectra of discrete and continuous operators is because it is based on a concrete Nystr\"om interpolation and restriction operators that are commonly used in practical kernel regression algorithms in alternative to the more abstract interpolation and restriction operators introduced in \cite{burago2014graph,trillos2020error,ct:2022}. In fact, in Section \ref{sec6}, we will use Nystr\"om interpolation to approximate the estimated eigenvectors on the same grid points where the reference solution is available when the former estimates are generated based on different sample points (that can be either randomly sampled or well-sampled).
To bound the approximation error in \eqref{generalequality}, we will leverage a recent asymptotic expansion result for the integral operator with Gaussian kernel on compact manifolds with boundary in \cite{vaughn2024diffusion,v:2020}.
We should point out that while one can consider $L_\epsilon$ corresponding to the continuous version of the symmetric normalized graph Laplacian formulation, i.e., 
$L_\epsilon f(x) = c(f(x) - d(x)^{-1/2}K_\epsilon ({f(x)d^{-1/2}(x)})$ for some $c>0$, this formulation only works for closed manifolds {\color{black}(compact manifolds with no boundary, such as sphere or torus)}. For manifolds with boundary, this formulation induces biases in the approximation error, i.e., $\|\nabla f\|_{L^2(\mathcal{M})} - \langle L_\epsilon f, f\rangle_{L^2(\mathcal{M})} \not\to 0$ as $\epsilon\to 0$. In fact, it was shown in \cite{vaughn2024diffusion,v:2020} that the approximation error converges as $\epsilon\to 0$ when $L_\epsilon$ is either $L_\epsilon f(x) = c\big(d(x)f(x) - K_\epsilon f(x)\big)$ or $L_\epsilon f(x) = c\big(f(x) - d(x)^{-1}K_\epsilon f(x)\big)$, which is induced by the unnormalized graph Laplacian or the random-walk normalized graph Laplacian formulations, respectively \cite{chung1997spectral}. Since neither formulations meet our condition of $L_\epsilon = c( I - \hat{K}_\epsilon)$ with a symmetric $\hat{K}_\epsilon$, we will devise a symmetric $\hat{K}_\epsilon$ that allows for a consistent approximation in Equation \eqref{generalequality}, even when the sampling distribution of $X$ is non-uniform. 
 
The main result in this paper can be summarized less formally as follows. For closed manifolds and manifolds under Neumann boundary conditions, we conclude that with high probability,
\BEA
\Big|\lambda_i(\Delta) -\lambda_i(\tilde{L}_{\epsilon,n})\Big| = \mathcal{O}\left( \frac{\sqrt{\log n}}{\epsilon^{d/2+1}n^{1/2}}\right) + \mathcal{O}\left(\epsilon^{1/2} \right), \notag
\EEA
as $\epsilon \to 0$ after $n\to\infty$ (see Theorem~\ref{spectralconvneumann}). Balancing these error bounds, we achieve spectral error estimate of order 
$
\sqrt{\epsilon} \sim \left( n^{-1}\log n \right)^{\frac{1}{2d+6}},
$ 
which is slower than the results presented for closed manifolds in \cite{trillos2020error,ct:2022,wormell2021spectral}. The slower convergence rate here is due to the use of the weak convergence rate \cite{vaughn2024diffusion,v:2020} to overcome the blowup of pointwise convergence near the boundary \cite{cl:2006}. For closed manifolds, one can improve the approximation error rate by using pointwise convergence result such that the overall error is $(n^{-1}\log n )^{\frac{1}{d+4}}$, which is equivalent to the improved rate reported in \cite{ct:2022}. For manifolds with Dirichlet boundary conditions, if we replace $\tilde{L}_{\epsilon,n}$ with a truncated graph Laplacian matrix of size $n_1\times n_1$, where $n_1< n$ corresponds to the number of interior points whose distance to the boundary,  $\partial \mathcal{M}$, is at least $\epsilon^{\gamma}$ with $\gamma \to 1/2$, the convergence rate is found to be $({n}^{-1}\log n)^{\frac{1}{2d+6}}$ (see Theorem~\ref{spectralconvdiri uniform}). Given these spectral error bounds, we follow the method of proof in \cite{ct:2022} to obtain the $L^2$-convergence of the eigenvectors (corresponding to eigenvalues of multiplicity $k\geq 1$) with an error rate 
$
\epsilon^{1/4} \sim \left(n^{-1}\log n\right)^{\frac{1}{4d + 10}}
$ 
in high probability,  after rescaling $\epsilon$ in terms of $n$ to achieve an optimal rate of convergence (see Theorem~\ref{conveigvec}) {\color{black}for the Neumann case.} Again, this can be improved in the case of closed manifolds. {\color{black}For the Dirichlet case, the $L^2$-convergence rate for the eigenvectors is similar to that of the eigenvalue error rate, $({n}^{-1}\log n)^{\frac{1}{2d+6}}$ (see Theorem~\ref{conveigvec-dirichlet-supplement} and Remark~\ref{remonDir}).}
Similar rates are also achieved for non-uniformly sampled data (see Theorems~\ref{spectconvwithq} and \ref{conveigvecwithq} for errors of the eigenvalues and eigenvectors estimations). 

The remainder of this paper is organized as follows. In Section~\ref{sec2}, we give a brief overview of the discrete and continuous estimators for the Laplace-Beltrami operator. In Sections~\ref{neumann-section} and \ref{dirichlet-section}, we present our main results for Neumann and Dirichlet boundary conditions with uniformly sampled data, respectively. To improve the readability, we only present the proof of the spectral convergence results in the main text. We report some intermediate results and the eigenvector convergence proof in Appendices~\ref{rkhs-section}-\ref{sm:conveigvec}. In Section~\ref{sec6}, we document numerical simulations for the truncated graph Laplacian on simple manifolds and numerically demonstrate its consistency with the Dirichlet Laplacian. We close the paper with a summary in Section~\ref{sec7}. For completeness, we report the generalization to non-uniform sampling data in Appendix~\ref{sec5}.

\section{Discrete and Continuous Estimators} \label{sec2}
\noindent
In this section, we give a brief overview of the Laplace-Beltrami operator (with appropriate boundary conditions), the integral operators which serve as continuous estimators, and the matrix discretizations which are used for concrete estimation. Throughout the discussion, we emphasize several theoretical properties of such operators which are key to our results.  

Let $\mathcal{M}$ be a compact, connected, orientable, smooth $d$-dimensional Riemannian manifold embedded in $\mathbb{R}^m$. Consider finitely many data points $X:=\{x_1, \dots , x_n\}\subset \mathcal{M}$ sampled uniformly. Corresponding to the data points is a discrete measure $\mu_n$  given by the formula 
\begin{equation}
    \mu_n : = \frac{1}{n} \sum_{i=1}^n \delta_{x_i} \label{empiricalmeasure}.
\end{equation}
We define discrete inner product $\langle \cdot, \cdot \rangle_{L^2(\mu_n)}$ on functions $u, \tilde{u}: X \to \mathbb{R}$ by 
\begin{equation}
    \langle u, \tilde{u} \rangle_{L^2(\mu_n)} = \int_{\mathcal{M}} u(x) \tilde{u}(x) d\mu_n(x) = \frac{1}{n} \sum_{i=1}^n u(x_i) \tilde{u}(x_i).\notag
\end{equation}
We denote the usual inner product on $L^2(\mathcal{M})$ by $\langle \cdot, \cdot \rangle_{L^2(\mathcal{M})}$. 
\subsection{Laplace-Beltrami and Kernel Integral Operators} \noindent
When $\mathcal{M}$ has no boundary, the Laplacian $\Delta: C^\infty (\mathcal{M}) \to C^\infty (\mathcal{M})$ defined by
\begin{equation}
    \Delta f : = -  \mbox{div}\circ(\nabla f), \notag
\end{equation}
is positive, semi-definite with eigenvalues $0=\lambda_1 < \lambda_2 \leq \dots$. Moreover, there exists an orthonormal basis for $L^2(\mathcal{M})$ consisting of smooth eigenfunctions of $\Delta$ (see Theorem.1.29 of \cite{rosenberg1997laplacian}). The $\ell$-th eigenvalue of the Laplacian, $\lambda_\ell$, satisfies the following min-max condition: 
\begin{equation}
    \label{Continuous courant fisher}
    \lambda_\ell =\min_{S \in \mathcal{G}_\ell} \max_{f \in S \setminus \{0 \}} \frac{\langle \Delta f , f \rangle_{L^2(\mathcal{M})}}{\|f \|^2_{L^2(\mathcal{M})}},
\end{equation}
where $\mathcal{G}_\ell$ is the set of all linear subspaces of $C^\infty(\mathcal{M})$ of dimension $\ell$. This property is key to the proof of spectral convergence. 

A popular method used to estimate these spectra from data points $X$ is by solving an eigenvalue problem of a graph Laplacian type matrix. For example, given a smooth symmetric kernel $k_\epsilon: \mathcal{M} \times \mathcal{M} \to \mathbb{R}$, where $k_\epsilon$ is of the form $k_\epsilon(x,y) = h\left( \frac{\|x - y\|^2}{\epsilon}\right)$ for some exponentially decaying function $h:[0, \infty) \to [0, \infty)$ with first two derivatives also exponentially decaying, we define the corresponding integral operator $K_\epsilon: L^2(\mathcal{M}) \to L^2(\mathcal{M})$ by 
\begin{equation*}
    K_\epsilon f(x) = \epsilon^{-d/2}\int_\mathcal{M} k_\epsilon (x,y) f(y) dV(y),
\end{equation*}
where $V$ denotes the volume form inherited by $\mathcal{M}$ from the ambient space $\mathbb{R}^m$. For closed manifolds (as well as away from the boundary if the manifolds have boundary), it is well known that (see Lemma 8 of \cite{cl:2006}) one has the following pointwise asymptotic expansion
\BEA
  (K_\epsilon f)(x) = m_0 f(x) + \epsilon \frac{m_2}{2} (\omega(x)f(x) - \Delta f(x)) + \mathcal{O}(\epsilon^2),\label{DMasymptotic}
\EEA
for $f\in C^3(\mathcal{M})$ and all $x\in \mathcal{M}$, where the constants $m_0, m_2$ depend on the kernel and $\omega$ depends on the geometry of $\mathcal{M}$. Choosing a Gaussian kernel
$
k_\epsilon(x,y)= \exp(-\frac{\|x-y\|^2}{\epsilon}),
$
we have that the constants $$m_0 = \int_{\mathbb{R}^d} \exp(-\|x\|^2)\,dx,\,\, m_2 = \int_{\mathbb{R}^d} x_1^2 \exp(-\|x\|^2)\,dx, \text{ and }  m_0=\frac{m_2}{2}.$$

Define by $\hat{K}_\epsilon : L^2(\mathcal{M}) \to L^2(\mathcal{M})$ the corresponding integral operator:
\begin{equation}
    \hat{K}_\epsilon f(x) = \int_\mathcal{M} \hat{k}_\epsilon (x,y) f(y) dV(y),\label{operatorK_epsilon}
\end{equation}
with a symmetric kernel,
\BEA
\hat{k}_\epsilon(x,y) : = \epsilon^{-d/2} k_\epsilon(x,y)\left(\frac{1}{2d(x)} + \frac{1}{2d(y)}\right),\label{kernelkhat}
\EEA
where $d(x) := K_\epsilon 1(x)$.
For convenience of later discussion, we also define operators 
\BEA
L_{\textrm{rw},\epsilon}:=  \frac{1}{\epsilon} \Big(I - d^{-1}K_\epsilon  \Big),\label{operatorL_rw}
\EEA
{\color{black}where subscript `rw' corresponds to random-walk (as $d^{-1}K_\epsilon$ corresponds to the transition kernel of a Markov chain defining random walk on data \cite{cl:2006}),}
and
\BEA
L_\epsilon : = \frac{1}{\epsilon}(I - \hat{K}_\epsilon ).\label{symmetricDM}
\EEA
We point out that the definition in Equation \eqref{operatorL_rw} is motivated by the pointwise convergence to Laplace-Beltrami operator $\Delta$ as one can verify with the asymptotic expansion in Equation \eqref{DMasymptotic}. In this paper, we instead consider the symmetric formulation in Equation \eqref{symmetricDM} since the symmetry allows us to conveniently use the RKHS theory and the spectral convergence result of the integral operator \cite{rosasco2010learning} for deducing our main result. Additionally, the symmetry leads to a min-max result for the eigenvalues of the corresponding operator $L_\epsilon$, which is key in relating its spectrum to that of the Laplace Beltrami operator. We outline this fact presently.

Note that $\hat{K}_\epsilon$ is a compact, self-adjoint, operator with positive definite kernel. Hence, the eigenvalues $\sigma^\epsilon_{\ell}$ of $\hat{K}_\epsilon$ are nonnegative, accumulate at $0$, and can be enumerated in decreasing order $\sigma^\epsilon_1 \geq \sigma^\epsilon_2 \geq \dots$ Moreover, 
\BEA
    \sigma_\ell^\epsilon = \max_{S \in (L^2(\mathcal{M}))_\ell} \min_{f \in S \setminus \{0 \}} \frac{\langle \hat{K}_\epsilon f, f \rangle_{L^2(\mathcal{M})}}{\|f \|^2_{L^2(\mathcal{M})}},\notag
\EEA
where $(L^2(\mathcal{M}))_\ell$ denotes the set of all $\ell$-dimensional subspaces of $L^2(\mathcal{M})$. Note that the above minimum is achieved when $S$ is the span of the first $\ell$ eigenfunctions of $\hat{K}_\epsilon.$ Since the kernel $\hat{k}_\epsilon$ is smooth, it follows that $L^2(\mathcal{M})$ has an orthonormal basis consisting of eigenfunctions of $\hat{K}_\epsilon$ which are smooth. Hence, we have the following: 
\BEA
    \sigma_\ell^\epsilon = \max_{S \in \mathcal{G}_\ell} \min_{f \in S \setminus \{0 \}} \frac{\langle \hat{K}_\epsilon f, f \rangle_{L^2(\mathcal{M})}}{\|f \|^2_{L^2(\mathcal{M})}}\notag,
\EEA
where $\mathcal{G}_\ell$ is the set of all linear subspaces of $C^\infty(\mathcal{M})$ of dimension $\ell$.
Let us denote the eigenvalues of $L_\epsilon$ by $\lambda^\epsilon_{\ell} = \frac{2}{m_2\epsilon}\left(1 - \sigma^\epsilon_\ell \right)$. It is easy to see that,
\begin{equation}
    \lambda^\epsilon_\ell = \min_{S \in \mathcal{G}_\ell} \max_{f \in S \setminus \{0 \}} \frac{\langle L_\epsilon f, f \rangle_{L^2(\mathcal{M})}}{\|f \|^2_{L^2(\mathcal{M})}}.\label{L-min-max}
\end{equation}
In \cite{vaughn2024diffusion,v:2020}, it is proved that $L_{rw,\epsilon}$ shares this property for compact manifolds with boundary, which are \textit{manifolds with boundary and of bounded geometry.} For a detailed definition of \textit{manifolds with boundary and of bounded geometry}, see Definition 3.3 in \cite{vaughn2024diffusion}. To informally summarize the key properties, manifolds with boundary and of bounded geometry have: uniform bounds on the curvature, the existence of a sufficiently small radius $r$ such that the exponential map $\exp_x : T_x \mathcal{M} \to  \mathcal{M}$  is a diffeomorphism on a ball of radius $r$,
{\color{black} and the existence of a normal collar, which we formally define presently. 
\begin{definition}
    A manifold $\mathcal{M}$ with boundary admits a normal collar if there exists sufficiently small $R>0$ such that the 
    mapping $\phi: \partial \mathcal{M} \times [0, R) \to \mathcal{M}$ defined by 
    $$
    \phi(x,t) = \exp_x(- t \eta_x),
    $$
     is a diffeomorphism onto its image. In the above, $\eta_x$ denotes the inward facing normal at $x$.
\end{definition}

Given this definition, we can adapt the results outlined above to $L_\epsilon$ as follows.}


\begin{lemma}
\label{weak convergence}
Let $\mathcal{M}$ be a compact, smooth manifold either without boundary, or with a $C^3$ boundary and normal collar. Let $\epsilon>0$, then for $f, \phi \in C^\infty(\mathcal{M})$
\begin{equation*}
    \int_{\mathcal{M}} L_\epsilon f(y)\phi(y) dV(y) =  \int_\mathcal{M}  \nabla f(y) \cdot \nabla \phi(y) dV(y) + \mathcal{O}(\epsilon^{1/2}),
\end{equation*}
as $\epsilon\to 0$. 
\end{lemma}

\begin{proof}[Proof of Lemma~\ref{weak convergence}]
Under these conditions, the result from Corollary 5.3.1 in \cite{v:2020} for the random walks graph Laplacian holds 
\BEA
{\frac{1}{\epsilon}}\int_\mathcal{M} \phi(x) \left(I - \frac{K_\epsilon}{d}\right)f(x) dV(x) = \int_\mathcal{M} \nabla f   \cdot \nabla \phi dV + \mathcal{O}\left( \epsilon^{1/2} \right),\notag
\EEA
for smooth $f, \phi$. Using the definition in \eqref{operatorL_rw}, we can rewrite this equation as follows
\BEA
\langle L_{\textrm{rw},\epsilon}f , \phi \rangle_{L^2(\mathcal{M})} =  \int_\mathcal{M} \nabla f \cdot \nabla \phi dV + \mathcal{O}\left( \epsilon^{1/2} \right).\notag
\EEA
Notice therefore that 
\BEA
\langle (L_{\textrm{rw},\epsilon})^* \phi , f \rangle_{L^2(\mathcal{M})} =  \langle  f , (L_{\textrm{rw},\epsilon})^* \phi \rangle_{L^2(\mathcal{M})}  = \langle  L_{\textrm{rw}, \epsilon}f , \phi \rangle_{L^2(\mu_n)} =
\int_\mathcal{M}   \nabla f \cdot \nabla \phi  dV + \mathcal{O}\left( \epsilon^{1/2} \right) \notag
\EEA
Hence, the same is true for $\frac{1}{2} \left(L_{\textrm{rw},\epsilon} + (L_{\textrm{rw},\epsilon})^* \right)$. It is easy to check that this is precisely $L_\epsilon$. This completes the proof.  
\end{proof}

We note that one can collect the higher-order terms in Propositions  2.3.11 and 3.4.16 of \cite{v:2020} and verify that the constant of the order-$\epsilon^{1/2}$ term is independent of $\epsilon$. This fact will play a crucial role in the spectral convergence analysis where our $f$ will be smooth eigenfunctions of $L_\epsilon$.

When considering the estimation of the Dirichlet Laplacian, we introduce an additional parameter $\gamma$. In particular, fix $0 < \gamma < \frac{1}{2}$, and consider $n$ points sampled uniformly i.i.d. from $\mathcal{M}$, $X=\{x_1, \dots, x_n\}$. Denote by 
\BEA
\mathcal{M}_{\epsilon^\gamma} : = \left\{ x \in \mathcal{M} : \textrm{inf}_{y \in \partial \mathcal{M}} \| x - y\|_g > \epsilon^\gamma \right\}, \label{Mawayfrbdry}
\EEA
as the set of all points of distance at least $\epsilon^\gamma$ away from the boundary. Denote by $n_0$ the cardinality of $X \cap (\mathcal{M} \setminus \mathcal{M}_{\epsilon^\gamma})$, and denote by $n_1$ the cardinality of $X \cap \mathcal{M}_{\epsilon^\gamma}$. A key integral operator relevant for the estimation of the Dirichlet Laplacian is analogous to $L_\epsilon$, except that the integration occurs only over points sufficiently far away (on $\mathcal{M}_{\epsilon^\gamma}$) from the boundary. Define $L^\gamma_\epsilon: L^2(\mathcal{M}_{\epsilon^\gamma}) \to L^2(\mathcal{M}_{\epsilon^\gamma})$ by 
\BEA
L^\gamma_\epsilon f(x) = \frac{1}{\epsilon} \left( f(x) - \int_{\mathcal{M}_{\epsilon^\gamma}} \hat{k}_\epsilon(x,y)f(y) dV(y) \right). \label{truncated-integral-operator}
\EEA
The key theoretical properties for this operator are analogous to those of $L_\epsilon$ (such as the min-max formula for its eigenvalues), and hold using the same arguments. 

\subsection{Discretized Integral Operators} \noindent
Discretizing $L_\epsilon$ using the data set, we obtain a matrix $L_{\epsilon, n} : L^2(\mu_n) \to L^2(\mu_n)$ defined by 
\begin{equation}
    (L_{\epsilon, n} u)(x) = \frac{1}{\epsilon} \left( u(x) - \frac{1}{n} \sum_{i=1}^n \hat{k}_\epsilon (x,x_i) u(x_i)\right).\label{Lepsilonn}
\end{equation}
It is easy to see that $L_{\epsilon, n} $ is a positive definite, self-adjoint operator with respect to the inner product $L^2(\mu_n)$. As such, the eigenvalues of $L_{\epsilon,n}$ can be listed in increasing order: $0 \leq \lambda_1^{\epsilon,n} \leq \lambda_2^{\epsilon,n} \leq \ldots \leq \lambda^{\epsilon,n}_n$. In practice, unfortunately, this discretization is not directly accessible, since the evaluation of $\hat{k}_\epsilon(x_i,x_j)$ involves the computation of integrals $d(\cdot) = \epsilon^{-\frac{d}{2}} \int_\mathcal{M} k_\epsilon (\cdot,y) dV(y)$ evaluated on $x_i$ and $x_j$. To amend this, we introduce a second discretization. Let 
\begin{equation}\label{kernelktilde}
    \tilde{k}_{\epsilon,n}(x_i,x_j) = k_\epsilon (x_i,x_j) \left( \frac{1}{\frac{2}{n}\sum_{k=1}^n  k_\epsilon(x_i,x_k) } + \frac{1}{\frac{2}{n}   \sum_{k=1}^n k_\epsilon(x_j,x_k)} \right).
\end{equation}
The resulting discrete operator, which is computationally accessible, is given by,
\begin{equation*}
    \tilde{L}_{\epsilon,n}u(x) = \frac{1}{\epsilon} \left( u(x) - \frac{1}{n} \sum_{i=1}^n \tilde{k}_{\epsilon,n}(x,x_i)u(x_i) \right).
\end{equation*}
We denote the $i$th-eigenvalue of this matrix as $\tilde{\lambda}^{\epsilon,n}_{i}$.  

For Dirichlet Laplacian, we discretize \eqref{truncated-integral-operator} with the kernel in \eqref{kernelktilde}. This results in the $n_1 \times n_1$ submatrix of $\tilde{L}_{\epsilon, n}$ corresponding to points of distances of at least $\epsilon^\gamma$ away from the boundary. In particular, reorder $X=\{x_1, \dots, x_n\}$ so that the first $n_1$ points are a distance of at least $\epsilon^\gamma$ away from the boundary. We define $\tilde{L}^\gamma_{\epsilon,n}: L^2(\mu_{n_1}) \to L^2(\mu_{n_1})$ by 
\begin{equation}
    \tilde{L}^\gamma_{\epsilon,n}u(x) = \frac{1}{\epsilon} \left( u(x) - \frac{1}{n} \sum_{i=1}^{n_1} \tilde{k}_{\epsilon,n}(x,x_i)u(x_i) \right).\label{Legamma}
\end{equation}
Denote the eigenvalues of the above matrix by $\tilde{\lambda}^{\gamma, \epsilon, n}_i$. Since $\tilde{L}^\gamma_{\epsilon,n}$ is an $n_1 \times n_1$ truncation of  $\tilde{L}_{\epsilon,n}$, it is clear this matrix is also positive definite, and self-adjoint with respect to the $L^2(\mu_{n_1})$ inner product.

\section{Spectral Convergence Results for Neumann Boundary Conditions} \label{neumann-section}
In this section, we present our main results for the convergence of eigenvalues and eigenvectors of $\tilde{L}_{\epsilon,n}$ to $\Delta$ with homogeneous Neumann boundary conditions. Recall that when $\mathcal{M}$ is a compact manifold with boundary, a Neumann eigenfunction (resp. Neumann eigenvalue) is a function $f$ (resp. scalar $\lambda$) satisfying the following system of equations
\BEA 
\Delta f = \lambda f, \quad\quad \frac{\partial f}{\partial \hat{n}}{\color{black}\Big|_{\partial \mathcal{M}}} = 0,\notag
\EEA
{\color{black}where $\hat{n}$ denotes the unit normal vector to $\partial M$.}
Let us enumerate the Neumann eigenvalues $0 = \lambda_1 < \lambda_2 \leq \lambda_3 \leq \dots.$ We have the following min-max result: 
\BEA
\label{Neumann energy}
\lambda_i = \min_{S_i} \max_{f\in S_i\setminus\{0\}} \frac{\|\nabla f\|^2_{L^2}}{\|f\|^2_{L^2}},
\EEA
where $S_i$ is traditionally taken over all $i-$dimensional subspaces of $H^1(\mathcal{M})$. This minimum is achieved when $S_i$ is taken to be the span of the first $i$ eigenfunctions of $\Delta$ {\color{black}(the min-max formulation is consistent with the variational formulation in  \cite{colbois2013eigenvalues} and with the Rayleigh formulation for bounded domain in $\BR^n$ in \cite{ashbaugh1993universal,benguria2020sharp})}. Since it can be shown that the eigenfunctions of the Neumann Laplacian are smooth, the above formula holds when the minimization is taken over all $i-$dimensional spaces of smooth functions. In the setting below, the discrete estimator $\tilde{L}_{\epsilon,n}$ is constructed based on uniformly i.i.d samples in $X$. For non-uniform i.i.d.~case, see Appendix~\ref{sec5}.

\begin{theorem} \label{spectralconvneumann} (Convergence of Neumann eigenvalues) Fix $i \in \mathbb{N}$, and
let $\lambda_i$ be the $i$-th eigenvalue of the Neumann Laplacian. For $n>i$ and $\epsilon>0$ with probability greater than $1 - \frac{3}{n^2}$, we have that 
 \begin{equation*}
    |\lambda_{i} - \tilde{\lambda}^{\epsilon,n}_{i}| = \mathcal{O}\left( \frac{  \sqrt{\log n}}{\epsilon^{d/2+1}n^{1/2}} ,\epsilon^{1/2} \right).
\end{equation*}
\end{theorem}

{\color{black}
\begin{remark}
Throughout this paper, we use a shorthand big-O notation $\mathcal{O}(f,g)=\mathcal{O}(f)+\mathcal{O}(g)$, where the right hand implies as $f\to 0$ and $g\to 0$. In the above estimate (and in many other subsequent results), since the first error term of order, $ \frac{\sqrt{\log n}}{\sqrt{n} \epsilon^{d/2+1}}$, is computed for a fixed $\epsilon>0$, the big-oh notation means $n\to\infty$. Together with the second error term $\mathcal{O}(\epsilon^{1/2})$, we conclude that $\epsilon\to 0$ after $n\to\infty$. 
\end{remark}}

\begin{proof} Let $f \in C^{\infty}(\mathcal{M})$. Notice that we have 
\begin{equation*}
     \| \nabla f  \|^2_{L^2(\mathcal{M})} - \tilde{\lambda}^{\epsilon,n}_i = \| \nabla f  \|^2_{L^2(\mathcal{M})}  -  \langle L_\epsilon f, f \rangle_{L^2(\mathcal{M})}  +   \langle L_\epsilon f , f  \rangle_{L^2(\mathcal{M})} -  \tilde{\lambda}^{\epsilon,n}_i.
\end{equation*}
Fix any subspace $S \subseteq C^{\infty}(\mathcal{M})$ of dimension $i$. Maximizing both sides over all $f \in S \setminus \{0\}$ with $\|f\|_{L^2(\mathcal{M})} = 1$, the above equality still holds. Using subadditivity of the maximum, we obtain
\BEA
     \max_{f \in S, \|f\|=1} \| \nabla f  \|^2_{L^2(\mathcal{M})} - \tilde{\lambda}^{\epsilon,n}_i \leq \max_{f \in S, \|f\| = 1}\left( \| \nabla f  \|^2_{L^2(\mathcal{M})} -  \langle L_\epsilon f, f \rangle_{L^2(\mathcal{M})} \right)+  \max_{f \in S, \|f\|=1} \langle L_\epsilon f , f  \rangle_{L^2(\mathcal{M})} -  \tilde{\lambda}^{\epsilon,n}_i.\notag
\EEA
Since this inequality holds for any $i$-dimensional subspace $S$, choosing the subspace $S'$ for which the term $\max_{f \in S, \|f\|=1} \langle L_\epsilon f , f  \rangle_{L^2(\mathcal{M})}$ is at a minimum, the inequality again holds. Using min-max principle, this shows that 
\BEA
 \max_{f \in S', \|f\| = 1}  \| \nabla f  \|^2_{L^2(\mathcal{M})} - \tilde{\lambda}^{\epsilon,n}_i  \leq    \max_{f \in S', \|f\| = 1} \left( \| \nabla f  \|^2_{L^2(\mathcal{M})} - \langle L_\epsilon f, f \rangle_{L^2(\mathcal{M})} \right) +   \lambda^\epsilon_i - \tilde{\lambda}^{\epsilon,n}_i. \notag
\EEA
Notice that the left hand side of the above equation is certainly an upperbound for the minimum of $\textrm{max}_{f \in S, \|f\| = 1}  \| \nabla f  \|^2_{L^2(\mathcal{M})} - \tilde{\lambda}^{\epsilon,n}_{i}$ over all $i$ dimensional subspaces $S$ consisting of smooth functions. Hence, using min-max principle we obtain: 
\begin{equation}
   \lambda_{i} - \tilde{\lambda}^{\epsilon,n}_i \leq \max_{f \in S', \|f\| = 1} \left( \| \nabla f  \|^2_{L^2(\mathcal{M})} - \langle L_\epsilon f, f \rangle_{L^2(\mathcal{M})} \right)  +   |\lambda^\epsilon_i - \tilde{\lambda}^{\epsilon,n}_i|. \label{oneside}
\end{equation}
The first term on the right-hand side can be bounded via Lemma \ref{weak convergence}, while the second term is bounded using an adaptation of well known results in the RKHS literature (see Lemma \ref{adapted-rbv-supplement}) and Lemma~\ref{matrix-eigenvalues-supplement} that accounts for the approximation in \eqref{kernelktilde} of the kernel \eqref{kernelkhat}. The fact that the upper bound in Lemma~\ref{weak convergence} is of order-$\epsilon^{1/2}$ with constant that is independent to $\epsilon$ allows the min-max argument over $S'$, an $i$-dimensional subspace span of first $i$ eigenfunctions of $L_\epsilon$, to be valid. The proof is completed with a lower bound with similar arguments.
\end{proof}

The rate presented in the theorem above is slower than the available rate in literature. This is expected, and is due to the use of weak convergence result of Lemma \ref{weak convergence} in bounding the first error difference in \eqref{oneside}. A sharper convergence rate can be achieved by Cauchy-Schwartz and using the pointwise convergence of order-$\epsilon$  such as \eqref{DMasymptotic}. If one applies this strategy, one achieves a rate convergence rate of $\mathcal{O}\Big(\epsilon,\frac{\sqrt{\log n}}{\sqrt{n}\epsilon^{d/2 + 1}}\Big)$. Balancing these rates, we obtain an error of $\epsilon \sim (n^{-1}\log n)^{\frac{1}{d+4}}$, which agrees with the rate reported in \cite{ct:2022}, and is competitive with recent results in the literature.

We also obtain results for the convergence of eigenvectors, which can be improved for closed manifolds using the same argument as above. Since the proof mimics arguments in the literature \cite{ct:2022}, we leave it to Appendix~\ref{app:proof of eig}.
\begin{theorem}\label{conveigvec}
Let $\Delta$ denote the Neumann Laplacian. For any $\ell$, let $c_\ell$ be a constant such that $C' \left( \frac{\sqrt{\log n}}{\epsilon^{d/2+1}n^{1/2}} + \epsilon^{1/2}  \right) < c_\ell$, {\color{black} where $C'>0$ denotes the largest constant absorbed in the spectral convergence rate in Theorem~\ref{spectralconvneumann}.}
For any normalized eigenvector $u$ of $\tilde{L}_{\epsilon,n}$ with eigenvalue $\tilde{\lambda}^{\epsilon,n}_\ell$, there is a normalized eigenfunction $f$ of $\Delta$ with eigenvalue $\lambda_\ell$ such that with probability higher than $1- \frac{2k^2 + 4k + 6}{n^2}$,
\begin{equation*}
    \| R_n f - u \|_{L^2(\mu_n)} =  \mathcal{O}\left(\frac{ \sqrt{\log (n)} }{\epsilon^{ d/2+1}\sqrt{n}},\epsilon^{\gamma/2} \right),
\end{equation*}
as $\epsilon\to 0$ after $n\to\infty$, where $k$ is the geometric multiplicity of eigenvalue $\lambda_\ell$, and $\gamma$ can be chosen arbitrarily close to $1/2$. {\color{black}Here, $R_n:C^\infty(\mathcal{M})\to \BR^n$ denotes the restriction operator defined as $R_nf = (f(x_1),\ldots,f(x_n))$ for any $f\in C^\infty(M)$.} 
\end{theorem}

\section{Spectral Convergence Results for Dirichlet Boundary Conditions}
\label{dirichlet-section} 
In this section, we present convergence results of eigenvalues of the truncated graph laplacian to $\Delta$ satisfying Dirichlet boundary conditions. For the convergence of eigenvectors in this setting, see Theorem~\ref{conveigvec-dirichlet-supplement} for the details.

Recall now that a Dirichlet eigenfunction (resp. eigenvalue) is a function $f$ (resp. scalar $\lambda$) satisfying the following system of equations: 
\BEA 
\Delta f = \lambda f , \quad\quad f\vert_{\partial \mathcal{M}} = 0. \notag
\EEA
Let us enumerate the Dirichlet eigenvalues $0 < \lambda_1 < \lambda_2 \leq \lambda_3 \leq \dots $. In this case, the min-max result in \eqref{Neumann energy} is valid except that
$S_i$ is taken over all $i$ dimensional subspaces of $H_0^1(\mathcal{M})$, functions in $H^1(\mathcal{M})$ which vanish on the boundary $\partial \mathcal{M}$. This minimum is achieved when $S_i$ is taken to be the span of the first $i$ eigenfunctions of the Dirichlet Laplacian. Since it can be shown that the eigenfunctions of the Dirichlet Laplacian are smooth, the above formula holds when the minimization is taken over all $i$ dimensional spaces of smooth functions which vanish on the boundary of $\mathcal{M}$. The main result is stated as follow:

\begin{theorem}
\label{spectralconvdiri uniform}
Fix an eigenvalue $\lambda_i$, and suppose $X = \{x_1 , \dots , x_n\}$ is a dataset of $n$ points which are sampled uniformly from $\mathcal{M}$. Let $\epsilon>0$ and $0 < \gamma < 1/2$ be such that 
\[
n^{-1/2} \log(n) \ll \epsilon^{1/2} < \epsilon^\gamma \ll 1.
\]
Suppose that $n_1 > i$, where $n_1$ denotes the number of points whose distance from the boundary is at least $\epsilon^\gamma$.
Then with probability higher than  $1 - \frac{6}{n^2} - \frac{4(2 \sqrt{n} \log(n) + 2n)}{n^3}$,
 \begin{equation}
  |\lambda_{i} - \tilde{\lambda}^{\gamma,\epsilon,n}_{i}| = \mathcal{O}\left( \epsilon^{1/2}, \epsilon^{3 \gamma -1}, \frac{\sqrt{\log(n)}}{\epsilon^{d/2 + 1}\sqrt{n}}\right).\label{spectralconvdirirate}
\end{equation}
as $\epsilon \to 0$ after $n \to \infty$, where $\tilde{\lambda}^{\gamma,\epsilon,n}_{i}$ denotes eigenvalues of  $\tilde{L}^\gamma_{\epsilon,n}$ as defined in \eqref{Legamma}.
\end{theorem} 

Since our theoretical estimate of the eigenvalues for a truncated graph Laplacian relies on applying Proposition 10 from \cite{rosasco2010learning} to a dataset of points that lie sufficiently far away from the boundary, which requires the data to be sampled i.i.d. on this region (i.e., independent to the choice of $\gamma$), we define:
\begin{definition}\label{sampling2domain}
Let $X^0_\gamma$ and $X_\gamma^1$ denote datasets of size $n_0$ and $n_1$ points, respectively. They are two separate uniformly i.i.d. from $\mathcal{M} \setminus \mathcal{M}_{\epsilon^\gamma}$ and  $\mathcal{M}_{\epsilon^\gamma}$, respectively. Also, let $n=n_0+n_1$. 
\end{definition}

For a fixed $0< \gamma < 1/2$, a set of $n$ points sampled i.i.d. uniformly from $\mathcal{M}$ is related to the sampling setup in Definition~\ref{sampling2domain} in the following way. Note that the density of a uniform distribution over the entire manifold, $q: \mathcal{M} \to [0,1]$ can be rewritten as $  \frac{\textup{Vol}(\mathcal{M}_{\epsilon^\gamma} )}{\textup{Vol}(\mathcal{M})}\left( \frac{1}{\textup{Vol}(\mathcal{M}_{\epsilon^\gamma})} \chi_{\mathcal{M}_{\epsilon^\gamma}}(x) \right) +  \frac{\textup{Vol}(\mathcal{M} \setminus \mathcal{M}_{\epsilon^\gamma})}{\textup{Vol}(\mathcal{M})} \left( \frac{1}{ \textup{Vol}( \mathcal{M} \setminus \mathcal{M}_{\epsilon^\gamma} )} \chi_{\mathcal{M} \setminus \mathcal{M}_{\epsilon^\gamma}(x) }  \right)$, for all $x\in \mathcal{M}$. From this relation, it is clear that each trial of the uniform distribution over $\mathcal{M}$ corresponds to a uniform sample from $\mathcal{M}_{\epsilon^\gamma}$ with probability $\frac{\textup{Vol}(\mathcal{M}_{\epsilon^\gamma})}{\textup{Vol}(\mathcal{M})}$ or a uniform sample from $\mathcal{M} \setminus \mathcal{M}_{\epsilon^\gamma}$ with probability $1-\frac{\textup{Vol}(\mathcal{M}_{\epsilon^\gamma})}{\textup{Vol}(\mathcal{M})}$. 

 Hence, given $n$ points sampled uniformly i.i.d. from $\mathcal{M}$, the number of points in $\mathcal{M}_{\epsilon^\gamma}$ forms a binomial random variable with $n$ trials and success rate $\frac{\textup{Vol}(\mathcal{M}_{\epsilon^\gamma} )}{\textup{Vol}(\mathcal{M})}$, while the number of points in $\mathcal{M} \setminus \mathcal{M}_{\epsilon^\gamma}$ forms a binomial random variable with $n$ trials and success rate $\frac{\textup{Vol}(\mathcal{M} \setminus \mathcal{M}_{\epsilon^\gamma})}{\textup{Vol}(\mathcal{M})}$. To simplify the notation, we assume $\textup{Vol}(\mathcal{M}) = 1$, without loss of generality.
 
 In the proposed sampling setup, an unbiased estimator for the normalization factor $d$ is given by 
\BEA
\quad\quad \tilde{d}_\gamma(x_i) &=& \frac{\textup{Vol} (\mathcal{M}_{\epsilon^\gamma})}{n_1} \sum_{i=1}^{n_1}  \epsilon^{-d/2} k_\epsilon (x_i, x_\ell)  
\notag \\ &+& \frac{\textup{Vol}(\mathcal{M} \setminus \mathcal{M}_{\epsilon^\gamma})}{n_0} \sum_{\ell = n_1+1}^{n}  \epsilon^{-d/2} k_\epsilon(x_i,x_\ell).  \label{dgamma}
\EEA
Using this identity, we define the symmetrized kernel as,
$$
\tilde{k}^\gamma_\epsilon (x_i,x_j) := \epsilon^{-d/2}k_\epsilon(x_i,x_j) \left( \frac{1}{2 \tilde{d}_{\epsilon^\gamma}(x_i)} + \frac{1}{2 \tilde{d}_\gamma (x_j)} \right).
$$
Denote by $\tilde{L}^\gamma_{\epsilon,n}$ the following $n_1 \times n_1$ matrix: 
\BEA
\tilde{L}^\gamma_{\epsilon,n} u (x) = \frac{1}{\epsilon} \left( u(x) -  \frac{\textup{Vol}(\mathcal{M}_{\epsilon^\gamma})}{n_1}\sum_{i=1}^{n_1} \tilde{k}^\gamma_\epsilon(x,x_i)u(x_i) \right) . \label{Lrtilde}
\EEA
Here, we used the same notation $\tilde{L}^\gamma_{\epsilon,n}$ as in \eqref{Legamma} for the following reason.
We note that for a fixed $\gamma$, given $X$, a set of $n$ points sampled uniformly i.i.d.~from $\mathcal{M}$, denoting by $n_0$ the number of points in $\mathcal{M} \setminus \mathcal{M}_{\epsilon^\gamma}$ and denoting by $n_1$ the number of points in $\mathcal{M}_{\epsilon^\gamma}$, the unbiased estimators for $\textup{Vol}(\mathcal{M}_{\epsilon^\gamma})$ and $\textup{Vol}(\mathcal{M} \setminus \mathcal{M}_{\epsilon^\gamma})$ are given by $ \frac{n_1}{n}$ and $\frac{n_0}{n}$, respectively. Using such estimates to replace each volume term in \eqref{dgamma} and \eqref{Lrtilde}, one sees that the kernel $\tilde{k}^\gamma_\epsilon = \tilde{k}_\epsilon$  and \eqref{Lrtilde} becomes equivalent to the $n_1 \times n_1$ truncation of  $\tilde{L}_{\epsilon,n}$.

To attain the desired error bound, we define the following intermediate matrix of size $n_1 \times n_1$, denoted by $L^\gamma_{\epsilon,n}$, defined as,  
$$
L^\gamma_{\epsilon,n} u(x) = \frac{1}{\epsilon} \left( u(x) - \frac{1}{n_1} \sum_{i=1}^{n_1}  \textup{Vol}(\mathcal{M}_{\epsilon^\gamma}) \hat{k}_\epsilon (x,x_i) u(x_i) \right) .
$$
We remark that since the definition of $\hat{k}_\epsilon$ involves computing integrals over the manifold, such a matrix is not accessible numerically.
Label the $i$-th eigenvalues of $L^\gamma_\epsilon$ (defined in \eqref{truncated-integral-operator}), $L^\gamma_{\epsilon,n}$, and $\tilde{L}^\gamma_{\epsilon,n}$ (defined in \eqref{Lrtilde}) to be $\lambda^{\gamma,\epsilon}_i$, $\lambda^{\gamma,\epsilon,n}_i$ and $\tilde{\lambda}^{\gamma,\epsilon,n}_i$, respectively. We first have the following lemma, which applies the result from \cite{rosasco2010learning} to those points sampled uniformly from $\mathcal{M}_{\epsilon^\gamma}$.
\begin{lemma}
\label{dirichlet rosasco}
Suppose $n_1$ points $X_{\gamma}^0 :=\{x_1, \dots , x_{n_1}\}$ are sampled uniformly from $\mathcal{M}_{\epsilon^\gamma}$. For any $\epsilon>0$, with probability higher than $1 - \frac{2}{(n_0 + n_1)^2}$, 
$$
\sup_{1 \leq i \leq n_1} |\lambda^{\gamma,\epsilon}_i - \lambda^{\gamma,\epsilon,n}_i| = \mathcal{O}\left( \frac{\sqrt{\log(n_0 + n_1)}}{\epsilon^{d/2+1} \sqrt{n_1}} \right). 
$$

\end{lemma} 
\begin{proof}
Notice that $L_\epsilon^\gamma$ can be rewritten more suggestively as 
\BEA
L^\gamma_\epsilon f(x) = \frac{1}{\epsilon} \left( f(x) - \int_{\mathcal{M}_{\epsilon^\gamma}} \textup{Vol}(\mathcal{M}_{\epsilon^\gamma}) \hat{k}_\epsilon(x,y) f(y)   \chi_{\mathcal{M}_{\epsilon^\gamma}}(y) \frac{dV(y)}{\textup{Vol}(\mathcal{M}_{\epsilon^\gamma})} \right), \notag
\EEA
whence the result follows immediately from Proposition 10 of \cite{rosasco2010learning}, using the same argument as in Lemma \ref{adapted-rbv-supplement}. 
\end{proof}

The next result suggests that eigenvalues of $L^\gamma_{\epsilon,n}$ and $\tilde{L}^\gamma_{\epsilon,n}$ are close (see Appendix~\ref{pfofLemma4.2} for its detailed proof). 
\begin{lemma}
\label{truncated matrix eigenvalues}
Let $\epsilon>0$ and $0<\gamma< 1/2.$ Let $x_i$ be sampled as in Definition~\ref{sampling2domain}. For any $1 \leq i \leq n_1$, with probability higher than $1 - \frac{2n_1}{(n_0 + n_1)^3}$, 
$$
| \lambda_i^{\gamma,\epsilon,n} - \tilde{\lambda}_i^{\gamma,\epsilon,n}| = \mathcal{O}\left(\frac{\sqrt{\log(n_0 + n_1)}}{ \epsilon^{d/2+1} \sqrt{n_1}}, \frac{ \sqrt{\epsilon^\gamma} \sqrt{\log(n_0 + n_1)}}{\epsilon^{d/2+1} \sqrt{n_0}} \right).
$$
\end{lemma} 

Consider a smooth function $\hat{k}^{c,\gamma}_\epsilon : \mathcal{M} \times \mathcal{M} \to \mathbb{R}$, compactly supported in $\mathcal{M}_{\epsilon^\gamma} \times \mathcal{M}_{\epsilon^\gamma} $, such that 
\BEA
\int_{\mathcal{M}_{\epsilon^\gamma}} \int_{\mathcal{M}_{\epsilon^\gamma}} \left| \hat{k}_\epsilon(x,y) - \hat{k}^{c,\gamma}_\epsilon(x,y) \right|^2 dV(x) dV(y) <\epsilon^3,\label{kcminusk_diff}
\EEA
Since $\hat{k}_\epsilon$ is symmetric, it is without loss of generality to assume that $\hat{k}^{c,\gamma}_\epsilon$ is as well (Indeed, otherwise replace $\hat{k}^{c,\gamma}_\epsilon$ by $\frac{1}{2}\left( \hat{k}^{c,\gamma}_\epsilon(x,y) +  \hat{k}^{c,\gamma}_\epsilon(y,x)\right)$, and verify that this still approximates $\hat{k}_\epsilon$). Define an operator $L^{c,\gamma}_\epsilon: L^2(\mathcal{M}) \to L^2(\mathcal{M})$ by 
\BEA
L^{c,\gamma}_\epsilon f (x) = \frac{1}{\epsilon}\left( f(x) - \int_{\mathcal{M}_{\epsilon^\gamma}} \hat{k}^{c,\gamma}_\epsilon(x,y) f(y) dV(y) \right).\notag
\EEA
Since $\hat{k}^{c,\gamma}_\epsilon$ is smooth, it follows that the eigenvalues $\lambda_i^{c,\gamma,\epsilon}$ of $L^{c,\gamma}_\epsilon$ are given by
\BEA
\lambda_i^{c,\gamma,\epsilon} = \min_{S_i} \max_{f \in S_i \|f\|=1} \langle L^{c,\gamma}_\epsilon f , f \rangle_{L^2(\mathcal{M})}, \label{def_lambdac}
\EEA
where the minimum is taken over all $i$ dimensional subspaces of  $C^\infty(\mathcal{M})$. Since $\hat{k}^{c,\gamma}_\epsilon$ is compactly supported in $\mathcal{M}_{\epsilon^\gamma} \times \mathcal{M}_{\epsilon^\gamma}$, it follows that any eigenfunction of $L^{c,\gamma}_\epsilon$ vanishes on $\mathcal{M}\setminus\mathcal{M}_{\epsilon^\gamma}$. 
Hence, the above minimum can be taken over $C_0^\infty(\mathcal{M})$, smooth functions vanishing on the boundary. We have the following result on the eigenvalues of $L^{c,\gamma}_\epsilon$ and $L^\gamma_\epsilon$.
\begin{lemma}
\label{perturbation lemma} For any $\epsilon>0$, $0 < \gamma < 1/2$, let $\hat{k}^{c,\gamma}_\epsilon$ be defined to satisfy \eqref{kcminusk_diff}, then the difference between the $i$-th eigenvalues $\lambda^{c,\gamma,\epsilon}_i$ and $\lambda^{\gamma,\epsilon}_i$ of $L^{c,\gamma}_\epsilon$ and $L^\gamma_\epsilon$, respectively, are given by: 
\BEA
 \sup_i \left| \lambda^{c,\gamma,\epsilon}_i - \lambda_i^{\gamma,\epsilon} \right| \leq \epsilon^{1/2}.\notag
\EEA
\end{lemma}
\begin{proof}
 This follows immediately from Theorem 5 in \cite{rosasco2010learning} (Kato perturbation result).  Particularly,
 \BEA
  \sup_i \left| \lambda^{c,\gamma,\epsilon}_i - \lambda_i^{\gamma,\epsilon} \right| &\leq&  \| L_\epsilon^{c,\gamma} - L_\epsilon^\gamma\| = \sup_{f\in L^2(\mathcal{M}_{\epsilon^\gamma}),\|f\|_{L^2}=1}\|(L_\epsilon^{c,\gamma} - L^\gamma_\epsilon)f\|_{L^2(\mathcal{M}_{\epsilon^\gamma})} \notag \\ &\leq& \sup_{\|f\|=1}\frac{1}{\epsilon} \int_{\mathcal{M}_{\epsilon^\gamma}}\int_{\mathcal{M}_{\epsilon^\gamma}} \left| (\hat{k}_\epsilon(x,y) - \hat{k}_\epsilon^{c,\gamma}(x,y)) f(y) \right| dV(y) dV(x) \notag \\
  &\leq & \frac{1}{\epsilon} \left( \int_{\mathcal{M}_{\epsilon^\gamma}} \int_{\mathcal{M}_{\epsilon^\gamma}} \left| \hat{k}_\epsilon(x,y) - \hat{k}^{c,\gamma}_\epsilon(x,y) \right|^2 dV(x) dV(y)\right)^{1/2}, \notag 
 \EEA
 where we have used Cauchy-Schwarz in the last line. The main result follows immediately from the assumption in Equation \eqref{kcminusk_diff}.
\end{proof}
The final lemma needed before proving a spectral convergence result is to bound an error which is introduced by truncating an integral to points sufficiently far away from the boundary. 
\begin{lemma}
\label{small integral}
Let $f \in C_0^\infty(\mathcal{M})$, and $0< \gamma< 1/2$. Then 
$$
\left\langle \int_{\mathcal{M} \setminus \mathcal{M}_{\epsilon^\gamma}} \epsilon^{-d/2} k_\epsilon(\cdot,y) f(y) dV(y), f  \right\rangle_{L^2(\mathcal{M})} = \mathcal{O}\left(\epsilon^{3 \gamma}\right).
$$
\end{lemma}

\begin{proof}
The proof is a simple application of Fubini-Tonelli. We note that for any $y \in \mathcal{M} \setminus \mathcal{M}_{\epsilon^\gamma}$, since $f$ vanishes on the boundary and is smooth, and hence Lipshitz, it follows that $|f(y)| \leq L d_g (x_0, y) \leq L\epsilon^\gamma$, where $x_0 \in \partial\mathcal{M}$ is the closest point on the boundary to $y$. 
It follows by Tonelli's theorem that the order of integration can be exchanged, whence it's clear that 
\BEA
\left\langle \int_{\mathcal{M} \setminus \mathcal{M}_{\epsilon^\gamma}} \epsilon^{-d/2}k_\epsilon(\cdot,y) f(y) dV(y), f  \right\rangle_{L^2(\mathcal{M})}    =\int_{\mathcal{M} \setminus \mathcal{M}_{\epsilon^\gamma} } f(x) \int_\mathcal{M} \epsilon^{-d/2} k_\epsilon (x,y)f(y) dV(y) dV(x). \notag
\EEA
Using the expansion in Theorem 4.6 of \cite{vaughn2024diffusion}, 
\BEA
\int_\mathcal{M} \epsilon^{-d/2} k_\epsilon (x,y)f(y) dV(y) = m_0^\partial (x) f(x) 
+ \epsilon^{1/2} \underbrace{m_1^\partial(x)\left(n_x\cdot \nabla f(x)+ \frac{d-1}{2}H(x)f(x)\right)}_{=g(x)} + \mathcal{O}(\epsilon), \notag 
\EEA
where the constant in the big-oh notation depends on $f$.
Hence,
\BEA
\Bigg| \left\langle \int_{\mathcal{M} \setminus \mathcal{M}_{\epsilon^\gamma}} \epsilon^{-d/2} k_\epsilon(\cdot,y) f(y) dV(y), f  \right\rangle_{L^2(\mathcal{M})} \Bigg|  
\leq \int_{\mathcal{M} \setminus \mathcal{M}_{\epsilon^\gamma}} |m_0^\partial (x)| |f(x)|^2 dV(x)
+ \sqrt{\epsilon}   \int_{\mathcal{M} \setminus \mathcal{M}_{\epsilon^\gamma}} |f(x)g(x)| dV(x). \notag
\EEA
Since $\textup{Vol}(\mathcal{M} \setminus \mathcal{M}_{\epsilon^\gamma}) =\mathcal{O}(\epsilon^{\gamma})$ (see Proposition~\ref{dependence on r}), it follows that 
\BEA
\Bigg| \left\langle \int_{\mathcal{M} \setminus \mathcal{M}_{\epsilon^\gamma}} \epsilon^{-d/2} k_\epsilon(\cdot,y) f(y) dV(y), f  \right\rangle_{L^2(\mathcal{M})} \Bigg| = \mathcal{O}(\epsilon^{3 \gamma}). \notag
\EEA
\end{proof} 

We are now ready to state the following spectral bound for data sampled according to  Definition~\ref{sampling2domain}, which proof uses the four lemmas above.

\begin{lemma}\label{spectralconvdiri-supplement} Let $\mathcal{M}$ be a manifold with boundary, and let $\lambda_i$ denote the $i$-th eigenvalue of the Dirichlet Laplacian. Suppose that $n_0$ and $n_1$ points are uniformly sampled based on Definition~\ref{sampling2domain}. Let $0 < \gamma < 1/2$ be fixed. For any $\epsilon>0$ and $i<n_1$, with probability greater than $1 - \frac{4n_1}{(n_0 + n_1)^3} - \frac{4}{(n_0 + n_1)^2}$, we have that 
 \begin{equation*}
  |\lambda_{i} - \tilde{\lambda}^{\gamma,\epsilon,n}_{i}| = \mathcal{O}\left( \epsilon^{1/2}, \epsilon^{3\gamma - 1}, \frac{ \log(n_0 + n_1)^{1/2}}{\epsilon^{d/2 + 1}\sqrt{n_1}},  \frac{\sqrt{\epsilon^\gamma} \log(n_0 + n_1)^{1/2}}{\epsilon^{d/2 + 1} \sqrt{n_0}} \right)
\end{equation*}
as $\epsilon\to 0$ after $n_1, n_0 \to \infty$. Here, $\tilde{\lambda}^{\gamma,\epsilon,n}_{i}$ denotes the eigenvalues of $\tilde{L}_{\epsilon,n}^\gamma$.
\end{lemma}

\begin{proof}
Let $f \in C_0^{\infty}(\mathcal{M})$. Notice that we have 
\BEA
     \| \nabla f  \|^2_{L^2(\mathcal{M})} -  \tilde{\lambda}^{\gamma,\epsilon,n}_{i} &=& \| \nabla f  \|^2_{L^2(\mathcal{M})}  -  \langle L_\epsilon f, f \rangle_{L^2(\mathcal{M})} +   \langle L_\epsilon f , f  \rangle_{L^2(\mathcal{M})} - \langle L^{c,\gamma}_\epsilon f, f \rangle_{L^2(\mathcal{M})} + \langle L^{c,\gamma}_\epsilon f, f \rangle_{L^2(\mathcal{M})}  -   \tilde{\lambda}^{\gamma,\epsilon,n}_{i}. \notag
\EEA
Fix any subspace $S \subseteq C_0^{\infty}(\mathcal{M})$ of dimension $i$. Maximizing both sides over all $f \in S \setminus \{0\}$ with $\|f\|_{L^2(\mathcal{M})} = 1$, the above equality still holds. Using subadditivity of the maximum, we obtain
\BEA
     \max_{f \in S, \|f\|=1} \| \nabla f  \|^2_{L^2(\mathcal{M})} -  \tilde{\lambda}^{\gamma,\epsilon,n}_{i} &\leq&  \max_{f \in S, \|f\| = 1}\left( \| \nabla f  \|^2_{L^2(\mathcal{M})} -  \langle L_\epsilon f, f \rangle_{L^2(\mathcal{M})} \right) 
    \notag \\ &+& \max_{f \in S, \|f\|=1} \left( \langle L_\epsilon f , f  \rangle_{L^2(\mathcal{M})} - \langle L^{c,\gamma}_\epsilon f, f \rangle_{L^2(\mathcal{M})} \right) \notag \\
     &+& \max_{f \in  S, \|f\| = 1} \langle L^{c,\gamma}_\epsilon f, f \rangle_{L^2(\mathcal{M})} -  \tilde{\lambda}^{\gamma,\epsilon,n}_{i}.\notag
\EEA
Since this relation is valid for all $i-$dimensional  subspaces $S$, choosing the subspace $S' \subset C^\infty_0(\mathcal{M})$ for which the term $\max_{f \in S, \|f\|=1} \langle L^{c,\gamma}_\epsilon f , f  \rangle_{L^2(\mathcal{M})}$ is at a minimum, the inequality again is also valid. Using Equation \eqref{def_lambdac}, we have
\BEA
     \lambda_i -  \tilde{\lambda}^{\gamma,\epsilon,n}_{i} &\leq&  \max_{f \in S', \|f\| = 1}\left( \| \nabla f  \|^2_{L^2(\mathcal{M})} -  \langle L_\epsilon f, f \rangle_{L^2(\mathcal{M})} \right) \notag \\ &+& \max_{f \in S', \|f\|=1} \left( \langle L_\epsilon f , f  \rangle_{L^2(\mathcal{M})} - \langle L^{c,\gamma}_\epsilon f, f \rangle_{L^2(\mathcal{M})} \right) + \lambda^{c,\gamma,\epsilon}_i -   \tilde{\lambda}^{\gamma,\epsilon,n}_{i}.\notag
\EEA
 By the weak convergence result, the first term is $\mathcal{O}\left( \epsilon^{1/2} \right)$. The second term can be simplified as 
 \BEA
 \frac{1}{\epsilon}\Bigg\langle  \int_{\mathcal{M}} \hat{k}_\epsilon (\cdot, y) f(y) dV(y) - \int_{\mathcal{M}_{\epsilon^\gamma}} \hat{k}_\epsilon^{c,\gamma} (\cdot, y) f(y) dV(y), f \Bigg\rangle_{L^2(\mathcal{M})} 
 \notag \hspace{6cm}\\  =\frac{1}{\epsilon}\Bigg\langle  \int_{\mathcal{M} \setminus \mathcal{M}_{\epsilon^\gamma}} \hat{k}_\epsilon (\cdot, y) f(y) dV(y) 
+  \int_{\mathcal{M}_{\epsilon^\gamma}} (\hat{k}_\epsilon - \hat{k}_\epsilon^{c,\gamma}) (\cdot, y) f(y) dV(y), f \Bigg\rangle_{L^2(\mathcal{M})} 
 \hspace{0cm}= \mathcal{O}\left( \epsilon^{3 \gamma -1} \right) +  \mathcal{O}\left(\epsilon^{\frac{1}{2}}\right),\notag 
\EEA
where the first bound is due to Lemma~\ref{small integral} and the second bound is due to the construction of $L^{c,\gamma}_\epsilon$ which satisfies $\|L_\epsilon^{c,\gamma} - L_\epsilon^\gamma\|_{HS} = O(\epsilon^{1/2})$.
 
For the third term, we split
 $$
\lambda^{c,\gamma,\epsilon}_i -  \tilde{\lambda}^{\gamma,\epsilon, n}_{i} = (\lambda^{c,\gamma,\epsilon}_i - \lambda^{\gamma,\epsilon}_{i})+(\lambda^{\gamma,\epsilon}_{i} -  \lambda^{\gamma,\epsilon,n}_{i}) + (\lambda^{\gamma,\epsilon,n}_{i} - \tilde{\lambda}^{\gamma,\epsilon,n}_{i}), 
 $$
 where $\lambda_i^{\gamma,\epsilon}$ is the $i$-th eigenvalue of the integral operator $L^\gamma_\epsilon$ as defined in Equation \eqref{truncated-integral-operator}. By Lemma  \ref{perturbation lemma}, 
 $
 |\lambda^{c,\gamma,\epsilon}_i - \lambda^{\gamma,\epsilon}_i| = \mathcal{O}(\epsilon^{1/2}).
 $
 From Lemma \ref{dirichlet rosasco}, with confidence greater than $1-\frac{2}{(n_0 + n_1)^2}$, the separation of eigenvalues of $L^\gamma_\epsilon$ and $L^\gamma_{\epsilon,n}$ is bounded above by $\mathcal{O} \left(\frac{\log(n_0 + n_1)}{\epsilon^{d/2 + 1}\sqrt{n_1}}\right)$. The last term is bounded using Lemma~\ref{truncated matrix eigenvalues}.  
 The upper bound on $\tilde{\lambda}^{\gamma,\epsilon,n}_i - \lambda_i$ is the same, and the proof is similar. This completes the proof. 
\end{proof}
 

With this result, we now transfer it to the uniform sampling setup.

\begin{proof}[Proof of Theorem~\ref{spectralconvdiri uniform}]
By Proposition~\ref{dependence on r}, the volume of $\mathcal{M} \setminus \mathcal{M}_{\epsilon^\gamma}$ scales linearly with $\epsilon^\gamma$, $\textup{Vol}(\mathcal{M} \setminus \mathcal{M}_r) \sim c \epsilon^\gamma$. Since $n_0$ is a binomial random variable with mean $c\cdot \epsilon^\gamma \cdot n$, by Hoeffding, with probability higher than $1 - 2e^{-t^2/n}$, 
$
|n_0 - cn \epsilon^\gamma| < t. 
$
Choosing $t = \sqrt{n \beta}$, we have that with probability higher than $1 - 2e^{-\beta}$, 
$
|n_0 - cn \epsilon^\gamma | < \sqrt{n \beta}.
$
In particular, 
$$
cn \epsilon^\gamma -\sqrt{n \beta} \leq n_0 \leq \sqrt{n \beta} + cn \epsilon^\gamma, \,\, \textup{ and } \,\, (1-c \epsilon^\gamma )n - \sqrt{n \beta} \leq n_1 \leq \sqrt{ n \beta} + (1-c \epsilon^\gamma )n. 
$$ 
Since $n_0$ points are sampled uniformly from $\mathcal{M} \setminus \mathcal{M}_{\epsilon^\gamma}$, and $n_1$ points are sampled uniformly from $\mathcal{M}_{\epsilon^\gamma}$, we can apply {\color{black}Lemma~\ref{spectralconvdiri-supplement}}, replacing $n_1$ and $n_0$ within the accuracy bounds with lower bounds for $n_0$ and $n_1$, and replacing $n_1$ and $n_0$ in the probability bounds with the upper bounds for $n_0$ and $n_1$. That is, with total probability higher than $1 - 2e^{-\beta} - \frac{4(\sqrt{n \beta} + (1-c \epsilon^\gamma )n)}{n^3} - \frac{4}{n^2}$,
 \begin{equation*}
  |\lambda_{i} - \tilde{\lambda}^{\gamma,\epsilon,n}_{i}| = \mathcal{O}\left( \epsilon^{1/2}, \epsilon^{3 \gamma - 1},  \frac{\log(n)^{1/2}}{\epsilon^{d/2 + 1}\sqrt{(1-c\epsilon^\gamma)n - \sqrt{n \beta}}},  \frac{\sqrt{\epsilon^\gamma}\log(n)^{1/2}}{\epsilon^{d/2 + 1} \sqrt{cn \epsilon^\gamma - \sqrt{n \beta}}}  \right).
\end{equation*}
Choosing $\beta = \log(n)^2$, it is clear from the denominator in the final term that we must require
$
\epsilon^\gamma \gg \frac{ \log(n)}{c \sqrt{n}}. 
$
The denominator of the final term simplifies to 
$$
\sqrt{cn \epsilon^\gamma - \sqrt{n \beta}} = \sqrt{cn \epsilon^\gamma - n^{1/2}  \log(n) }. 
$$
Since  $\epsilon^\gamma \gg \frac{\log(n)}{c n^{1/2}}$, this term behaves as $\sqrt{cn \epsilon^\gamma}$. Hence, the final term in the rate can be simplified to 
$$
\mathcal{O}\left(\frac{ \sqrt{\epsilon^\gamma} \log(n)^{1/2}}{\epsilon^{d/2+1} \sqrt{n \epsilon^\gamma}}\right) = \mathcal{O}\left(\frac{\log(n)^{1/2}}{\epsilon^{d/2+1} \sqrt{n}}\right). 
$$
Similarly, the denominator in the third term simplifies to 
$$
\epsilon^{d/2+1} \sqrt{n - cn \epsilon^\gamma - \sqrt{n}\log(n) } \gg \epsilon^{d/2+1} \sqrt{ n(1- 2 c\epsilon^\gamma) }.
$$
For $\epsilon^\gamma \ll 1$, this term behaves as $\epsilon^{d/2+1} \sqrt{n}$. Hence, with probability higher than $1 - \frac{6}{n^2} - \frac{4(2 \sqrt{n} \log(n) + 2n)}{n^3}$, we have \eqref{spectralconvdirirate} and the proof is completed.
\end{proof}
\begin{remark}
\label{valid parameter regimes} 
From the rate in \eqref{spectralconvdirirate}, the dominant error is $\epsilon^{3\gamma-1}$. This follows since $\gamma<1/2$ implies that $\epsilon^{1/2} < \epsilon^{3\gamma-1}$. This in turn implies that the method converges when $3\gamma-1>0$ or $\gamma>1/3$, which gives us a guaranteed convergence when
\[
n^{-1/2}\log(n) \ll \epsilon^{1/2} < \epsilon^\gamma < \epsilon^{1/3},
\] 
where the condition that $\epsilon^{1/2} \gg \frac{\log(n)}{\sqrt{n}}$ comes from the hypothesis of  Theorem~\ref{spectralconvdiri uniform}.
If we choose $\gamma \to 1/2$, then we have balanced the first two error rates and obtain an effective rate of $\epsilon^{1/2}$.
The third term gives a lower bound for how fast $\epsilon \to 0$, while the first term requires $\epsilon \ll 1$. In particular, 
$$
\left(n^{-1}\log(n)\right)^{\frac{1}{d+2}} \ll \epsilon \ll 1.
$$
Balancing the first and the third terms, we obtain
\[
\epsilon (n) \propto  \left(n^{-1}\log(n)\right)^{\frac{1}{d+3}}  > \left(n^{-1}\log(n)\right)^{\frac{1}{d+2}} \gg \left(n^{-1}\log(n)\right)^{\frac{2}{d}},
\]
which is larger than the connectivity threshold (as we noted in Remark~\ref{connectivity threshold}).
Together with the second term, if we take $\gamma$ arbitrarily close to $1/2$, we achieve the theoretical convergence rate,
\[
 |\lambda_{i} - \tilde{\lambda}^{\gamma,\epsilon,n}_{i}| = \mathcal{O}\left( \left(n^{-1} \log(n)\right)^{\frac{1}{2d + 6}} \right).
\]
\end{remark}

\begin{remark}
\label{truncation remark} The above result formalizes the observation made by others \cite{tg:2019},\cite{hs:2022} that truncating the matrix obtained from Diffusion Maps algorithm to the interior yields spectral convergence of the Dirichlet Laplacian. We numerically verify this further for a few examples in Section \ref{sec6}, where we will set $\gamma = 1/2$. One can view truncating the graph Laplacian near the boundary as forcing $0$ boundary conditions. With this study, it is clear that the truncated graph Laplacian with $\gamma=1/2$ converges to the Dirichlet Laplacian. 
\end{remark}

\section{Numerical Results for Dirichlet Boundary Conditions}\label{sec6} \noindent
Since numerical results for closed manifolds and Neumann eigenvalue problems are well documented \cite{harlim2018data,jiang2023ghost}, we only report results for the Dirichlet eigenvalue problems.
Particularly, we numerically verify the algorithm suggested by Theorem \ref{spectralconvdiri uniform}. Namely, we show for two fundamental examples of manifolds with boundary that the Diffusion Maps algorithm, when modified by restricting the discretization matrix to points on the interior whose distance from the boundary satisfies the assumption in Theorem \ref{spectralconvdiri uniform}, converges to the Dirichlet Laplacian. For convenience, we refer to this method for approximating the Dirichlet Laplacian as the Truncated Graph Laplacian (TGL). We note that the validity of TGL, demonstrated below, emphasizes the growing importance of accurately estimating the distance of a data point to the boundary, since in the traditional manifold learning set up, the structure of the manifold is unknown. Numerical methods for such estimations are formally discussed in \cite{vaughn2024diffusion}. However, in the examples tested below, distance from the boundary is known a priori, and therefore such methods are not employed in this paper. In practice, we treat the points in $ X \cap \mathcal{M}_{\epsilon^\gamma}$ as those data points whose distance from the boundary is greater than $\epsilon^\gamma$. Based on Remark~\ref{valid parameter regimes}, we perform the truncation by setting $\gamma = 1/2$ unless stated. {\color{black}In the following subsection, we will also show how the estimate varies as a function of $\gamma$.} 
 
 \subsection{Dirichlet Laplacian on a Semi-Circle} \noindent 
In this example, let $\mathcal{M}$ denote the $1$-dimensional semicircle in $\mathbb{R}^2$, with the embedding  $\{ (\cos(\theta), \sin(\theta)): 0 \leq \theta \leq \pi\}.$  
We consider the Dirichlet eigenvalue problem. 
Recall that the Riemannian metric on a semi-circle with embedding $\{ (\cos(\theta), \sin(\theta)): 0 \leq \theta \leq \pi\}$ is simply $g \equiv 1$. It is then easily checked that the above eigenvalue problem has exact solutions $
\lambda_k = k^2, \textrm{ and } f_k(\theta) = \sin(k\theta) \textrm{ for }k=1,2,3,\dots.$ 

The eigenvalues and eigenfunctions of Dirichlet Laplacian for this example are estimated using the eigenvalues and eigenfunctions of a $K$-nearest neighbors version of the matrix
$\tilde{L}_{\epsilon,n}^{\color{black}\gamma}$ with a Gaussian kernel, using $K = \sqrt{n}$. Note that for the Gaussian kernel, we have the following explicit formula for the normalization $m^\partial_0$:
$$
m_0^\partial (x) = \frac{\pi^\frac{d}{2}}{2}\left( 1 + \textrm{erf}(b_x/\epsilon) \right)
$$
where $b_x$ denotes distance to the boundary (as shown in \cite{vaughn2024diffusion}). This normalization is especially important when the distribution $q$ is nonuniform, and is involved in the discretization in Equation \eqref{mat_L_qen}. In what follows, however, we numerically observe that this normalization outperforms the constant normalization even for well-sampled, uniformly distributed data. For well-sampled data, the matrix is constructed from a grid of data points $\left\{(\cos\left( \frac{\pi i}{n}\right), \sin\left(\frac{\pi i}{n} \right)) \right\}_{i=1}^n$, and the parameter $\epsilon$ was automatically tuned using the method outlined in Chapter $6$ of \cite{harlim2018data}. For randomly generated data, the parameter $\epsilon$ was hand tuned for higher values of $n$ to minimize the error of the eigenvalues. For lower values of $n$, the $\epsilon$ was specified by a linear fitting of $\log(\epsilon)$ on $\log(n)$ on the hand-tuned $\epsilon$ of high values of $n$. The numerical results for fixed $n = 10,000$ are shown in Figure \ref{fig1}. 

In Figure \ref{fig1}(a), we compare the absolute error of the first $10$ eigenvalues using the constant normalization $m_0$ (red curve with circle), which we refer to as no normalization in the discussion below, and using the normalization $m^\partial_0$ (blue curve with circle) depending on distance from the boundary. In Figure \ref{fig1}(b) we show the corresponding root-mean-square errors of the eigenfunctions for these cases, which is precisely the $L^2(\mu_n)$ norm of the difference $\tilde{f}_k - R_nf_k$, where $f_k$ denotes the analytic eigenfunction, and $\tilde{f}_k$ denotes the approximate eigenvector of $\widetilde{L}_{\epsilon,n}$.  These figures demonstrate that for well-sampled data, the $m_0^\partial$ normalization slightly outperforms that without normalization. In the same panels, we also show estimates of 10 uniformly sampled random realizations (black curves) with the $m_0^\partial$ normalization. These results suggest that using this normalization the spectrum can be accurately approximated for low eigenvalues. For random data, the accuracy is between $10^{-2}$ and $10^{-1}$.
\begin{figure}[htbp]
{\scriptsize \centering
\begin{tabular}{cc}
\normalsize (a)  & \normalsize (b) \\
\includegraphics[width=.45\textwidth]{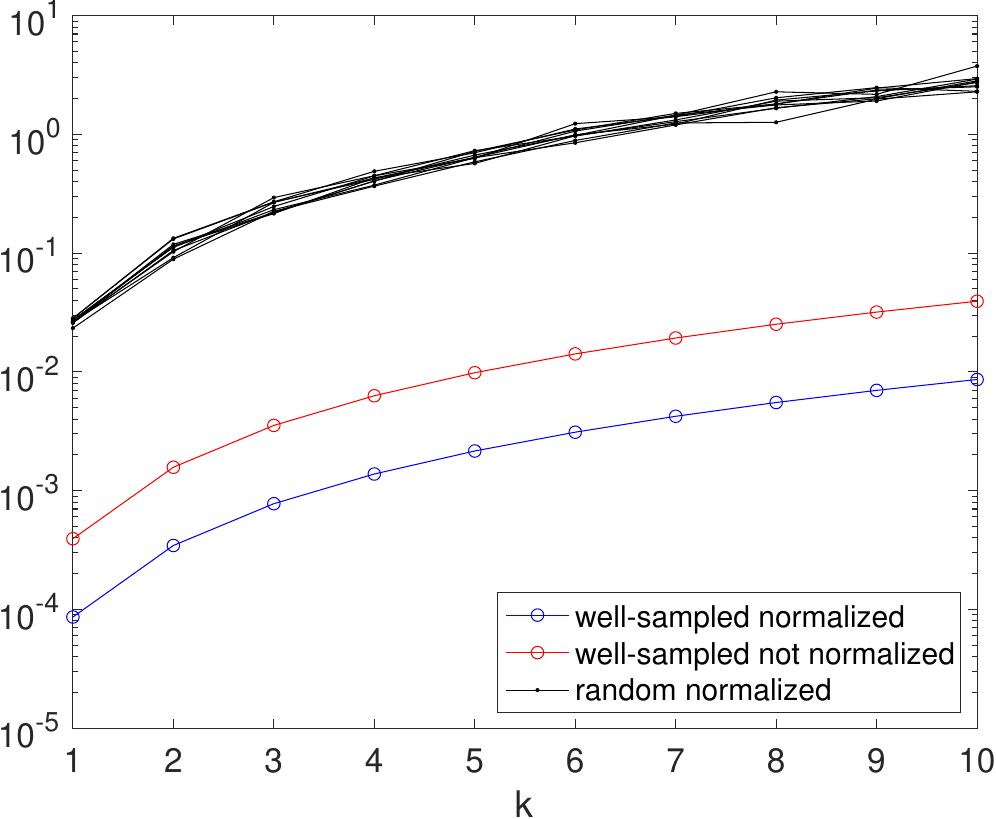} &
\includegraphics[width=.45\textwidth]{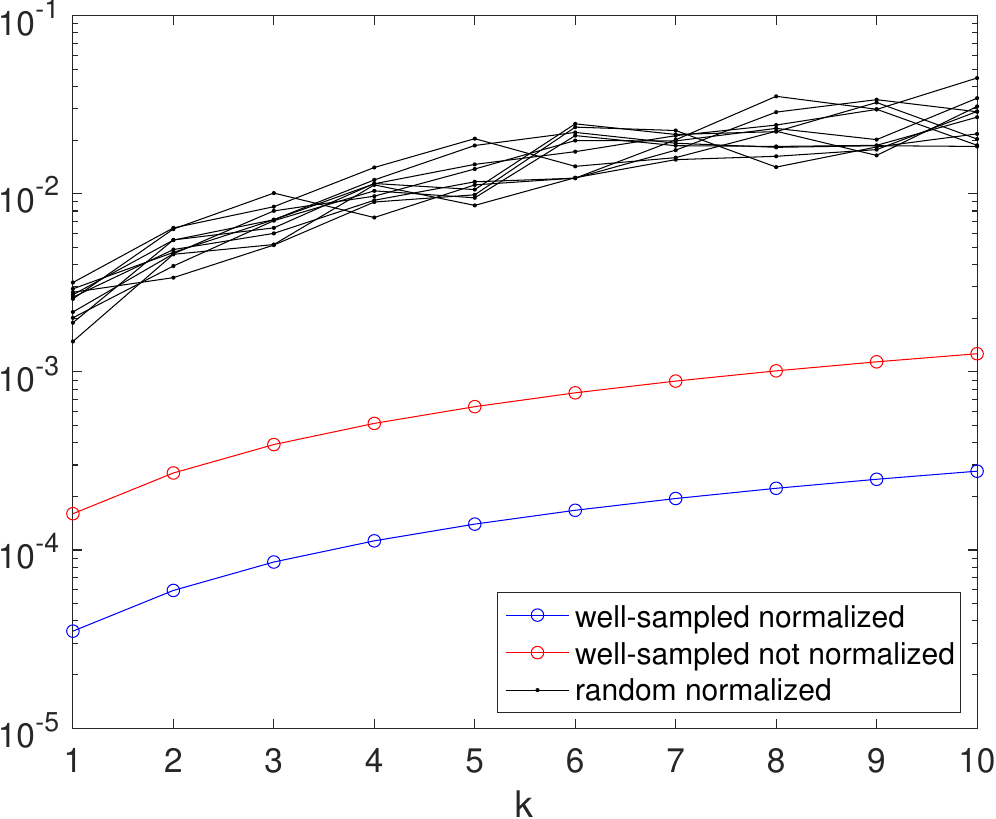}\\
\normalsize (c)  & \normalsize (d) \\
\includegraphics[width=.45\textwidth]{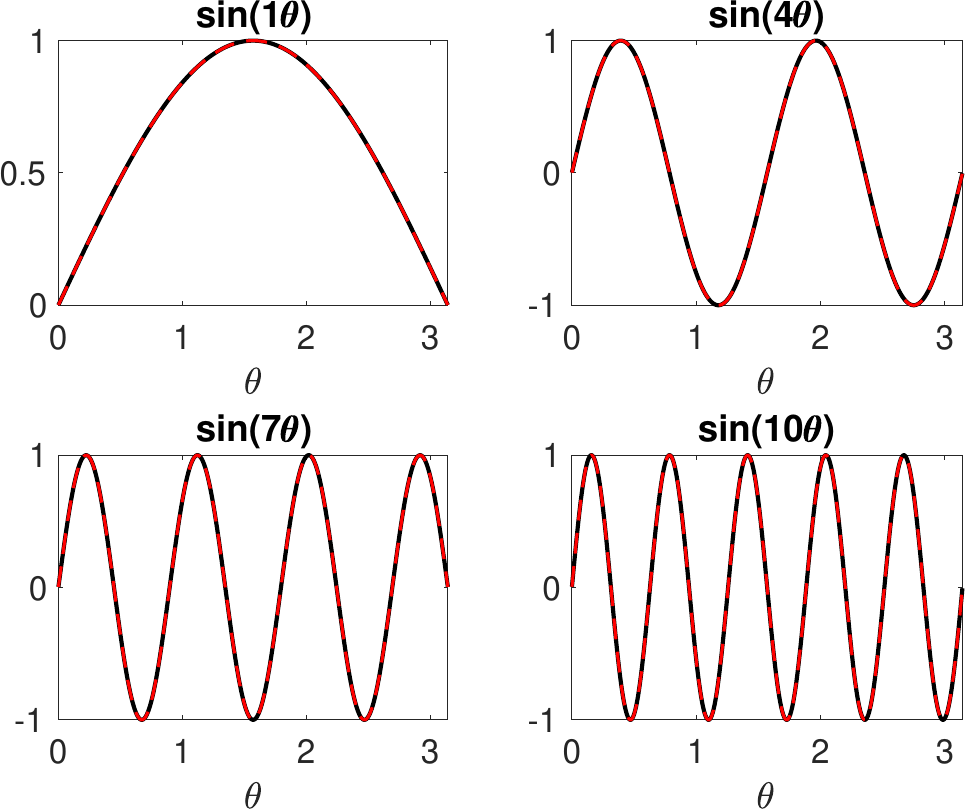} &
\includegraphics[width=.45\textwidth]{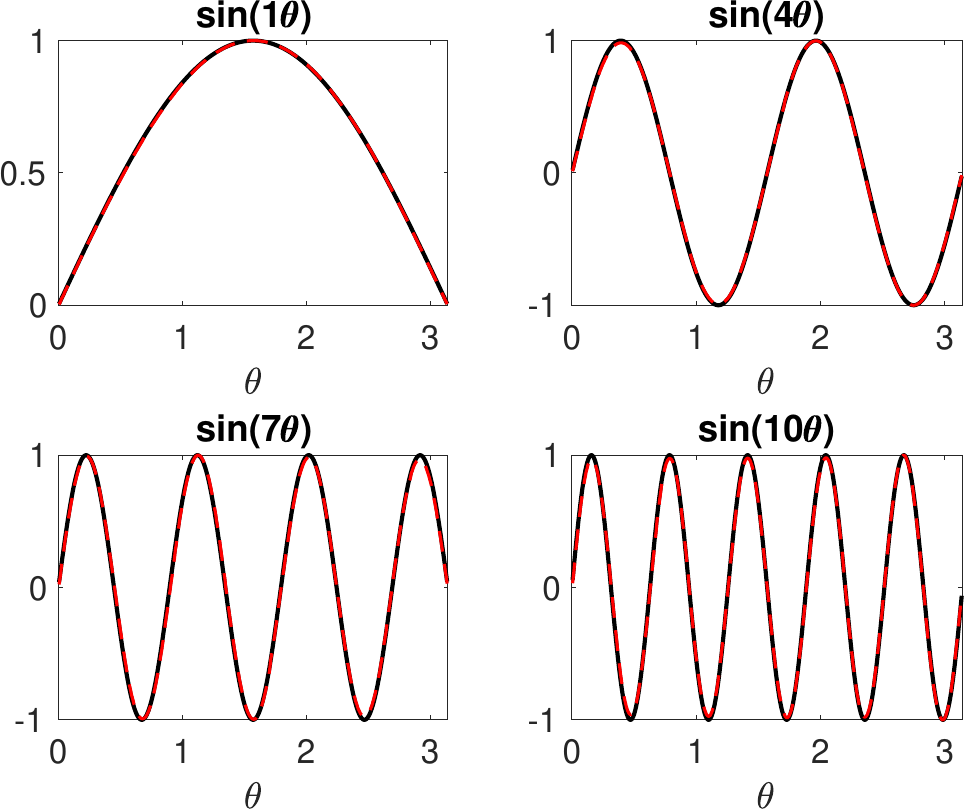}
\end{tabular} 
} 
\caption{Semi-Circle Example with $n=10000$: (a) Absolute errors of eigenvalues as functions of mode $k$; (b) Root-mean-square errors of eigenvectors as functions of mode $k$; 
(c) Estimated eigenvectors from well-sampled data, normalized with $m_o^\partial$; (d) Estimated eigenvectors from random data, normalized with $m_o^\partial$. In panels (a) and (b), we show estimates from 10 random realizations (black curves).}
\label{fig1}
\end{figure}
Based on this result, in the remainder of this paper, all simulations use the normalization $m_o^\partial$. Figure \ref{fig1}(c) demonstrates the accuracy of TGL in approximating eigenfunctions for uniformly distributed data which is well-sampled, while Figure \ref{fig1}(d) shows the same estimates with random data. We see that even for random data, the approximate eigenfunctions are indistinguishable from the analytic eigenfunctions. 
 
Figure \ref{fig2} demonstrates convergence as a function of $n$ for both well-sampled uniform data (\ref{fig2}(a)) and random uniformly distributed data (\ref{fig2}(b)). To be consistent with the error metrics in Theorems~\ref{spectralconvdiri uniform} and \ref{conveigvec-dirichlet-supplement}, we numerically compute,
\BEA
\begin{aligned}
\textup{Error of Eigenvalues}&:= \frac{1}{M} \sum_{i=1}^M \frac{|\lambda_i - \tilde{\lambda}^{\gamma,\epsilon,n}_{i}|}{|\lambda_i|},  \\
\textup{Error of Eigenvectors}&:= \frac{1}{M} \sum_{i=1}^M \|R_nf_i -f_i^{\gamma,\epsilon,n} \|^2_{L^2(\mu_n)}, 
\end{aligned}\label{empiricalerror}
\EEA
where $\{\lambda_i, f_i\}_{i=1}^{M}$ and {\color{black}$\{\tilde{\lambda}_i^{\gamma,\epsilon,n}, f_i^{\gamma,\epsilon,n}\}_{i=1}^{M}$} denote the leading $M$ eigensolutions of the Laplace-Beltrami operator $\Delta$ and the matrix ${\color{black}\tilde{L}^{\gamma}_{\epsilon,n}}$, respectively. {\color{black}Again, $R_nf_i = (f_i(x_1),\ldots,f_i(x_n))$ denotes the restriction of $f_i$ on the training data set $X=\{x_1,\ldots, x_n\}$.} For non-uniform data {\color{black}with sampling density $q$}, we compute the errors for the leading $M$ eigensolutions {\color{black}$\{\tilde{\lambda}_i^{q,\gamma,\epsilon,n}, f_i^{q,\gamma,\epsilon,n}\}_{i=1}^{M}$ of the matrix $\tilde{L}^{\gamma}_{q,\epsilon,n}$, which is the $n_1\times n_1$ submatrix of $\tilde{L}_{q,\epsilon,n}$ defined in \eqref{mat_L_qen}. Here, the components of the matrix $\tilde{L}^{\gamma}_{q,\epsilon,n}$ are the components of $\tilde{L}_{q,\epsilon,n}$ associated to data points that are $\epsilon^\gamma$ away from the boundary. Its construction is analogous to how the matrix $\tilde{L}_{\epsilon,n}^\gamma$ in \eqref{Legamma} is defined as a submatrix of  $\tilde{L}_{\epsilon,n}$.}  

In this numerical experiment, we will verify errors correspond to the leading $M=10$ estimated eigensolutions. For random uniformly distributed data, $10$ trials were performed for each value of $n$. In both cases, we still observe convergence faster than {\color{black}the predicted rate of $n^{-\frac{1}{8}}$ (resp. $n^{-\frac{1}{4}}$) for the eigenvalues (resp. eigenvectors, see Remark~\ref{remonDir})} in this one-dimensional example. {\color{black}The fact that the $L^2$ square error rate for the eigenvectors is two times faster than the error in the eigenvalues validates theoretical findings noted in Remark~\ref{remonDir}. What is surprising is that such a consistency in the error rates (between the $L^2$ error in the eigenvectors and the error in the eigenvalues) is valid even when the $L^2$ metric used in quantifying the eigenvector error in \eqref{empiricalerror} is computed over the entire data set. In contrast, the $L^2$ metric used in theory is only defined over the data points that are $\epsilon^\gamma$ away from the boundary (see Theorem~\ref{conveigvec-dirichlet-supplement}).} {\color{black}It is worth noting that one can verify these slower predicted rates if one sets the kernel bandwidth parameter $\epsilon \sim n^{-1/4}$ (results are not shown).} Figure \ref{fig3} shows the errors in Equation \eqref{empiricalerror} for non-uniformly distributed data. The non-uniformly sampled data used to generate Figure \ref{fig3}(a) (resp. \ref{fig3}(b)) was distributed in accordance to $\pi  \cos(x)$, where $x$ are well-sampled (resp. randomly sampled) data uniformly in the interval $(0, \pi/2)$. With this distribution, there are comparably more data points sampled near the boundary. Again, for random data, $10$ trials were performed for each value of $n$. As expected, the estimations from well-sampled data is more accurate compared to those from randomly sampled data. Nevertheless, in either case we observe convergence as $n \to \infty$ at a rate faster than the rate predicted by the proofs of the main theorems for non-uniformly sampled data.
\begin{figure}[htbp]
{\scriptsize \centering
\begin{tabular}{cc}
\normalsize (a)  & \normalsize (b) \\
\includegraphics[width=.45\textwidth]{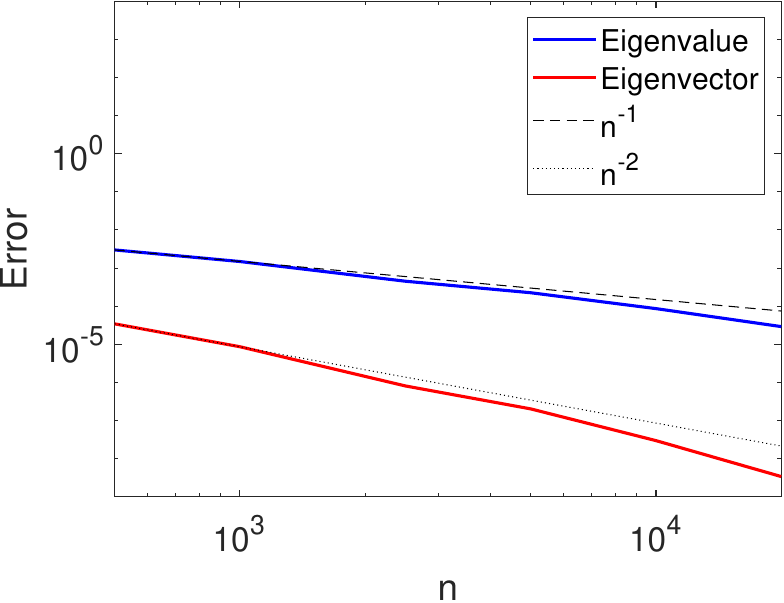} &
\includegraphics[width=.45\textwidth]{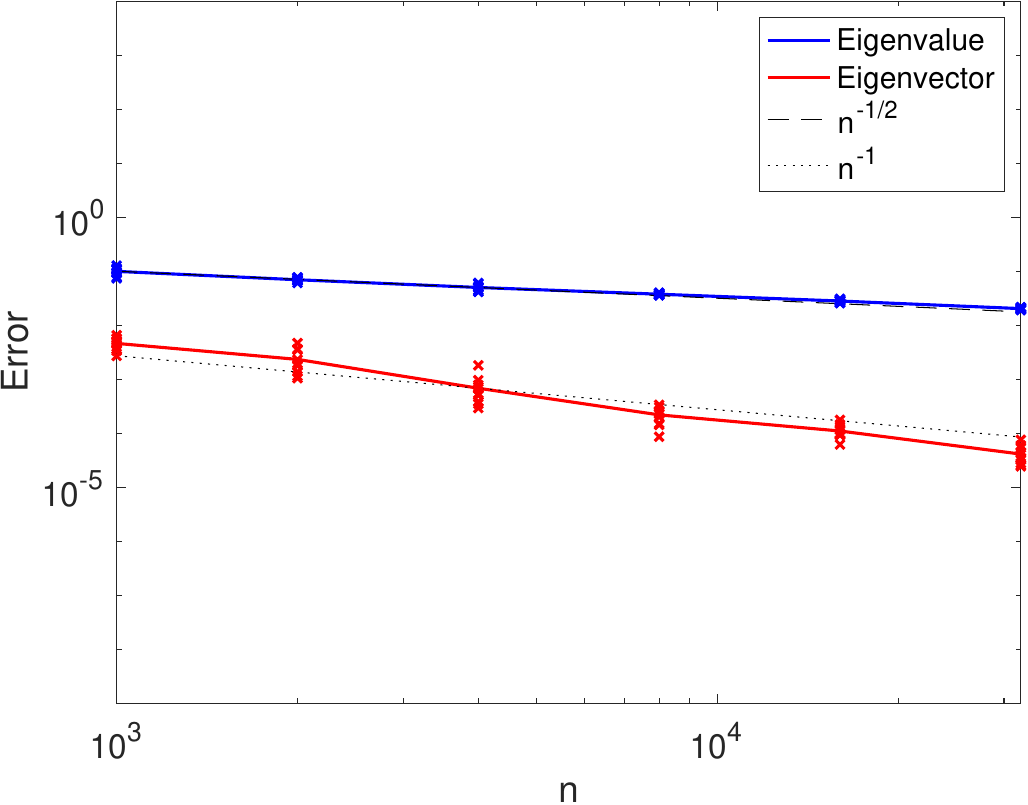}
\end{tabular} }
\caption{Semi-Circle Example for estimation with $m_o^\partial$ normalization, uniform sampling distribution. Mean of relative error (eigenvalues) and mean of mean-square-error (eigenvectors) as defined in Equation \eqref{empiricalerror} over the first $M=10$ modes as functions of $n$: (a) Well-sampled data; (b) Random data.}
\label{fig2}
\end{figure}
\begin{figure}[htbp]
{\scriptsize \centering
\begin{tabular}{cc}
\normalsize (a)  & \normalsize (b) \\
\includegraphics[width=.45\textwidth]{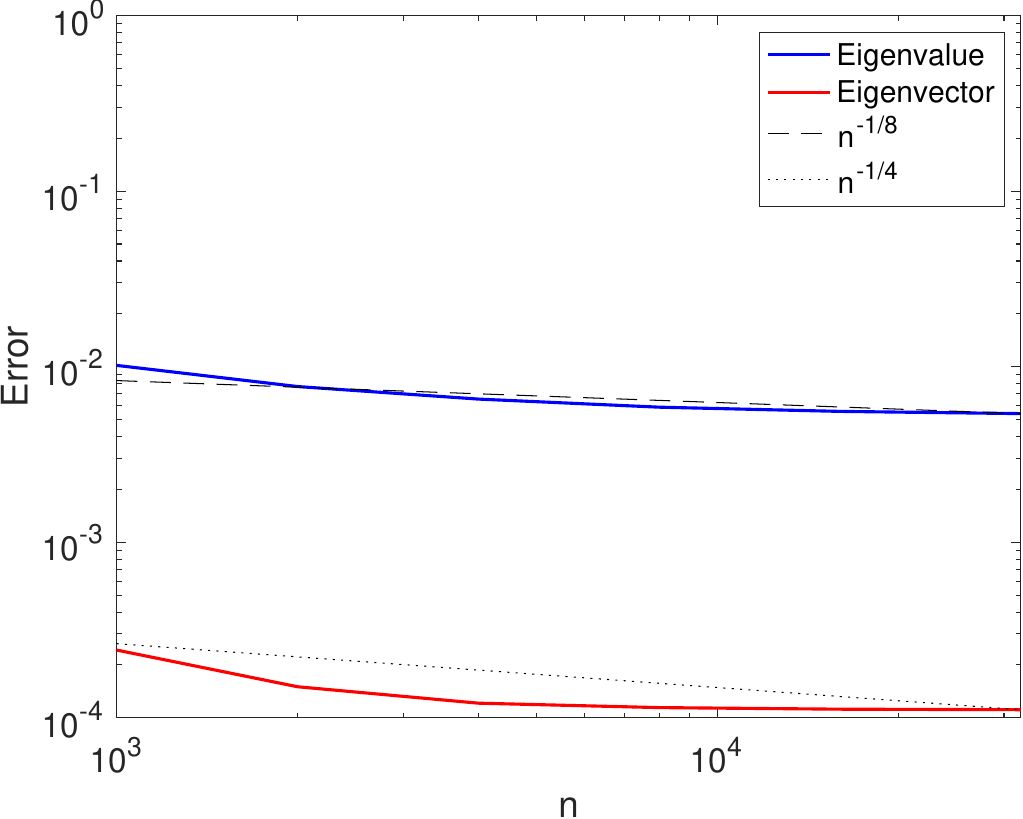} &
\includegraphics[width=.45\textwidth]{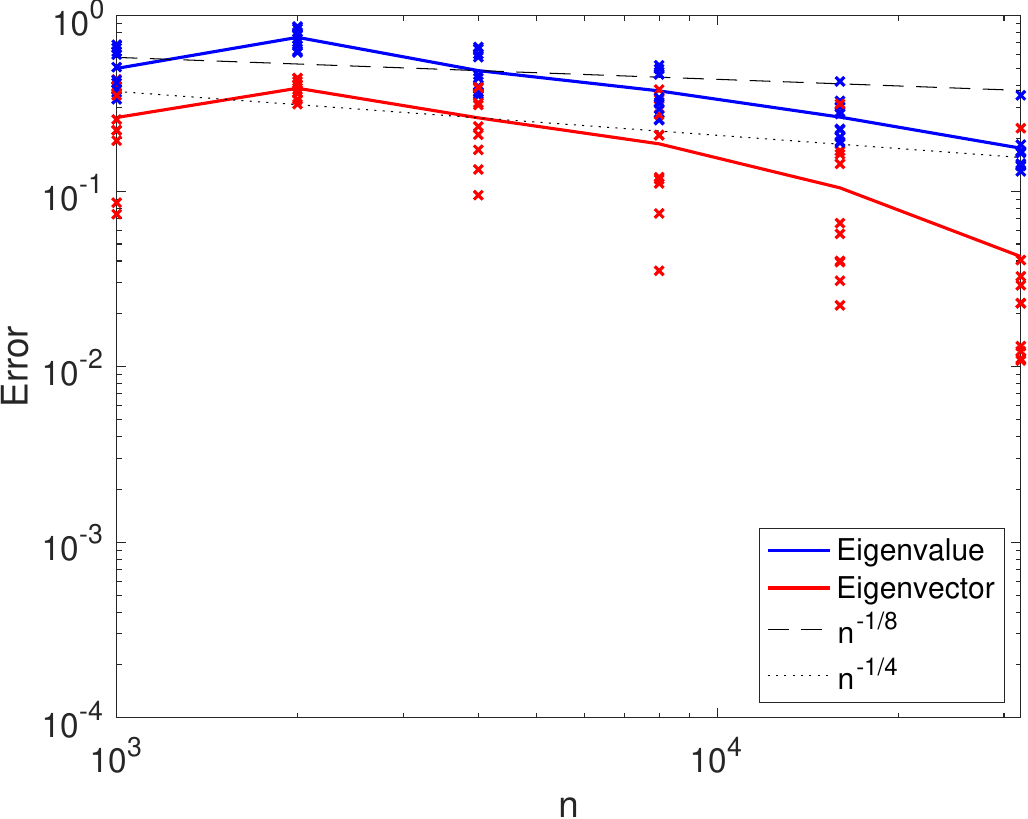}
\end{tabular}}
\caption{Semi-Circle Example for estimation with $m_o^\partial$ normalization, nonuniform sampling distribution. Mean of relative error (eigenvalues) and mean of mean-square-error (eigenvectors) as defined in Equation \eqref{empiricalerror}  over the first $M=10$ modes as functions of $n$: (a) Well-sampled data; (b) Random data. } 
\label{fig3}
\end{figure}

{\color{black}
In Figure~\ref{fignew4}, we show the leading spectrum, $\tilde{\lambda}_1^{\gamma,\epsilon,n}$, of $\tilde{L}_{\epsilon,n}^\gamma$ as a function of $\gamma$. In this experiment, we set $\epsilon=4\times 10^{-5}$. Here, we show the estimate of each of the 10 trials we performed (in red '$\times$') and their average (in blue curve) of random uniformly distributed data for a fixed $n=10,000$. For $\gamma \leq 1/2$, the estimated leading eigenvalues $\tilde{\lambda}_1^{\gamma,\epsilon,n}$ are close to the leading eigenvalue of the Dirichlet Laplacian, $\lambda_1 = 1$, validating the theoretical regime noted in Remark~\ref{valid parameter regimes}. For $\gamma>1/2$, as the region of truncation becomes smaller, we see larger errors in estimating the Dirichlet leading spectrum (some trials produce a leading eigenvalue that is closer to 0), and eventually, for $\gamma> 1$, the leading eigenvalue $\tilde{\lambda}_1^{\gamma,\epsilon,n}$ estimates the Neumann eigenvalue, $\lambda_1 = 0$.

\begin{figure}
 \begin{center}
\includegraphics[width=.45\textwidth]{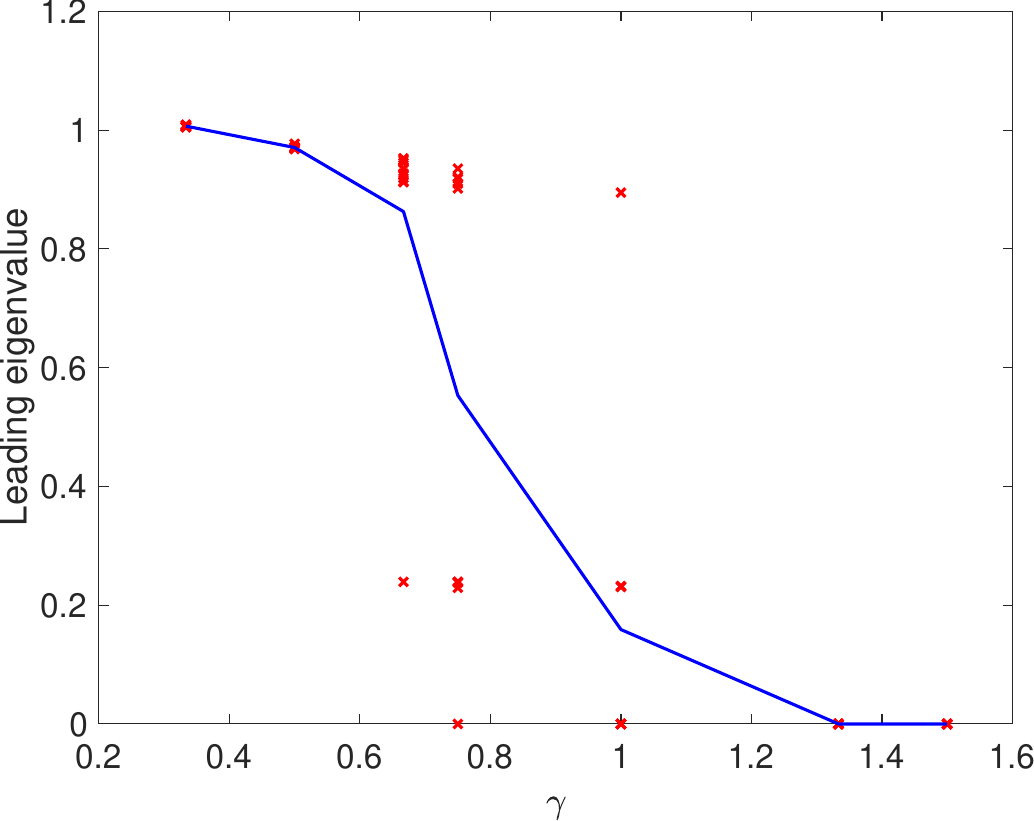}
\end{center}
\caption{Semicircle example: Estimated leading eigenvalue as a function of $\gamma$. There are 10 trials for each $\gamma$, the estimate $\tilde{\lambda}_1^{\gamma,\epsilon,n}$ for each trial is denoted in red `x'. The blue curve denotes the average of these estimates.}
\label{fignew4}
 \end{figure}

}

\subsection{Dirichlet Laplacian on a Semi-Torus} \noindent
In this example, let $\mathcal{M}$ denote the two-dimensional semi-torus embedded in $\mathbb{R}^3$, with standard embedding $$\left\{\left((2 + \cos \theta)\cos \phi,(2+ \cos \theta)\sin \phi, \sin \theta\right) :   0 \leq \phi \leq \pi, 0 \leq \theta \leq 2 \pi \right\}.$$ 
We consider the eigenvalue problem 
$$
\Delta f_k = \lambda_k f_k, \qquad f_k|_{\partial \mathcal{M}} = 0,
$$
where $\Delta$ is the Laplace-Beltrami operator on $\mathcal{M}$. It is well known that the Riemannian metric on $\mathcal{M}$ in coordinates $(\theta , \phi)$ is given by 
$$
g_{\theta, \phi}(u,v) = u^T \begin{bmatrix} 1 & 0 \\
0 & \sin^2 \theta
\end{bmatrix}v.
$$
It can then be checked that the Laplace-Beltrami operator on coordinates $(\theta , \phi)$ takes the form $$
\Delta f_k = -\frac{1}{(2 + \cos \theta)^2} \frac{ \partial^2 f_k}{\partial \phi^2} - \frac{\partial^2 f_k}{\partial \theta^2} + \frac{\sin \theta}{2 + \cos \theta} \frac{\partial f_k}{\partial \theta }. 
$$
The eigenvalue problem can then be semi-analytically solved using a separation of variables method. Assuming the solution $f_k$ is of the form $f_k = \Theta_k(\theta) \Phi_k(\phi)$, we substitute back into the above and obtain two equations: 
$$
\Phi''_k = m_k^2 \Phi, \qquad \Theta''_k - \frac{\sin \theta}{2 + \cos \theta} \Theta'_k - \frac{m^2_k}{(2 + \cos \theta)^2} \Theta_k = \lambda_k \Theta_k.
$$
Here, the value of $m_k$ will be specified such that $\{\Phi_k\}$ satisfy the boundary condition $\Phi_k (0) = 0 = \Phi_k(\pi)$. Particularly, $m_k=k$ and $\Phi_k(\phi) = \sin{k\phi}$. The eigenvalue problem corresponding to the second equation above, subjected to the periodic boundary condition, will be solved numerically with a finite-difference scheme on equal-spaced grid points. We treat eigenvalues and eigenvectors obtained in this semi-analytic fashion as the true solution. This is the same example used in \cite{jiang2023ghost} to verify other means (with the Ghost Points Diffusion Maps) for solving the Dirichlet eigenvalue problem. 

Well-sampled data were obtained by generating $\sqrt{n}$ data points $\theta_i$ uniformly spaced in $[0, 2\pi]$, and $\sqrt{n}$ data points $\phi_i$ uniformly spaced in $[0, \pi]$. A grid was then constructed, resulting in a total of $n$ data points. For random data, a similar process was employed, but with the data randomly distributed in the corresponding intervals according to a uniform distribution. The spectrum was then approximated using TGL, normalized with $m_0^\partial$, as in the previous example. Since the semi-analytical solutions to the above eigenvalue problem used for comparison are represented by vectors, whose components correspond to eigenfunction values at equally spaced grid points (uniformly spaced $64\times 64$ points on intrinsic coordinates),
\begin{figure}[htbp]
{\scriptsize \centering
\begin{tabular}{cc}
\normalsize (a)  & \normalsize (b) \\
\includegraphics[width=.45\textwidth]{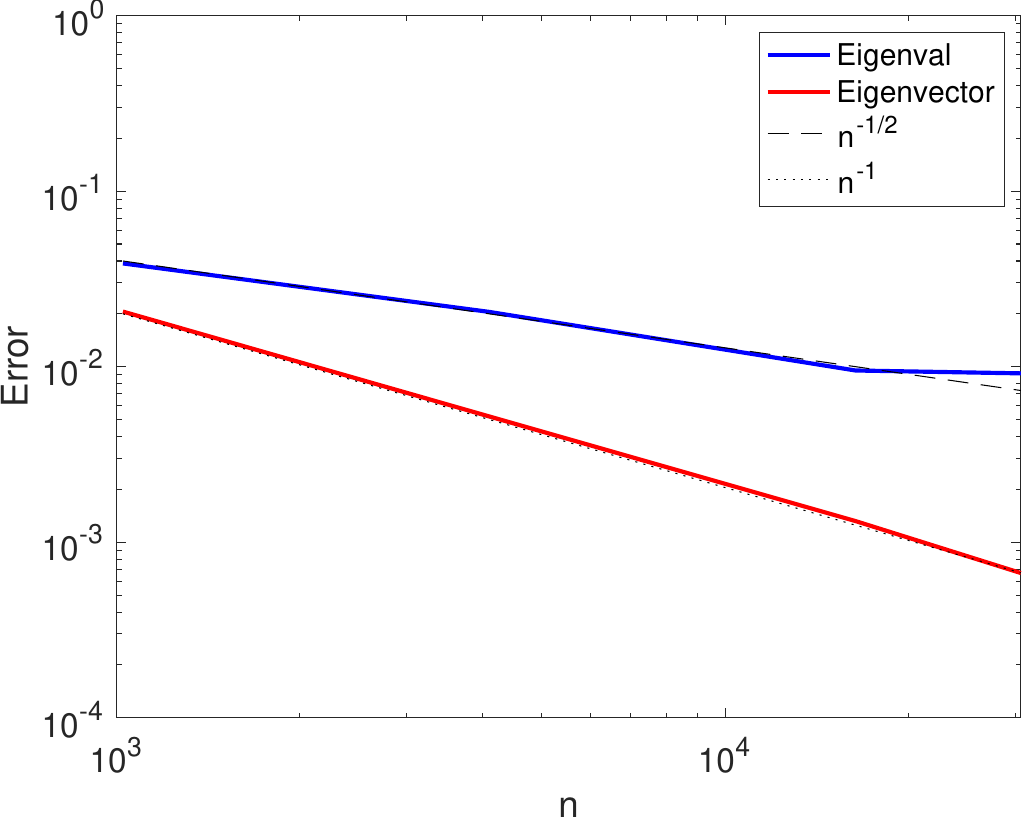} &
\includegraphics[width=.45\textwidth]{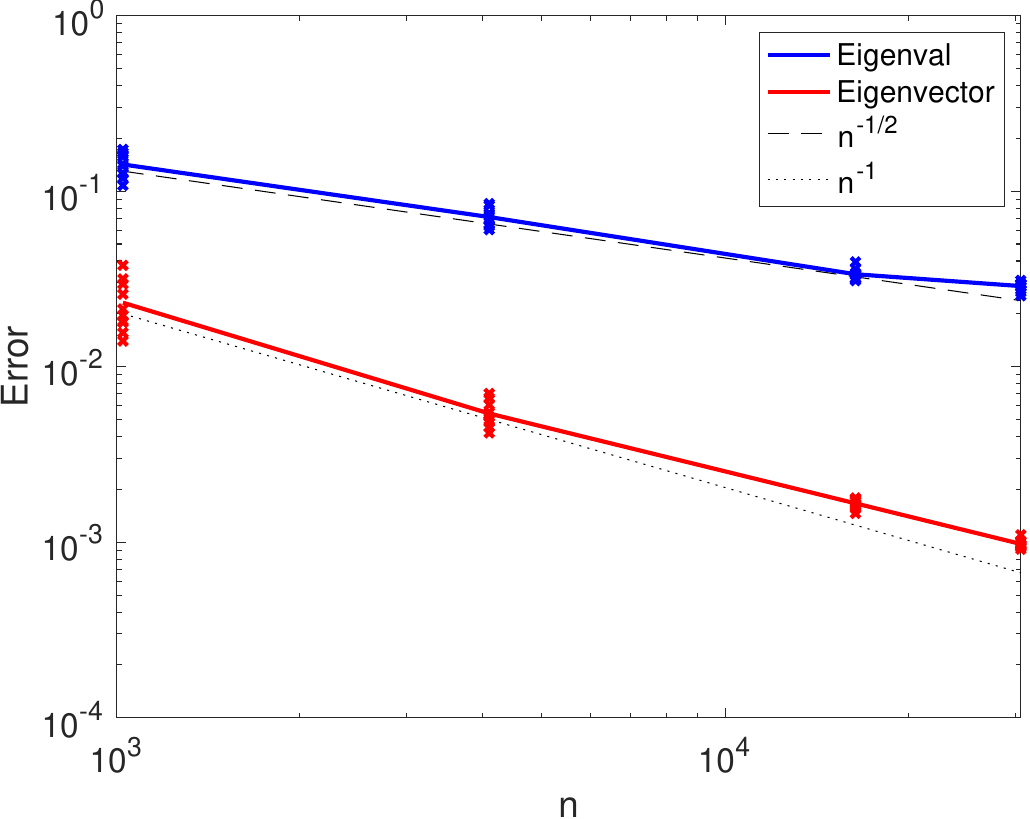}
\end{tabular}}
\caption{Semi-Torus Example for estimation with $m^\partial_0$ normalization, uniform sampling distribution. Mean of relative error (eigenvalues) and mean of mean-square-error (eigenvectors) as defined in Equation \eqref{empiricalerror}  as functions of $n$: (a) Over the first $M=10$ modes, well-sampled data; (b) Over the first $M=3$ modes, random data.}
\label{fig4}
\end{figure}
further considerations were needed to quantify the accuracy of eigenvectors of TGL since they are not necessarily discretized on the same grid points. Particularly, for random (and any well-sampled of different sizes of) data, the Nystr\"om extension method was used to interpolate the TGL eigenvectors to the grid on which the semi-analytic solutions are constructed. Eigenvector error was then evaluated by an average of the mean-square error in Equation \eqref{empiricalerror} on this grid. In both cases, $\epsilon$ was carefully hand-tuned. For well-sampled data, we set $\epsilon = \textrm{.0103,.003, .0008,.0005 }$ for $n = 32^2, 64^2, 128^2, 175^2$, respectively. For random data, we set $\epsilon = \textrm{.024,.009, .0035,.002}$ for $n = 32^2, 64^2, 128^2, 175^2$, respectively. Note that for random data, $10$ trials were performed for each value of $n$. 
In Figure \ref{fig4}, for both cases of well-sampled and random data, the errors in the estimation of eigenvalues and eigenvectors converges at a rate of $\mathcal{O}\left( n^{-1/2} \right)$ and $\mathcal{O}\left( n^{-1} \right)$, respectively, which are faster than the theoretical estimates. Here, we should point out that the errors for well-sampled data shown in Figure~\ref{fig4}(a) are averaged over the leading $M=10$ modes, whereas the errors for random data shown in Figure~\ref{fig4}(b) are averaged over the leading $M=3$ modes. We only considered the first $3$ modes in generating Figure \ref{fig4}(b) since, due to overlapping error in eigenvalues for higher modes for random data, it is unclear which analytic eigenvalue the approximating eigenvalues correspond to without a direct visual comparison of eigenvectors. In this case, the errors in the spectral estimation for higher modes are larger than the differences between consecutive eigenvalues, violating the assumption in Theorem~\ref{conveigvec}.
\begin{figure}[htbp]
{\scriptsize \centering
\includegraphics[width = .5\textwidth]{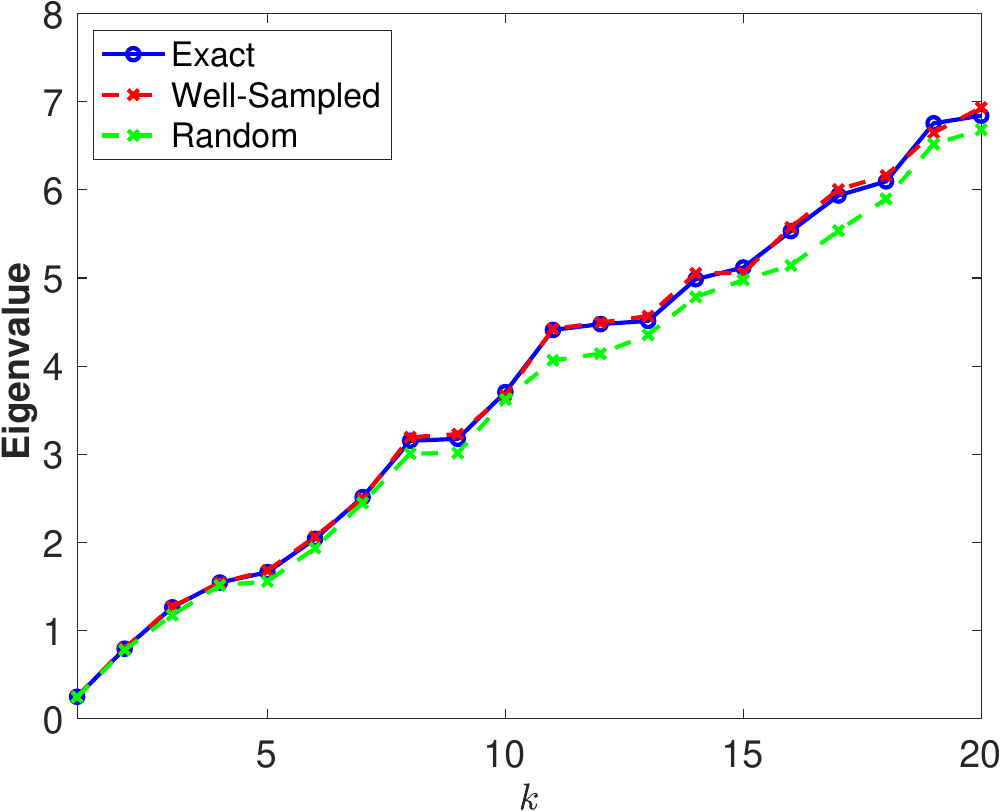}
\caption{Semi-Torus Example with $n = 128^2$: Comparison of eigenvalues.}
\label{fig5}}
\end{figure}

\begin{figure}[htbp]
{\scriptsize \centering
\begin{tabular}{ccc}
\normalsize  (a)  & \normalsize (b) & \normalsize (c) \\

\includegraphics[width=.3\textwidth]{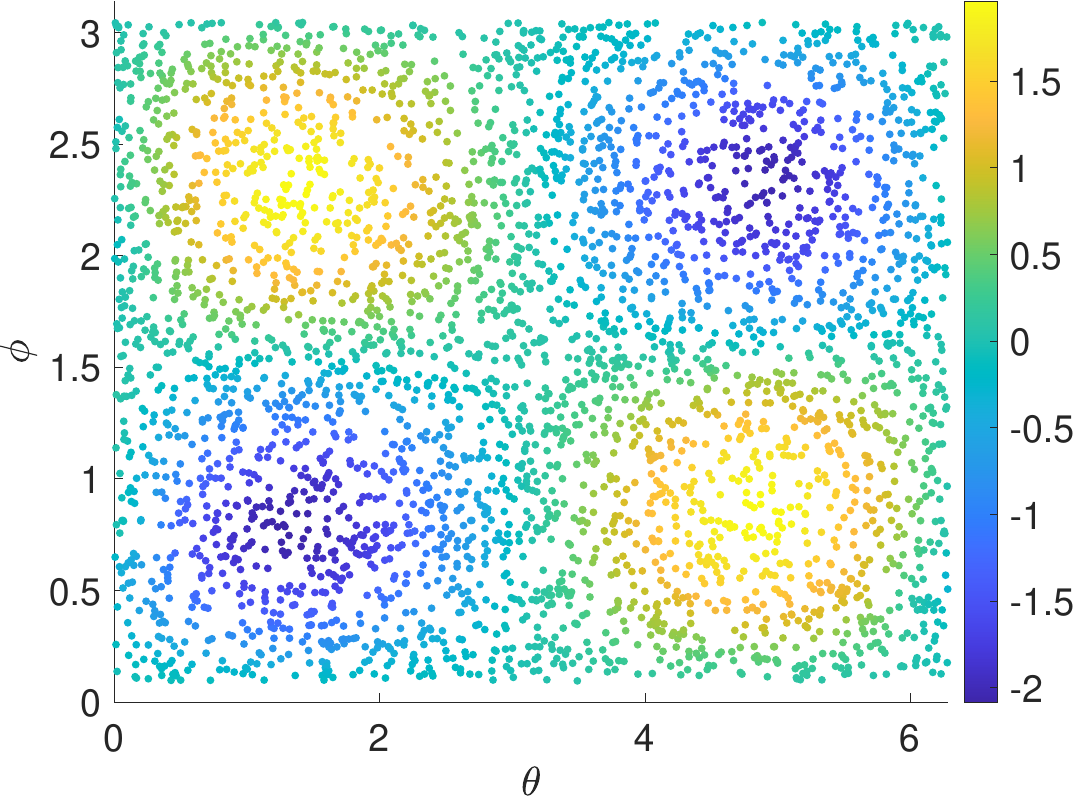} & 
\includegraphics[width=.3\textwidth]{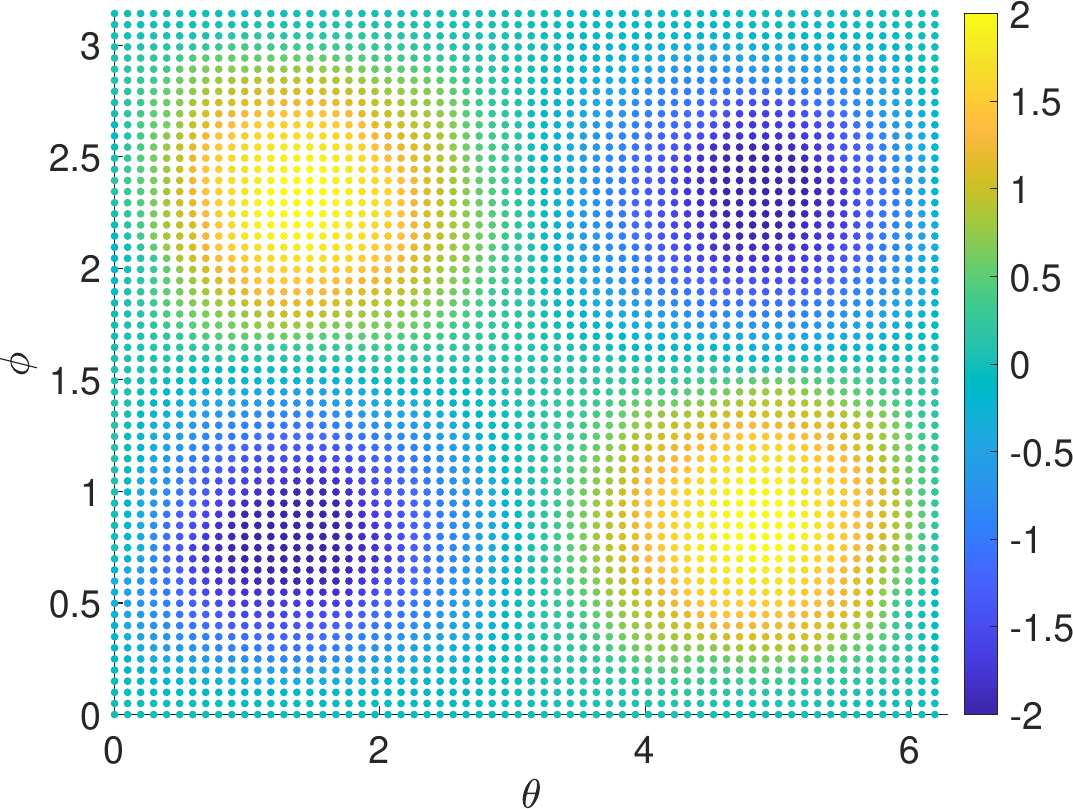} &
\includegraphics[width=.3\textwidth]{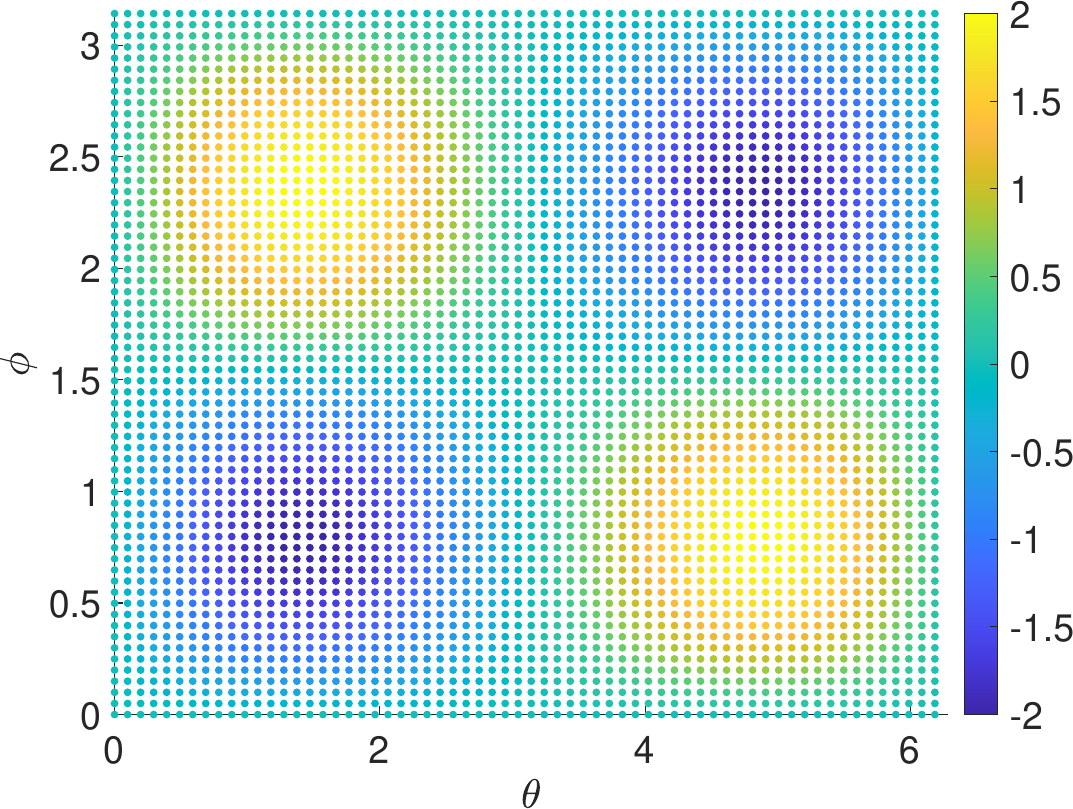} \\

\includegraphics[width=.3\textwidth]{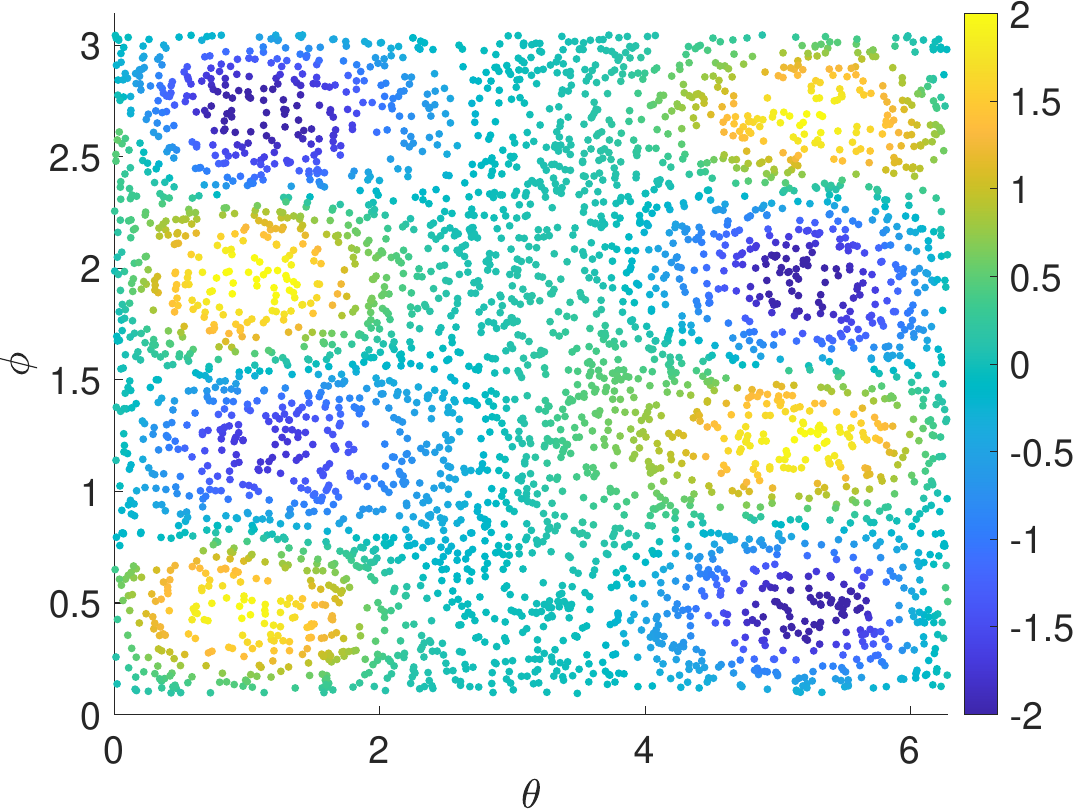} & 
\includegraphics[width=.3\textwidth]{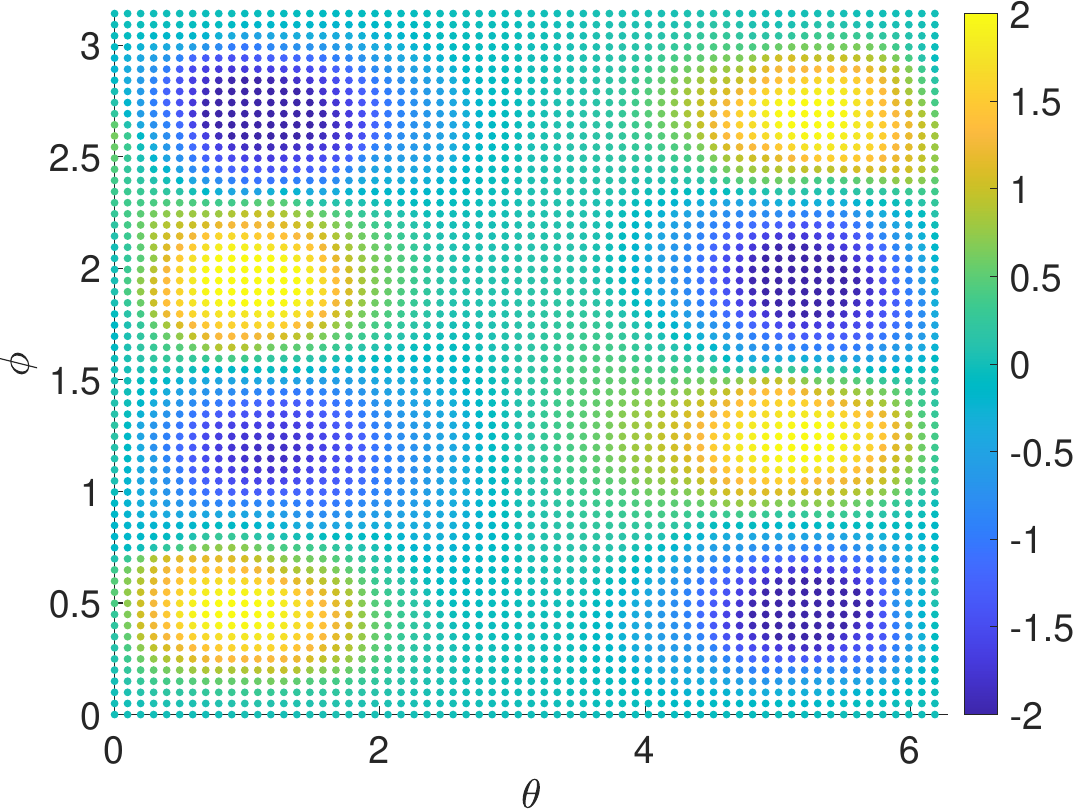} &
\includegraphics[width=.3\textwidth]{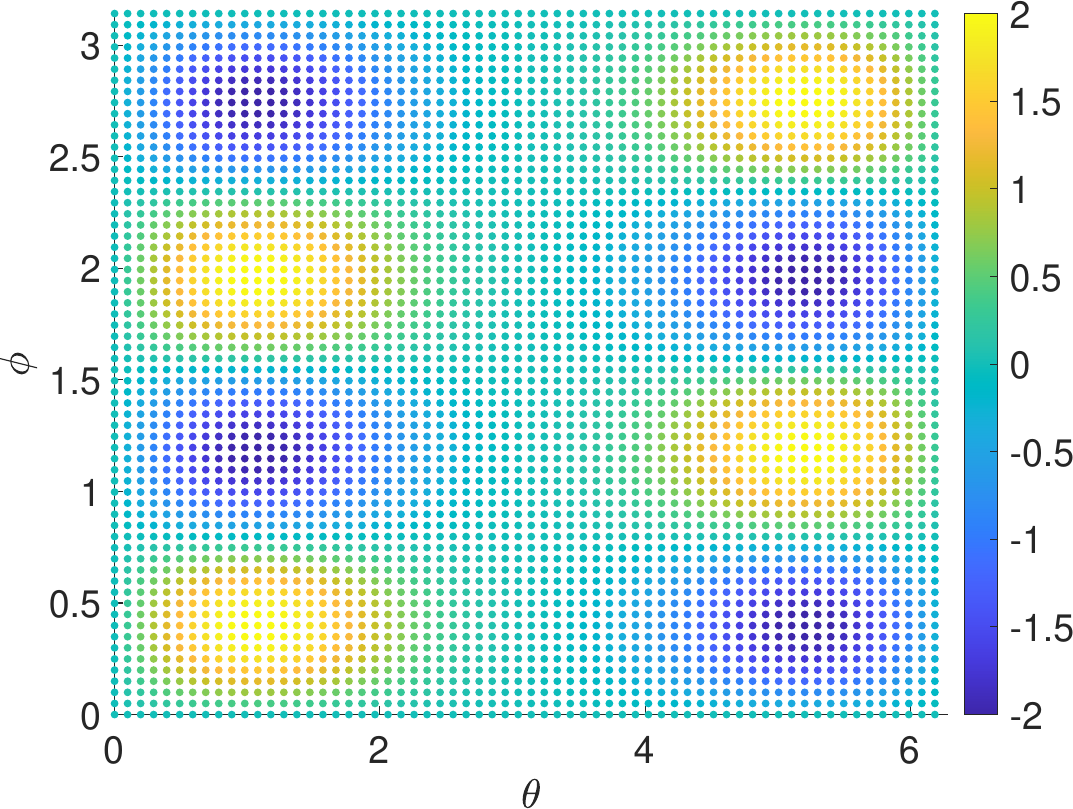} \\

\includegraphics[width=.3\textwidth]{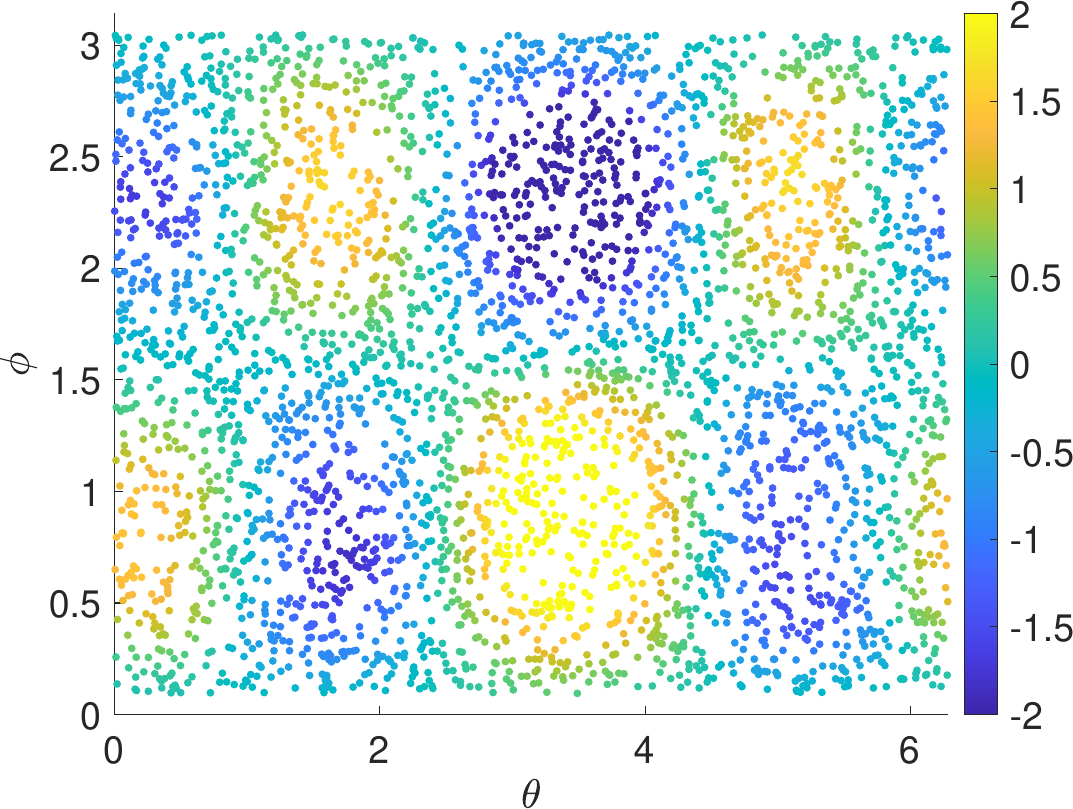} & 
\includegraphics[width=.3\textwidth]{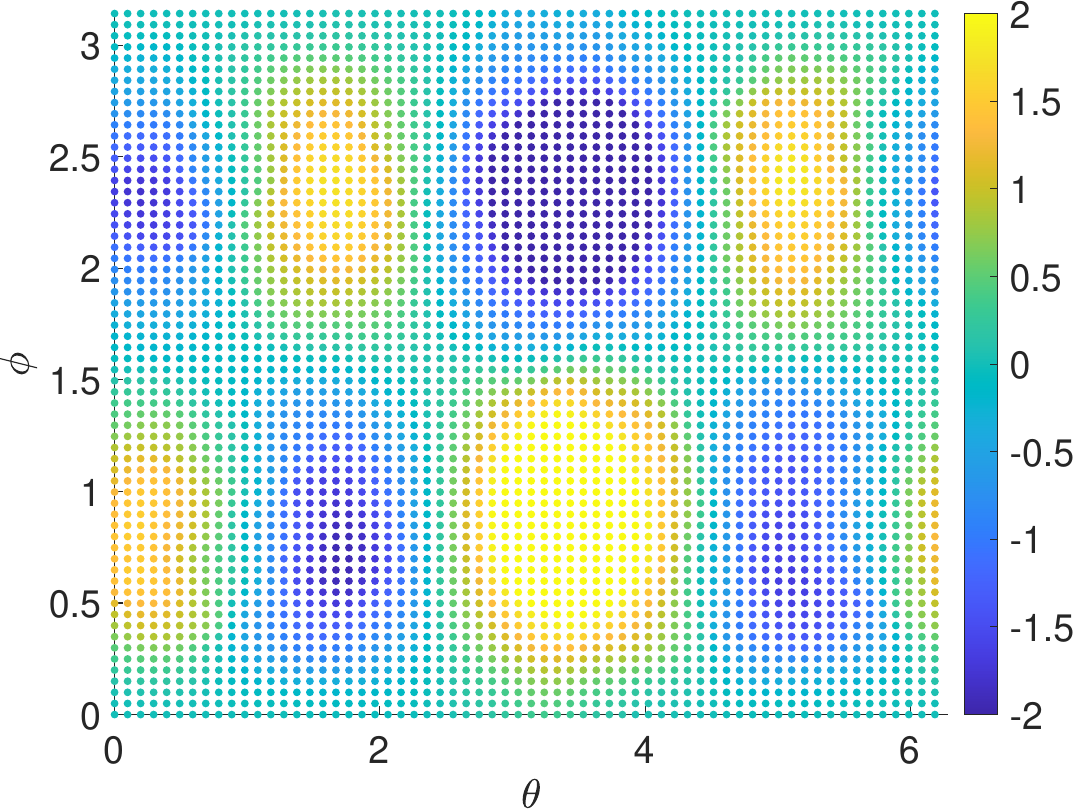} &
\includegraphics[width=.3\textwidth]{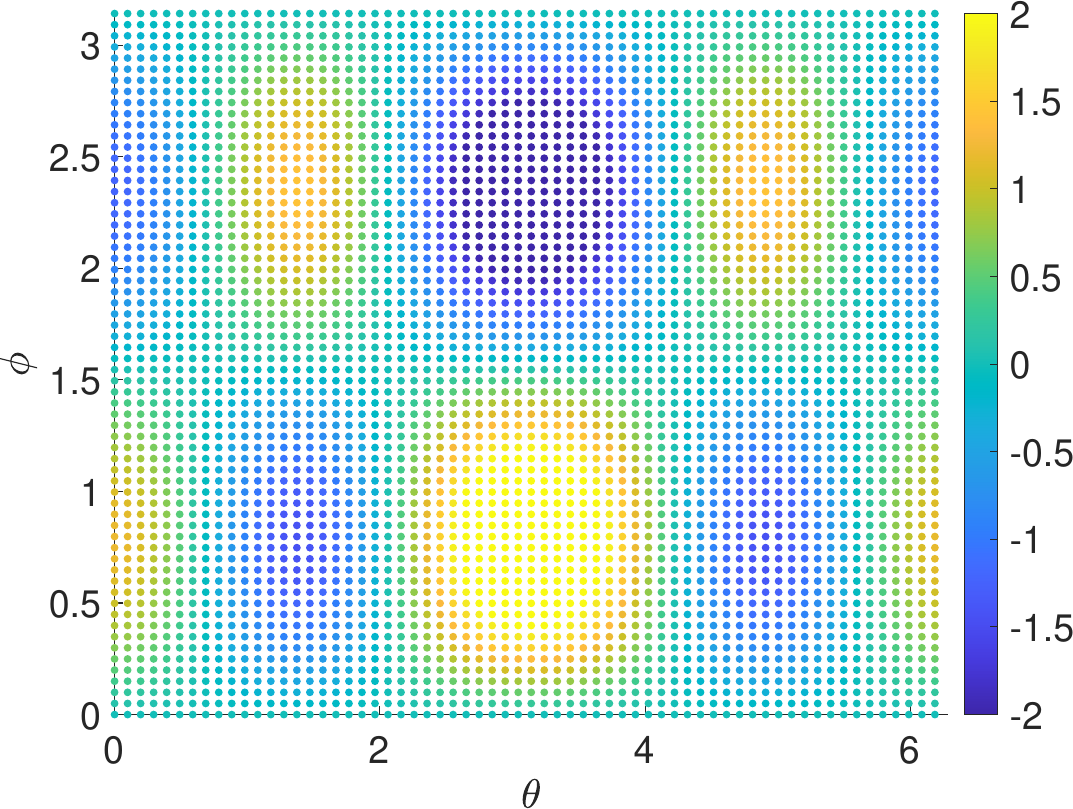} \\
\end{tabular} }
\caption{Semi-Torus Example with $n = 64^2$: Comparison of Eigenvectors in mode 6 (row 1), mode 13 (row 2) and mode 17 (row 3): (a) TGL with Random Data; (b) Nystrom Extension applied to (a); (c) Analytic. }
\label{fig6}
\end{figure}

For sufficiently large $n$, the spectral errors are no longer overlapping, and a larger portion of the spectrum can be analyzed, even for random data. In Figure \ref{fig5}, we display the performance of TGL in approximating the spectrum for fixed $n = 128^2$. For well-sampled data, the approximation and analytic are indistinguishable. For random data, TGL produces an accurate estimate for $k \leq 10$. In Figure~\ref{fig6}, we 
show the eigenvectors of higher modes for $n=64^2$. Here, the first column corresponds to the TGL eigenvectors constructed based on randomly distributed data, the second column shows the Nystrom interpolation to the grid points of the semi-analytic solutions, depicted in the third column. This result suggests that although the eigenvalue approximations become less accurate for large modes $k$, eigenvectors of higher modes can be well approximated, even with a relatively small amount of random data. In our numerics (results are not shown), we have seen that the first 20 modes can be well recovered (although they don't always occur in the same order as the semi-analytic solution when the data is random).

\section{Summary}\label{sec7} \noindent
In this paper, we established the spectral convergence (eigenvalues and eigenvectors) of a symmetrized graph Laplacian to the Laplace-Beltrami operator on closed manifolds by leveraging the result from \cite{rosasco2010learning} with the variational characterization of eigenvalues of the Laplace-Beltrami operator. For closed manifolds, the rates obtained are competitive with the best results in the literature. Moreover, using the weak convergence result from \cite{v:2020, vaughn2024diffusion}, we were able to adapt this proof with only slight modifications to demonstrate spectral convergence to the Laplace-Beltrami operator on compact manifolds with boundary satisfying either Neumann or Dirichlet boundary conditions. These results are, to our knowledge, the first spectral convergence results (with rates) in the setting of a compact manifold with Dirichlet boundary condition. The proofs of these results gave simple interpretations of two previously unexplained numerical results. Firstly, the spectral convergence of the Diffusion Maps algorithm on a compact manifold with boundaries to the Neumann Laplacian was fully explained as a combination of the weak consistency, as noted in \cite{vaughn2024diffusion}, and the well-known min-max result for the eigenvalues of the Neumann Laplacian on manifolds with boundary. Secondly, the numerical success of the Truncated Graph Laplacian (TGL) in approximating the Dirichlet Laplacian by truncating components of the graph Laplacian matrix corresponding to data points that are sufficiently close to the boundary was similarly seen as a combination of the two analogous facts for the Dirichlet Laplacian. Convergence of eigenvectors was obtained by using a method which is adapted from \cite{ct:2022}. This method showed the convergence of eigenvectors can be deduced from the convergence of eigenvalues, along with a $L^2(\mu_n)$ norm convergence result. In addition to these proofs, we numerically verified the effectiveness of TGL in approximating the Dirichlet Laplacian for some simple test examples of manifolds with boundary. In these simple examples, we saw that the empirical rate of convergence as $n \to \infty$ is somewhat faster than the theoretical bounds derived this paper when we truncate components of the graph Laplacian matrix corresponding to training data points whose distance from the boundary is less than $\sqrt{\epsilon}$, confirming the discussion in Remark~\ref{valid parameter regimes}. 

The spectral convergence results on manifolds with boundary outlined in this paper open up many avenues for future work. First, we suspect that the arguments in this paper can be modified to yield spectral convergence with rates when the discretization $\widetilde{L}_{\epsilon,n}$ is explicitly constructed using $K$-nearest neighbors. To our knowledge, the only other result of this form is reported in \cite{ct:2022}. Obtaining a result of this form in the setting of a compact manifold with boundary is of interest, particularly for determining the exact scaling of $K$ to optimize the convergence rate. Second, though rates of convergence of eigenvectors in $L^2(\mu_n)$ norm have been thoroughly investigated, this question remains primarily open for various other norms. Only recently have results for convergence of eigenvectors in  $L^\infty$ sense been established in \cite{dunson2021spectral}. Even so, such results only hold for closed manifolds, and not necessarily compact manifolds with boundaries. Third, though TGL yields convergence to the Dirichlet Laplacian, it is only a slightly modified version of the Diffusion Maps algorithm. It is plausible that further modifications, such as using the ghost points in constructing the TGL matrix as in \cite{hs:2022} or other methods such as Ghost Points Diffusion Maps (GPDM) \cite{jiang2023ghost}, can be shown to converge faster than TGL. Though \cite{jiang2023ghost} numerically demonstrated spectral convergence of both modified Diffusion maps and GPDM numerically for various Elliptic PDE's, the spectral convergence analysis is still an open problem as it involves solving the eigenvalue problem of non-symmetric matrices.

\section*{Acknowledgments}
The research of JH was partially supported under the NSF grants DMS-2207328, DMS-2229435, and the ONR grant N00014-22-1-2193.

\appendix

\section{Interpolation and Restriction via RKHS Theory} \label{rkhs-section} In this section, we provide a brief review of the RKHS theory, and adapt a well known result to our setting. The kernel $\hat{k}_\epsilon(x,y)$ is symmetric and positive definite. Hence, there exists a unique Hilbert space $\mathcal{H} {\subset} L^2(\mathcal{M})$, consisting only of smooth functions on $\mathcal{M}$, and possessing the reproducing property 
$$
f(x) = \langle f, (\hat{k}_\epsilon)_x \rangle_{\mathcal{H}} \qquad \forall f \in \mathcal{H},
$$
where $ (\hat{k}_\epsilon)_x$ denotes the function  $\hat{k}_\epsilon(x, \cdot)$.  Consider the inclusion operator $R_{\mathcal{H}}: \mathcal{H} \to L^2(\mathcal{M})$. We can form the self-adjoint, compact operator $R_\mathcal{H} R_\mathcal{H}^* : L^2(\mathcal{M}) \to L^2(\mathcal{M})$. It is a well known result that this coincides with $\hat{K}_\epsilon$ (see chapter 4 of \cite{steinwart2008support}). By the spectral theorem there is an orthonormal basis $\{\phi_0 , \phi_1, \dots \}$ consisting of eigenvectors of $R_\mathcal{H} R_\mathcal{H}^* $ with corresponding eigenvalues $\sigma^\epsilon_0 \geq \sigma^\epsilon_1 \geq \dots \searrow 0$. Note that for nonzero $\sigma^\epsilon_i$, defining $\psi_i : = \frac{1}{\sqrt{\sigma^\epsilon_i}}R_\mathcal{H}^* \phi_i$, we have that $\{\psi_0, \psi_1, \dots \}$ is an orthonormal set w.r.t the $\mathcal{H}$- inner product. With this in mind, we can define the Nystrom extension operator $\mathcal{N} : D(\mathcal{N}) \to \mathcal{H}$ as 
\begin{equation}
    \mathcal{N} \left(\sum_{i=0}^\infty c_i \phi_i \right) = \sum_{i=0}^\infty \frac{c_i}{\sqrt{\sigma^\epsilon_i}} \psi_i,\notag
\end{equation}
where $D(\mathcal{N})$ is the  subset of $L^2(\mathcal{M})$ consisting functions $f = \sum_{i=0}^\infty c_i \phi_i$ such that $\sum_{i=0}^\infty \frac{c^2_i}{\sigma_i^\epsilon} < \infty$. 
 Calling this an extension operator is justified in that $R_\mathcal{H} : \mathcal{H} \to L^2(\mathcal{M})$ is a left inverse for $\mathcal{N}$. Indeed, 
\begin{equation*}
    R_\mathcal{H} \mathcal{N} \left(\sum_{i=0}^\infty c_i \phi_i \right) = R_\mathcal{H} \left(\sum_{i=0}^\infty \frac{c_i}{\sqrt{\sigma_i^\epsilon}} \psi_i \right) = \sum_{i=0}^\infty \frac{c_i}{\sigma_i^\epsilon} R_\mathcal{H} R_\mathcal{H}^*\phi_i = \sum_{i=0}^\infty c_i \phi_i.
\end{equation*}
Hence, $\mathcal{N}$ establishes an isometric isomorphism between $D(\mathcal{N})$ and $\mathcal{H}$, such that for any $f, g\in D(\mathcal{N})$, where $f = \sum_i a_i \phi_i$ and  $g = \sum_i b_i \phi_i$, we have,
\begin{equation*}
    \left\langle f,g \right\rangle_{\mathcal{H}} : = \sum_i \frac{a_ib_i}{\sigma^\epsilon_i}.
\end{equation*}

While the above extension is theoretically useful, in practice it is more convenient to consider the following construction. Let $R_n : \mathcal{H} \to \mathbb{R}^n$ be the operator restricting functions to the data set. Define the Nystrom extension $\mathcal{N}_{\mu_n}:L^2(\mu_n) \to \mathcal{H}$ on the orthonormal eigenvectors $\{ \hat{u}_j\}$ of the matrix $R_nR_n^*$, whose components are $\frac{1}{n}\hat{k}_\epsilon(x_i,x_j)$ by
\begin{equation*}
    \mathcal{N}_{\mu_n} \hat{u}_j = \frac{1}{\sigma^{\epsilon,n}_{j}} R^*_n \hat{u}_j,
\end{equation*}
where $\sigma^{\epsilon,n}_{j}$ are the corresponding matrix eigenvalues. Defining $v_j := \frac{1}{\sqrt{\sigma^{\epsilon,n}_{j}}}R^*_n \hat{u}_j$, we can equivalently say $\mathcal{N}_{\mu_n} \hat{u}_j = \frac{1}{\sqrt{\sigma^{\epsilon,n}_{j}}} v_j$. This makes the analogy with the above Nystrom extension clear. Moreover, $\{v_j\}$ form an orthonormal set in $\mathcal{H}$. Indeed, 
\BEA
    \langle v_i,v_j\rangle_{\mathcal{H}} = \left\langle  \frac{1}{\sqrt{\sigma^{\epsilon,n}_{i}}}R^*_n \hat{u}_i, \frac{1}{\sqrt{\sigma^{\epsilon,n}_{j}}}R^*_n \hat{u}_j \right\rangle_{\mathcal{H}} = \frac{1}{\sqrt{\sigma^{\epsilon,n}_{i}}} \cdot \frac{1}{\sqrt{\sigma^{\epsilon,n}_{j}}} \left\langle \hat{u}_i, R_nR^*_n \hat{u}_j \right\rangle_{L^2(\mu_n)} = \delta_{i,j}.\notag
\EEA
Hence, given a vector $\hat{u} = \sum_{i=1}^n c_i \hat{u}_i$, we have a continuous analog $ \sum_{i=1}^n \frac{c_i}{\sqrt{\sigma^{\epsilon,n}_{i}}} v_i \in \mathcal{H}$. The following result, adapted from the spectral convergence result in \cite{rosasco2010learning}, relies heavily on the RKHS Theory mentioned above, is used extensively in the sequel. 
\begin{lemma}
\label{adapted-rbv-supplement}
( adapted from Proposition 10 of \cite{rosasco2010learning} ) With probability greater than $1 - \frac{2}{n^2}$, 
\begin{equation*}
    \sup_{1\leq i \leq n } |\lambda^\epsilon_i - \lambda_i^{\epsilon,n}| = \mathcal{O} \Big(\frac{ \sqrt{\log(n)}}{\epsilon^{d/2 + 1} \sqrt{n}}\Big),
\end{equation*}
as $n\to \infty$ for a fixed $\epsilon>0$, where $\{\lambda_i^\epsilon\}$ and $\{\lambda_i^{\epsilon,n}\}$ are the eigenvalues of $L_\epsilon$ as defined in Equation \eqref{symmetricDM} and $L_{\epsilon,n}$ as defined in Equation \eqref{Lepsilonn}, respectively.
\end{lemma} 
\begin{remark}\label{connectivity threshold}
While Lemma \ref{adapted-rbv-supplement} is valid as stated above, for any fixed $\epsilon>0$, convergence is also achieved when 
\[
\epsilon\gg  \left( \frac{\log(n)}{n} \right)^{\frac{1}{d+2}} \gg \left( \frac{\log n}{n} \right)^{\frac{2}{d}},
\]
where the last relation suggests that it is larger than the connectivity threshold for random graph \cite{penrose2003random,ct:2022}.
\end{remark}

\begin{proof}
 Let $\kappa$ be the essential supremum of $k_\epsilon$. Notice that $\frac{\kappa}{d_{\textrm{min}}}$ is an upperbound of the essential supremum of $\hat{k}_\epsilon$. Denote the $i$-th eigenvalue of the matrix $(K)_{ij} : = \frac{1}{n} \hat{k}_\epsilon(x_i,x_j)$ by $\sigma^{\epsilon,n}_{i}$, and the $i$-th eigenvalue of the integral operator $\hat{K}_\epsilon$ by $\sigma^{\epsilon}_i$. By Proposition $10$ in \cite{rosasco2010learning}, with probability greater than $1-2e^{-\tau}$,
\begin{equation*}
   \sup_{1\leq i \leq n} | \sigma^{\epsilon,n}_{i} - \sigma^{\epsilon}_{i} | \leq \frac{2\sqrt{2}\kappa \sqrt{\tau}}{\sqrt{n} \epsilon^{d/2} d_{\textrm{min}}}. 
\end{equation*}
However, we have already noted that $\lambda^\epsilon_i = \frac{2}{m_2\epsilon}(1 - \sigma^\epsilon_i)$ Similarly, $\lambda^{\epsilon,n}_i = \frac{2}{m_2\epsilon}(1 - \sigma^{\epsilon,n}_i)$. Hence, 
\begin{equation*}
    \sup_{1\leq i \leq n} | \lambda^{\epsilon,n}_{i} - \lambda^{\epsilon}_{i} | = \sup_{1\leq i \leq n} \left| \frac{2}{m_2\epsilon} \left(\sigma^{\epsilon,n}_{i} - \sigma^{\epsilon}_{i} \right) \right| \leq \frac{4\sqrt{2}\kappa \sqrt{\tau}}{m_2\epsilon^{d/2 + 1} \sqrt{n} d_{\textrm{min}}}.
\end{equation*}
Choose $\tau = \log(n^2)$. The above simplifies to,
\[ \sup_{1\leq i \leq n } |\lambda^\epsilon_i - \lambda_i^{\epsilon,n}| = \mathcal{O} \Big(\frac{\sqrt{\log n}}{\epsilon^{d/2 + 1} \sqrt{n}}\Big).
\]
This completes the proof. 
\end{proof} 
\begin{remark}
The above result relates the spectrum of the integral operator $L_\epsilon$ to the spectrum of the matrix $L_{\epsilon,n}$. What remains is to relate the spectra of the two matrices $L_{\epsilon,n}$ and $\tilde{L}_{\epsilon,n}$, and finally the spectrum of the integral operator $L_\epsilon$ to the spectrum of Laplace Beltrami. The former will be shown in Lemma \ref{matrix-eigenvalues-supplement}, while the latter is postponed until the proof of the main theorem. 
\end{remark}

\section{Lemma for Spectral Convergence}\label{sec3-supplement} 

In this section, we state and prove the convergence of eigenvalues of the matrices $L_{\epsilon,n}$ and $\tilde{L}_{\epsilon,n}$ that is needed in proving Theorem~\ref{spectralconvneumann}. Before this is stated, we emphasize that in \cite{rosasco2010learning}, the results hold when $\mathcal{M}$ is simply a subset of Euclidean space. In particular, the following result, as well as Lemma \ref{adapted-rbv-supplement}, holds for any compact manifold with or without boundary.  

\begin{lemma}
\label{matrix-eigenvalues-supplement}
With probability $1 - \frac{1}{n^2}$, 
\begin{equation*}
    \left|\tilde{\lambda}^{\epsilon,n}_i - \lambda^{\epsilon,n}_i \right| =  \mathcal{O}\left( \frac{ \sqrt{\log n}}{\epsilon^{d/2+1} \sqrt{n}}\right), 
\end{equation*}
as $n\to \infty$ for a fixed $\epsilon>0$.
\end{lemma}
\begin{remark} We point out that the estimate above is for the error of the eigenvalues of two matrices of fixed size $n\times n$, and the error (as $n\to \infty$) is induced by the Monte-Carlo estimate of $d(x) := \epsilon^{-d/2}\int_{\mathcal{M}}k_\epsilon(x,y)dV(y) \approx \frac{\epsilon^{-d/2}}{n}\sum_{k=1}^n k_\epsilon(x_i,x_k)$ that appears in the denominator of the kernel in Equation \eqref{kernelktilde} as an approximation to the kernel in Equation \eqref{kernelkhat}. 
Better estimates for the above quantity can be made, using sampling density approximation for manifolds with boundary \cite{bs:2017}. However, for our purposes, a rougher estimate suffices, since this rate is not worse than the one obtained in the Lemma~\ref{adapted-rbv-supplement}. 
\end{remark}
\begin{proof}
Notice that the $(i,j)$-th entry of $L_{\epsilon,n} - \tilde{L}_{\epsilon,n}$ is a constant multiple of 
$$
\frac{1}{\epsilon n} \left( \hat{k}_\epsilon (x_i,x_j) - \tilde{k}_\epsilon (x_i,x_j) \right).
$$
We begin by bounding the difference of terms in the denominator with high probability. Namely, fix $x_j$, and consider 
$$
\left| \frac{1}{n} \sum_{i=1}^n \epsilon^{-d/2} k_\epsilon(x_j,x_i) - \int_\mathcal{M} \epsilon^{-d/2} k_\epsilon(x_j,y) dV(y) \right|
$$
To this end, we use Chernoff bound (see, e.g., \cite{s:2006},\cite{bh:2016}). Note that 
\BEA
\mathbb{E}\left( \left[\epsilon^{-d/2} k_\epsilon(x_j, \cdot) - \epsilon^{-d/2} \int_\mathcal{M} k_\epsilon(x_j, y)dV(y) \right]^2  \right) =  \int_\mathcal{M}  \epsilon^{-d} |k_\epsilon (x_j,y)|^2 dV(y)   - \left( \epsilon^{-d/2}  \int_\mathcal{M}  k_\epsilon(x_j,y) dV(y) \right)^2 \notag
\EEA
For points $x_j$ within distance of $\sqrt{\epsilon}$ of the boundary, we use the asymptoptic expansion found in Lemma 9 in \cite{cl:2006}:
\BEA
 \int_\mathcal{M} \epsilon^{-d} |k_\epsilon (x_j,y)|^2 dV(y) = 2^{-d/2}\epsilon^{-d/2} \int_\mathcal{M}  \left(\frac{\epsilon}{2}\right)^{-d/2} \exp\left(- \frac{\| x_j - y \|^2}{\frac{\epsilon}{2}} \right) = 2^{-d/2} \epsilon^{-d/2} \left[m_0^{\frac{\epsilon}{2}}(x_j) + \mathcal{O}\left( \epsilon \right) \right] \notag
\EEA
where the $\epsilon^{1/2}$-term that involves $\frac{\partial f}{\partial n}$ vanishes since $f = 1$ is a constant. Similarly, \BEA
\left( \epsilon^{-d/2} \int_{\mathcal{M}} k_\epsilon(x_j,y) dV(y) \right)^2 = [m_0^\epsilon (x_j)]^2 + \mathcal{O}\left( \epsilon \right). \notag
\EEA 
In the above, $m_0^\epsilon$ is $\mathcal{O}(1)$. Hence, using Chernoff bound, 
\BEA
\mathbb{P}\left(  \left| \frac{1}{n} \sum_{i=1}^n \epsilon^{-d/2} k_\epsilon(x_j,x_i) - \epsilon^{-d/2}\int_\mathcal{M} k(x_j,y) dV(y)  \right| > \delta \right) \leq  \exp\left(\frac{-\delta^2 n}{4 m_0^{\epsilon/2} \epsilon^{-d/2}} \right) \notag.
\EEA
Balancing by choosing $\delta = \frac{ 2 \sqrt{C} \sqrt{\log(n^3)}}{ \epsilon^{d/4} \sqrt{n}}$, where $C$ denotes an upperbound for $m_0^{\epsilon/2}$  independent of $\epsilon$, we have that with probability higher than $1 - \frac{1}{n^3}$, 
$$
\left| \frac{1}{n} \sum_{i=1}^n \epsilon^{-d/2} k_\epsilon(x_j,x_i) - \int_\mathcal{M} \epsilon^{-d/2} k_\epsilon(x_j,y) dV(y) \right| = \mathcal{O}\left( \frac{\sqrt{\log(n)}}{ \epsilon^{d/4} \sqrt{n}} \right).
$$
The same argument holds for points $x_j$ at a distance greater than $\epsilon^{1/2}$ from the boundary, with the simpler expansion holding. Using the above bound $n$ times, we have that with probability higher than $1 - \frac{1}{n^2}$, 
$$
\left| \frac{1}{n} \sum_{i=1}^n \epsilon^{-d/2} k_\epsilon(x_j,x_i) - \int_\mathcal{M} \epsilon^{-d/2} k_\epsilon(x_j,y) dV(y) \right| = \mathcal{O}\left( \frac{\sqrt{\log(n)}}{ \epsilon^{d/4} \sqrt{n}} \right), 
$$
for $j = 1, \dots, n.$ Returning to the entries of $(L_{\epsilon,n} - \tilde{L}_{\epsilon,n})$, we have 
\BEA 
\frac{1}{\epsilon n} \left( \hat{k}_\epsilon (x_i,x_j) - \tilde{k}_\epsilon (x_i,x_j) \right) = \frac{1}{\epsilon n} \left( \frac{ \epsilon^{-d/2} k_\epsilon(x_i,x_j) \left( \tilde{d}(x_i) + \tilde{d}(x_j) - d(x_i) - d(x_j) \right) }{2(d(x_i) + d(x_j))(\tilde{d}(x_i) + \tilde{d}(x_j))}\right), \notag 
\EEA
where recall that 
$$
d(x) = \int \epsilon^{-d/2} k_\epsilon(x,y) dV(y),
$$
and we define 
$$
\tilde{d}(x) =  \frac{1}{n} \sum_{i=1}^n  \epsilon^{-d/2} k_\epsilon(x,x_i).
$$
The terms in the denominator are bounded away from $0$ by a constant, and the difference of $d, \tilde{d}$ terms in the numerator are $\mathcal{O}\left( \frac{ \sqrt{\log n} }{\epsilon^{d/4} \sqrt{n}} \right)$. Hence, the entries of $(L_{\epsilon,n} - \tilde{L}_{\epsilon,n})$ are bounded by constant multiple of 
$$
\frac{\sqrt{\log n}}{\epsilon \epsilon^{d/4} \sqrt{n}} \left(  \frac{\epsilon^{-d/2} k_\epsilon(x_i,x_j)}{n} \right). 
$$
It follows that 
\BEA
\sup_{\|v\|_2 = 1} \| L_{\epsilon,n} v - \tilde{L}_{\epsilon,n} v \|_2 &\leq& \frac{\textup{Const} \sqrt{\log n}}{\epsilon \epsilon^{d/4} n^{1/2}} \left( \sum_{i=1}^n \left( \sum_{j=1}^n \frac{ \epsilon^{-d/2} k_\epsilon(x_i,x_j) v(x_j)}{n} \right)^2 \right)^{1/2} \notag  \\
&\leq& \frac{\textup{Const}\sqrt{\log n}}{\epsilon \epsilon^{d/4} n^{1/2}} \left( \sum_{i=1}^n  \underbrace{\left( \sum_{j=1}^n \frac{ \epsilon^{-d/2} |k_\epsilon(x_i,x_j)|^2}{n} \right)}_{=\mathcal{O}(1)} \left(\frac{\epsilon^{-d/2}}{n} \sum_{j=1}^n v(x_j)^2\right)\right)^{1/2}    \notag  \\
&\leq& \frac{\textup{Const}\sqrt{\log n}}{\epsilon \epsilon^{d/4} n^{1/2}} \left( \sum_{i=1}^n  \underbrace{\left( \sum_{j=1}^n \frac{ \epsilon^{-d/2} |k_\epsilon(x_i,x_j)|^2}{n} \right)}_{=\mathcal{O}(1)} \left(\frac{\epsilon^{-d/2}}{n}\right)\right)^{1/2},\notag 
\EEA
where we've used the Cauchy-Schwarz inequality, and that $\|v\|_2=1$. The outside sum over $i$ cancels a factor of $n$, and we arrive at, 
\BEA
\sup_{\|v\|_2 = 1} \| L_{\epsilon,n} v - \tilde{L}_{\epsilon,n} v \|_2 \leq \frac{\textup{Const}\sqrt{\log n}}{\epsilon^{d/2 + 1} n^{1/2}}. \notag
\EEA
The desired result then follows from Weyl's inequality. 
\end{proof}

\section{Proof of Theorem~\ref{conveigvec}} \label{app:proof of eig}
 
Before we prove Theorem~\ref{conveigvec}, we will deduce several norm convergence results needed to prove convergence of the eigenvectors of $\tilde{L}_{\epsilon,n}$ to the eigenfunctions of $\Delta$. 
\comment{
In the course of the proof below, we will need the following definition. Given a manifold $\mathcal{M}$ with boundary, let $0< \gamma < 1/2$ and define 
\BEA
\mathcal{M}_{\epsilon^\gamma} : = \left\{ x \in \mathcal{M} : \textrm{inf}_{y \in \partial \mathcal{M}} \| x - y\|_g > \epsilon^\gamma \right\}, \label{Mawayfrbdry}
\EEA
as the set of points on the manifold whose geodesic distance from the boundary is at least $\epsilon^\gamma$.}

\begin{proposition}\label{dependence on r}
Let $\epsilon>0$ and $0<\gamma< 1/2$ such that $\epsilon^\gamma < r_C$, where $r_C>0$ denotes the normal collar width. Then
\[
 c \epsilon^\gamma \leq \textup{Vol}(\mathcal{M} \setminus \mathcal{M}_{\epsilon^\gamma})  \leq C \epsilon^{\gamma},
\]
for some $c,C>0$ that depend on $|\partial \mathcal{M}|$ but not on $\epsilon$.
\end{proposition}

\begin{proof}
Let $x \in \mathcal{M}$ whose geodesic distance from $\partial \mathcal{M}$ is $\epsilon^\gamma$. Without loss of generality, denote by $0\in \partial \mathcal{M}$ be the closest point (with respect to geodesic distance) to $x$ from the boundary. Since $\epsilon^\gamma < r_C$ where $r_C$ is the normal collar width, the Gauss lemma \cite{l:2003} states that the boundary of $\mathcal{M}_{\epsilon^\gamma}$  is a hypersurface parallel to $\partial \mathcal{M}$. 

Then we can employ a cylinder approximation to bound, 
\[
c |{\partial \mathcal{M}}| \epsilon^\gamma \leq  \textup{Vol}(\mathcal{M} \setminus \mathcal{M}_{\epsilon^\gamma})  \leq C |{\partial \mathcal{M}}| \epsilon^\gamma, 
\]
and the proof is complete.
\end{proof}

\begin{lemma}
\label{continuous norm lemma}
Fix $0 < \gamma < 1/2$ and any $\epsilon>0$ such that $\epsilon^\gamma < r_C$. Let $f \in C^\infty(\mathcal{M})$ be a Neumann eigenfunction of $\Delta$. Then 
\BEA
\|L_\epsilon f - \Delta f \|_{L^2(\mathcal{M})} = \mathcal{O}\left(\epsilon^{\gamma/2}\right), \notag
\EEA
as $\epsilon^{\gamma/2}\to 0$.
\end{lemma}
\begin{proof}
For $x\in \mathcal{M}_{\epsilon^\gamma}$, we have that 
\BEA
\frac{(K_\epsilon f)(x)}{d(x)} = \frac{m_0 f(x) + \epsilon \frac{m_2}{2}(\omega (x) f(x) + \Delta f (x)) + \mathcal{O}(\epsilon^2)}{m_0 + \epsilon \frac{m_2}{2}\omega(x) + \mathcal{O}(\epsilon^2)} = f(x) + \epsilon \frac{m_2 \Delta f (x)}{2m_0} + \mathcal{O}(\epsilon^2), \notag
\EEA
while 
\BEA
K_\epsilon  \left( \frac{f}{d} \right)(x) &=& m_0\frac{f(x)}{d(x)} + \epsilon \frac{m_2}{2}\left(\omega(x) \frac{f(x)}{d(x)} + \Delta(f/d)(x) \right) + \mathcal{O}(\epsilon^2) \notag \\
&=& \frac{m_0f(x)}{m_0 + \epsilon \frac{m_2}{2} \omega(x) + \mathcal{O}(\epsilon^2)} + \epsilon \frac{m_2}{2} \left( \frac{\omega(x)f(x)}{m_0 + \epsilon \frac{m_2}{2} \omega(x) + \mathcal{O}(\epsilon^2)} +  \Delta(f/d) \right) + \mathcal{O}(\epsilon^2) \notag \\
&=& f(x) - \epsilon  \frac{m_2 \omega(x) f(x)}{2 m_0} + \epsilon \frac{m_2 \omega(x) f(x)}{2m_0} + \epsilon \frac{m_2}{2} \Delta (f/m_0 + \mathcal{O}(\epsilon) ) + \mathcal{O}(\epsilon^2) \notag \\
&=& f(x) + \epsilon \frac{m_2}{2m_0} \Delta f(x) + \mathcal{O}(\epsilon^2). \notag
\EEA
From these expansions, we obtain
\[
L_\epsilon f(x) = \frac{1}{2} (L_{\textup{rw}, \epsilon}f(x) + L^*_{\textup{rw},\epsilon}f(x) ) = \Delta f(x) + \mathcal{O}(\epsilon^2),
\]
where we have used the normalization in the kernel that satisfies $m_0=\frac{m_2}{2}$. 
As for the point near the boundary, $x\in\mathcal{M}\backslash\mathcal{M}_{\epsilon^\gamma}$,
we use the asymptoptic expansion in Proposition 3.4.16 from \cite{v:2020}. Namely, note that 
\BEA
\frac{K_\epsilon f(x)}{d(x)} &=& \frac{m_0^\partial (x) f(x) + \sqrt{\epsilon}m_1^\partial(x) \left(\eta_x \cdot \nabla f (x) + \frac{d-1}{2}H(x)f(x)\right) + \mathcal{O}(\epsilon)}{m_0^\partial(x) + \sqrt{\epsilon}m_1^\partial(x) \left( \frac{d-1}{2}H(x)\right) + \mathcal{O}(\epsilon)} = f(x) + \sqrt{\epsilon} \left( \frac{m_1^\partial(x) \eta_x \cdot \nabla f(x)}{m^\partial_0(x)} \right) + \mathcal{O}(\epsilon), \notag
\EEA
while
\BEA
K_\epsilon \left(\frac{f}{d}\right)(x) &=& \frac{m_0^\partial(x)f(x)}{m_0^\partial(x) + \sqrt{\epsilon}m_1^\partial(x)\left( \frac{d-1}{2}H(x) \right) + \mathcal{O}(\epsilon)}  + \sqrt{\epsilon}m_1^\partial(x)\left( \eta_x \cdot \nabla(f/d) \notag +  \frac{H(x)f(x)}{m_0^\partial(x) + \mathcal{O}(\epsilon^{1/2})}  \right) + \mathcal{O}(\epsilon) \notag \\
&=& f(x) + \sqrt{\epsilon}\left( \frac{m_1^\partial(x) \eta_x \cdot  \nabla(f)(x)}{m_0^\partial(x)} \right) + \mathcal{O}(\epsilon).  \notag
\EEA
and deduce
\[
L_\epsilon f(x) = O(1), \quad \forall x\in \mathcal{M}\backslash \mathcal{M}_{\epsilon^\gamma}
\]
and so does,
\[
L_\epsilon f(x) - \Delta f(x) = O(1) 
\]
Since $f\in C^\infty(\mathcal{M})$ and both the kernel and the manifold are smooth, then it is clear that the constant in the big-O can be uniformly bounded since they are simply Taylor expansion terms. Thus,
\BEA
\|L_\epsilon f - \Delta f \|^2_{L^2(\mathcal{M})} =  \int_{\mathcal{M}_{\epsilon^\gamma}} (L_\epsilon f(x) - \Delta f(x))^2\,dV(x)  
+ \int_{\mathcal{M}\backslash\mathcal{M}_{\epsilon^\gamma}} (L_\epsilon f(x) - \Delta f(x))^2\,dV(x)  \notag \leq  C \epsilon^2 + C_2 \epsilon^\gamma,
\EEA
where we have used the fact that $\text{Vol}{\mathcal{M}_{\epsilon^\gamma}}) \leq \text{Vol}(\mathcal{M}) = \mathcal{O}(1)$ and Proposition~\ref{dependence on r} to upper bound $\text{Vol}(\mathcal{M}\backslash \mathcal{M}_{\epsilon^\gamma})$, for some constants $C, C_2$ which are independent of $x$ and $\epsilon$. Thus, we have the two-norm is bounded by $\mathcal{O}(\epsilon^{\frac{\gamma}{2}})$. 
\end{proof}
In the above proof, we see that convergence holds sufficiently far away from the boundary even when $f$ does not satisfy Neumann boundary conditions. This yields the following corollary. 
\begin{corollary}\label{L2boundcorollary}
Fix $0<\gamma<1/2$, and let $\mathcal{M}_{\epsilon^\gamma}$ denote the set of points whose distance from the boundary is at least $\epsilon^\gamma < r_C$. Then for any $f \in C^\infty(\mathcal{M})$,
\BEA
\|L_\epsilon f - \Delta f\|_{L^2(\mathcal{M}_{\epsilon^\gamma})} = \mathcal{O}(\epsilon). \notag
\EEA
Alternatively, if the manifold has no boundary,
\BEA
\|L_\epsilon f - \Delta f\|_{L^2(\mathcal{M})} = \mathcal{O}(\epsilon). \notag
\EEA
\end{corollary}
The following result is obtained by multiple applications of the law of large numbers to the above lemma. 
\begin{lemma}
\label{discrete norm lemma}
Fix $0 < \gamma < 1/2$, and let $f \in C^\infty(\mathcal{M})$ be a Neumann eigenfunction of $\Delta$. For any $\epsilon>0$ such that $\epsilon^\gamma < r_C$, with probability higher than $ 1 - \frac{4}{n^2}$,   
\BEA
\|\widetilde{L}_{\epsilon,n} R_n f - R_n \Delta f\|_{L^2(\mu_n)} = \mathcal{O}\left( \frac{\sqrt{\log n}}{\epsilon^{d/2+1}\sqrt{n}}, \epsilon^{\gamma/2} \right). \notag
\EEA
\end{lemma} 
\noindent{\color{black}Balancing these two error rates, the convergence rate is of order $$\epsilon^{\gamma/2} \propto \left(\frac{\log(n)}{n} \right)^\frac{\gamma}{2d+4+2\gamma}.$$ So, for $\gamma \to 1/2$, the rate is effectively $\epsilon^{1/4} \propto \left(\frac{\log(n)}{n} \right)^{\frac{1}{4d+10}}$.}

\begin{proof}
Consider the following:
\BEA
\|\widetilde{L}_{\epsilon,n} R_n f - R_n \Delta f\|_{L^2(\mu_n)} \leq \|\widetilde{L}_{\epsilon,n} R_n f -  L_{\epsilon,n} R_n  f\|_{L^2(\mu_n)}  +
 \|L_{\epsilon,n} R_n f - R_n L_\epsilon  f\|_{L^2(\mu_n)} + \|R_n L_{\epsilon} f - R_n \Delta f\|_{L^2(\mu_n)}\notag.
\EEA
For the first term, note that in the proof of Lemma \ref{matrix-eigenvalues-supplement}, it was shown that 
$$
\sup_{\|v\|_2 = 1} \|\tilde{L}_{\epsilon,n} v - L_{\epsilon,n} v \|_2 = \mathcal{O}\left(\frac{\sqrt{\log n}}{\epsilon^{d/2+1}\sqrt{n}}\right)
$$
with probability higher than $1 - \frac{1}{n^2}$. Since 
$$
\frac{\|\tilde{L}_{\epsilon,n} v - L_{\epsilon,n} v \|_2}{\|v\|_2} = \frac{\|\tilde{L}_{\epsilon,n} v - L_{\epsilon,n} v \|_{L^2(\mu_n)}}{\|v\|_{L^2(\mu_n)}},
$$
it follows that 
$$
\sup_{\|v\|_{L^2(\mu_n)} = 1} \|\tilde{L}_{\epsilon,n} v - L_{\epsilon,n} v \|_{L^2(\mu_n)} = \mathcal{O}\left(\frac{\sqrt{\log n}}{\epsilon^{d/2+1}\sqrt{n}}\right)
$$
with probability higher than $1 - \frac{1}{n^2}$. Hence,
$$
\|\widetilde{L}_{\epsilon,n} R_n f -  L_{\epsilon,n} R_n  f\|_{L^2(\mu_n)} \leq \|\widetilde{L}_{\epsilon,n} -  L_{\epsilon,n}\|_{2} \| R_n f \|_{L^2(\mu_n)}.
$$
The matrix norm is $\mathcal{O}\left(\frac{\sqrt{\log n}}{\epsilon^{d/2+1}\sqrt{n}}\right)$, by the calculation in Lemma \ref{matrix-eigenvalues-supplement}, while the second term approaches $\|f\|_{L^2(\mathcal{M})}$ as $n \to \infty$, and is hence bounded by a constant for sufficiently large $n$. For the second term, fix an $i$, and note that 
\BEA
\Bigg|\hat{k}_\epsilon(x_i,y) f(y) - \int_{\mathcal{M}} \hat{k}_\epsilon(x_i,z) f(z) dV(z)\Bigg|   \leq\Bigg| \frac{ \epsilon^{-d/2} k_\epsilon(x_i,y)}{d(x_i)} f(y) - \frac{1}{d(x_i)} \int_{\mathcal{M}} \epsilon^{-d/2} k_\epsilon(x_i,z) f(z) dV(z)\Bigg| \notag\hspace{.27cm} \\
+ \Bigg| \epsilon^{-d/2} k_\epsilon(x_i,y)(f/d)(y) - \int_{\mathcal{M}} \epsilon^{-d/2} k_\epsilon(x_i,z) (f/d)(z) dV(z)\Bigg| \notag\\
\leq   \frac{1}{d_{\textup{min}}}\Bigg| \epsilon^{-d/2} k_\epsilon(x_i,y)f(y) - \int_{\mathcal{M}} \epsilon^{-d/2} k_\epsilon(x_i,z) f(z) dV(z)\Bigg| \notag  \hspace{.45cm}\\
+ \Bigg| \epsilon^{-d/2} k_\epsilon(x_i,y)(f/d)(y) - \int_{\mathcal{M}} \epsilon^{-d/2} k_\epsilon(x_i,z) (f/d)(z) dV(z)\Bigg|. \notag \hspace{-.1cm}
\EEA
Hence, by the same Chernoff argument found in the proof of Lemma \ref{matrix-eigenvalues-supplement}, we have that with probability higher than $1 - \frac{1}{n^3}$, 
\BEA
\left| \frac{1}{n} \sum_{j=1}^n \hat{k}_\epsilon(x_i,x_j) f(x_j) - \int \hat{k}_\epsilon(x_i,z) f(z) dV(z)\right| = \mathcal{O}\left( \frac{\sqrt{\log n}}{\epsilon^{d/4} \sqrt{n}}\right). \notag
\EEA
It follows that 
\BEA
|(L_{\epsilon,n} R_n f)_i - L_\epsilon f (x_i) | = \mathcal{O}\left( \frac{\sqrt{\log n}}{\epsilon^{d/4+1} \sqrt{n}}\right). 
\EEA
Using this result for each $i$, we have that with probability higher than $1 - \frac{1}{n^2}$, the second term is $\mathcal{O}\left( \frac{\sqrt{\log n}}{\epsilon^{d/4+1}\sqrt{n}}\right)$. For the third and final term, since $f$ is Neumann, $|L_\epsilon f(x) - \Delta f(x)|^2 - \int_\mathcal{M} |L_\epsilon f - \Delta f|^2 dV =  \mathcal{O}(1)$. And so, using Hoeffding's inequality, with probability higher than $1 - \frac{2}{n^2}$, the third term $\|R_n L_\epsilon f - R_n \Delta f\|_{L^2(\mu_n)}$ term differs from its mean by Monte-Carlo error $\mathcal{O}\left( \sqrt{\log n}/\sqrt{n} \right)$. Hence, by Lemma~\ref{continuous norm lemma}, the error dominating the third term is $\mathcal{O}(\epsilon^{\gamma/2})$. This completes the proof. 
\end{proof}

\begin{remark}\label{improvedrate}
The rate of the above Lemma can be made significantly faster for manifolds without boundary. Particularly, the rate $\epsilon^{\gamma/2}$ is replaced with $\epsilon$. Balancing $\frac{\sqrt{\log n}}{\epsilon^{d/2+1}\sqrt{n}} \propto \epsilon$, we achieve $\epsilon \propto\left( \frac{\log n}{n}\right)^{\frac{1}{d+4}}$.
\end{remark}

With the above results, we are now ready to prove Theorem~\ref{conveigvec}.
Let $C'$ denote the largest of the constants absorbed into the rate of the eigenvalue convergence result. In the proof of the following theorem, it is important that the enumeration of the eigenvalues
\begin{equation*}
    \lambda_1 \leq \lambda_2 \leq \dots \quad \textrm{ and } \quad \tilde{\lambda}^{\epsilon,n}_1 \leq \tilde{\lambda}^{\epsilon,n}_2 \leq \dots \tilde{\lambda}^{\epsilon,n}_n
\end{equation*}
counts the multiplicities. For example, if the first eigenvalue of $\Delta$ has multiplicity $3$, then $\lambda_1 = \lambda_2 = \lambda_3$. The following proof adapts the argument found in \cite{ct:2022} to this setting.

\begin{proof}[Proof of Theorem~\ref{conveigvec}]
Fix any $\ell$. Let $k$ be the geometric multiplicity of the eigenvalue $\lambda_\ell$. There is an $i$ such that $ \lambda_{i+1} =  \lambda_{i+2} = \dots = \lambda_\ell = \dots = \lambda_{i+k}$. Let 
\begin{equation*}
    c_\ell = \frac{1}{2} \textrm{ min} \left\{ |\lambda_{\ell} - \lambda_{i}|, |\lambda_{\ell} - \lambda_{i+k+1}| \right\}.
\end{equation*}
{\color{black}Let $C'$ be the largest constant absorbed in the spectral error bound in Theorem~\ref{spectralconvneumann} such that} $C' \left( \frac{\color{black}\sqrt{\log(n)}}{\epsilon^{d/2+1}n^{1/2}} + \epsilon^{1/2}  \right) < c_\ell$, then with probability higher than $1 - \frac{6}{n^2}$,
\begin{equation*}
    |\tilde{\lambda}^{\epsilon,n}_{i} - \lambda_{i}| < c_\ell , \quad |\tilde{\lambda}^{\epsilon,n}_{i+k+1} - \lambda_{i+k+1}| < c_\ell.
\end{equation*}
Let $\hat{u}_{1}, \dots \hat{u}_{n}$ be an orthonormal basis of $L^2(\mu_n)$ consisting of eigenvectors of $\tilde{L}_{\epsilon,n}$, where $\hat{u}_j$ has eigenvalue $\tilde{\lambda}^{\epsilon,n}_j$.  Let $S$ be the $k$ dimensional subspace of $L^2(\mu_n)$ corresponding to the span of $\{ \hat{u}_{j} \}^{i+k}_{j=i+1}$, and let $P_S$ (resp. $P^\perp_S$) denote the projection onto $S$ (resp. orthogonal complement of $S$). Let $f$ be a norm $1$ eigenfunction of $\Delta$ corresponding to eigenvalue $\lambda_\ell$. Notice that
\begin{equation*}
    P^\perp_S R_n \Delta f = \lambda_\ell P^\perp_S R_n f = \lambda_{\ell} \sum_{j \neq i+1, \dots , i+k} \langle R_n f, \hat{u}_j \rangle_{L^2(\mu_n)} \hat{u}_j.
\end{equation*}
Similarly, 
\begin{equation*}
     P^\perp_S \tilde{L}_{\epsilon,n} R_n f = \sum_{j \neq i+1, \dots , i+k} \tilde{\lambda}^{\epsilon,n}_{j}\langle  R_n f, \hat{u}_j \rangle_{L^2(\mu_n)} \hat{u}_j.
\end{equation*}
Hence,
\BEA
\Bigg\|| P^\perp_S R_n \Delta f - P^\perp_S \tilde{L}_{\epsilon,n} R_n f, R_n \Bigg\|_{L^2(\mu_n)} 
&=& \Bigg\|  \sum_{j \neq i+1, \dots , i+k} (\lambda_\ell - \tilde{\lambda}^{\epsilon,n}_{j} )\langle  R_n f, \hat{u}_j \rangle_{L^2(\mu_n)} \hat{u}_j \Bigg\|_{L^2(\mu_n)} \notag\hspace{2.5cm} \\
&\geq& \textup{min}\left\{ \left| \lambda_\ell - \tilde{\lambda}^{\epsilon,n}_i \right|, \left|\lambda_\ell - \tilde{\lambda}^{\epsilon,n}_{i+k+1}\right|\right\}\Bigg\| \sum_{j \neq i+ 1 , \dots , i + k} \langle R_n f , \hat{u}_j \rangle \hat{u}_j \Bigg\|_{L^2(\mu_n)} \notag \\
&=& \textup{min}\left\{ \left| \lambda_\ell - \tilde{\lambda}^{\epsilon,n}_i \right|, \left|\lambda_\ell - \tilde{\lambda}^{\epsilon,n}_{i+k+1}\right|\right\} \Bigg\| P_S^\perp R_nf \Bigg\|_{L^2(\mu_n)} \notag 
\EEA
But $P^\perp_S$ is an orthogonal projection, so  
\BEA
\textrm{ min} \left\{ |\lambda_\ell - \tilde{\lambda}^{\epsilon,n}_{i}|, |\lambda_\ell - \tilde{\lambda}^{\epsilon,n}_{i+k+1}|  \right\} \| P_S^\perp R_n f \|_{L^2(\mu_n)}  \leq \left\|P^\perp_S R_n \Delta f - P^\perp_S \tilde{L}_{\epsilon,n} R_n f \right\|_{L^2(\mu_n)}   
= \| R_n \Delta f - \tilde{L}_{\epsilon, n} R_n f \|_{L^2(\mu_n)}.\notag
\EEA
Without loss of generality, assume $\textrm{ min} \left\{ |\lambda_\ell - \tilde{\lambda}^{\epsilon,n}_{i}|, |\lambda_\ell - \tilde{\lambda}^{\epsilon,n}_{i+k+1}|  \right\} = |\lambda_\ell - \tilde{\lambda}^{\epsilon,n}_{i}|$. Notice that
\begin{equation*}
    |\lambda_\ell - \tilde{\lambda}^{\epsilon,n}_{i}| \geq \left|  |\lambda_\ell - \lambda_{i}| - |\lambda_{i} - \tilde{\lambda}^{\epsilon,n}_{i}| \right| > c_\ell,
\end{equation*}
by hypothesis. Hence, using Lemma~\ref{discrete norm lemma}, with probability higher than $1-\frac{4}{n^2}$,
\begin{equation*}
    \| P_S^\perp R_n f \|_{L^2(\mu_n)} \leq \frac{1}{c_\ell} \| R_n \Delta f - \tilde{L}_{\epsilon, n} R_n f \|_{L^2(\mu_n)} = \frac{1}{c_\ell}  \mathcal{O}\left(\frac{ \log (n)^{1/2} }{\epsilon^{ \frac{d+2}{2}}n^{1/2}},\epsilon^{\gamma/2} \right),
\end{equation*}
where we have temporarily left the constant $\frac{1}{c_\ell}$ out of $\mathcal{O}$ to emphasize that this is inversely proportional to the spectral gap. Notice that $P_S^\perp R_n f = R_n f - P_S R_n f$. Hence, if $\{f_1, f_2, \dots , f_k \}$ are an orthonormal basis for the eigenspace corresponding to $\lambda_\ell$, applying Lemma \ref{discrete norm lemma} $k$ times, we see that
with probability  $1 - \frac{4k}{n^2}$, 
\begin{equation*}
    \|R_n f_j - P_S R_n f_j \|_{L^2(\mu_n)} = \frac{1}{c_\ell} \mathcal{O}\left(\frac{ \log (n)^{1/2} }{\epsilon^{ \frac{d+2}{2}}n^{1/2}},\epsilon^{\gamma/2} \right),   \qquad \textrm{ for } j = 1, 2, \dots k.\label{L2eigvecerror}
\end{equation*}

Let $C_l$ denote an upper bound on the essential supremum of the eigenvectors $\{f_1, f_2, \dots , f_k \}$. For any $1 \leq i,j \leq k$, 
\begin{equation*}
    \left|f_i(x)f_j(x) - \int_\mathcal{M} f_j(y)f_i(y) d\mu(y)\right| \leq C_l^2(1+\mbox{Vol}(\mathcal{M})).
\end{equation*}
Hence, using Hoeffding's inequality with $\alpha = 2\sqrt{2}C_l {\frac{\log(n)}{\sqrt{n}}}$, with probability $1 - \frac{2}{n^2}$, \begin{equation*}
    \left|\frac{1}{n}\sum_{l=1}^n f_i(x_l)f_j(x_l) - \int_\mathcal{M} f_i(y)f_j(y) d\mu(y)  \right| < \alpha.
\end{equation*}
Since these are orthonormal in $L^2(\mathcal{M})$, by Hoeffding's inequality used $k^2$ times, it is easy to see that with probability $1 - \frac{2k^2}{n^2}$,
\begin{equation*}
    \langle R_nf_i, R_n f_j \rangle_{L^2(\mu_n)} = \delta_{ij} + \mathcal{O}\left( \frac{\sqrt{\log(n)}}{\sqrt{n}} \right). 
\end{equation*}
Combining the above two estimates and absorbing the constant $\frac{1}{c_\ell}$, we see that with a total probability of $1 - \frac{2k}{n^2} - \frac{4k}{n^2} - \frac{6}{n^2}$ we have that
\BEA
\langle P_S R_n f_i , P_S R_n f_j \rangle_{L^2(\mu_n)} &=& \langle R_n f_i , R_n f_j \rangle_{L^2(\mu_n)} + 
\langle R_n f_i - P_S R_n f_i , R_n f_j - P_S R_n f_j \rangle_{L^2(\mu_n)}\notag
\\ &=& \delta_{ij} + \mathcal{O}\left( \frac{\sqrt{\log(n)}}{\sqrt{n}}, \frac{ \log (n) }{\epsilon^{d+2}n},\epsilon^{\gamma}\right),\notag
\EEA
or 
\BEA
\| P_S R_n f_i\|_{L^2(\mu_n)} = \sqrt{1+  \mathcal{O}\left(\frac{ \log (n) }{\epsilon^{d+2}n},\epsilon^\gamma\right)} = 1+ \mathcal{O}\left(\frac{ \log (n) }{\epsilon^{d+2}n} ,\epsilon^\gamma\right),\notag
\EEA
whenever $C' \left( \frac{\sqrt{\log n}}{\epsilon^{d/2+1}n^{1/2}}  + \epsilon^{1/2}  \right) < c_\ell$.
Letting $v_1 = \frac{P_SR_nf_1}{\| P_S R_n f_1 \|_{L^2(\mu_n)}}$, we see that, 
\[\| P_SR_n f_1 - v_1 \|_{L^2(\mu_n)} =  \mathcal{O}\left(\frac{ \log (n) }{\epsilon^{d+2}n} ,\epsilon^{\gamma}\right).\]
 Letting $\tilde{v}_2 = P_SR_nf_2 - \frac{\langle P_SR_nf_1, P_SR_nf_2 \rangle_{L^2(\mu_n)}}{\|P_SR_n f_1 \|^2_{L^2(\mu_n)}} P_SR_nf_1$, and $v_2 = \frac{\tilde{v}_2}{\|\tilde{v}_2\|_{L^2(\mu_n)}}$, it is easy to see that 
\begin{equation*}
\|P_SR_nf_2 - \tilde{v}_2\|_{L^2(\mu_n)} =  \mathcal{O}\left(\frac{ \log (n) }{\epsilon^{d+2}n} ,\epsilon^{\gamma}\right),
\end{equation*}
and hence, 
\begin{equation*}
   \|P_SR_nf_2 - v_2\|_{L^2(\mu_n)} =  \mathcal{O}\left(\frac{ \log (n) }{\epsilon^{d+2}n} ,\epsilon^{\gamma}\right).
\end{equation*}
Continuing this way, we see that the Gram-Schmidt procedure on  $\{ P_SR_nf_j \}^k_{j=1}$ yields an orthonormal set of $k$ vectors $\{ v_j\}_{j=1}^k$ spanning $S$ such that 
\begin{equation*}
    \|P_SR_nf_j - v_j\|_{L^2(\mu_n)} =  \mathcal{O}\left(\frac{ \log (n) }{\epsilon^{d+2}n} ,\epsilon^{\gamma}\right),  \qquad j = 1,2,\dots, k, 
\end{equation*}
and therefore 
\BEA
    \|R_nf_j - v_j\|_{L^2(\mu_n)} \leq \| R_n f_j - P_S R_n f_j \|_{L^2(\mu_n)} + \| P_S R_n f_j - v_j \|_{L^2(\mu_n)} 
    =  \mathcal{O}\left( \frac{ \log (n)^{1/2} }{\epsilon^{ \frac{d+2}{2}}n^{1/2}},\epsilon^{\gamma/2}\right) ,\qquad j = 1,2,\dots, k, \notag
\EEA
where the error is dominated by the first term on the right hand side. Therefore, for any eigenvector $v = \sum_{j=1}^k b_j v_j$ with $L^2(\mu_n)$ norm $1$, notice that $f = \sum_{j=1}^k b_j f_j$ is a $L^2(\mathcal{M})$ norm $1$ eigenfunction of $\Delta$ with eigenvalue $\lambda_l$ satisfying 
\begin{equation*}
    \| R_n f - v \|^2_{L^2(\mu_n)} \leq \sum_{j=1}^k |b_j|^2 \| R_n f_j - v_j \|^2_{L^2(\mu_n)} = \mathcal{O}\left( \frac{ \log(n)}{\epsilon^{d+2}n},\epsilon^{\gamma}\right).  
\end{equation*}
This completes the proof. 
\end{proof}

\section{Proof of Lemma~\ref{truncated matrix eigenvalues}}\label{pfofLemma4.2}

The $(i,j)$-th entry of $L^\gamma_{\epsilon,n} - \tilde{L}^\gamma_{\epsilon,n}$ is given by a constant multiple of 
$$
\frac{\textup{Vol}(\mathcal{M}_{\epsilon^\gamma})}{\epsilon n_1} \left( \hat{k}_\epsilon(x_i,x_j) - \tilde{k}^\gamma_\epsilon(x_i,x_j)\right). 
$$
We begin by estimating the integrals in the denominator  of $\hat{k}_\epsilon$. Since the sampling setup is defined as in Definition~\ref{sampling2domain}, we need to approximate the denominator $d$ separately on integrals over $\mathcal{M}_{\epsilon^\gamma}$ and $\mathcal{M}\backslash\mathcal{M}_{\epsilon^\gamma}$.
First, let $x_j \in X_\gamma^1$, with $1 \leq j \leq n_1$ and consider 
$$
\left| \frac{\textup{Vol}(\mathcal{M}_{\epsilon^\gamma})}{n_1} \sum_{i=1}^{n_1} \epsilon^{-d/2} k_\epsilon(x_j,x_i) - \int_{\mathcal{M}_{\epsilon^\gamma}} \epsilon^{-d/2} k_\epsilon(x_j,y) dV(y) \right|
$$
We use Chernoff bounds. Considering variance over the uniform measure on $\mathcal{M}_{\epsilon^\gamma}$,
\BEA
\mathbb{E}\left[\left(\epsilon^{-d/2}  \textup{Vol}(\mathcal{M}_{\epsilon^\gamma}) k_\epsilon(x_j, \cdot) - \epsilon^{-d/2} \int_{\mathcal{M}_{\epsilon^\gamma}} \textup{Vol}(\mathcal{M}_{\epsilon^\gamma}) k_\epsilon(x_j, y) \frac{dV(y)}{\textup{Vol}(\mathcal{M}_{\epsilon^\gamma})}  \right)^2\right]  \notag \hspace{.5cm}\\
=\int_{\mathcal{M}_{\epsilon^\gamma}} \epsilon^{-d} \textup{Vol}(\mathcal{M}_{\epsilon^\gamma})^2 |k_\epsilon (x_j,y)|^2dV(y) 
- \left( \epsilon^{-d/2}  \int_{\mathcal{M}_{\epsilon^\gamma}} k_\epsilon(x_j,y) dV(y) \right)^2 \notag\hspace{-.5cm}
\EEA
Using the same argument as in Lemma \ref{matrix-eigenvalues-supplement}, and the fact that $\textup{Vol}(\mathcal{M}_{\epsilon^\gamma}) = \mathcal{O}\left(1\right)$
, we see that 
\BEA
\mathbb{P}\left(  \left| \frac{1}{n_1} \sum_{i=1}^{n_1} \epsilon^{-d/2} k_\epsilon(x_j,x_i) - \epsilon^{-d/2}\int_{\mathcal{M}_{\epsilon^\gamma}} k(x_j,y) dV(y)  \right| > \delta \right) \leq  \exp\left( - \frac{\delta^2 n}{4 C \epsilon^{-d/2}} \right) \notag.
\EEA
Balancing by choosing $\delta = \frac{ 2 \sqrt{C} \log(n_0+n_1)^{1/2}}{ \epsilon^{d/4} \sqrt{n_1}}$, where $C$ denotes an upperbound for $ m_0^{\epsilon/2} \textup{Vol}(\mathcal{M}_{\epsilon^\gamma})$ independent of $\epsilon$ and $\gamma$, we have that with probability higher than $1 - \frac{1}{(n_1 + n_0)^3}$, 
\BEA
\left| \frac{1}{n_1} \sum_{i=1}^{n_1} \epsilon^{-d/2} \textup{Vol}(\mathcal{M}_{\epsilon^\gamma}) k_\epsilon(x_j,x_i) - \int_{\mathcal{M}_{\epsilon^\gamma}} \epsilon^{-d/2} k_\epsilon(x_j,y) dV(y) \right| \notag = \mathcal{O}\left( \frac{\log(n_1+n_0)^{1/2}}{ \epsilon^{d/4} \sqrt{n_1}} \right). \notag\hspace{0cm}
\EEA
Using this for all points $x_j \in X_{\gamma}^1$, we have that with probability higher than $1 - \frac{n_1}{(n_1+n_0)^3}$, 
\BEA
\left| \frac{1}{n_1} \sum_{i=1}^{n_1} \epsilon^{-d/2} \textup{Vol}(\mathcal{M}_{\epsilon^\gamma}) k_\epsilon(x_j,x_i) - \int_{\mathcal{M}_{\epsilon^\gamma}} \epsilon^{-d/2} k_\epsilon(x_j,y) dV(y) \right| = \mathcal{O}\left( \frac{\log(n_0 + n_1)^{1/2}}{ \epsilon^{d/4} \sqrt{n_1}} \right),\notag
\EEA
for $j = 1, 2, \dots , n_1.$ Using the same argument over the $n_0$ points sampled uniformly from $X^0_\gamma \subset \mathcal{M} \setminus \mathcal{M}_{\epsilon^\gamma}$, we have that
\BEA 
\Bigg| \frac{1}{n_0} \sum_{i=n_1 + 1}^{n_1 + n_0} \epsilon^{-d/2} \textup{Vol}(\mathcal{M} \setminus \mathcal{M}_{\epsilon^\gamma}) k_\epsilon(x_j,x_i) - \int_{\mathcal{M} \setminus \mathcal{M}_{\epsilon^\gamma}} \epsilon^{-d/2} k_\epsilon(x_j,y) dV(y) \Bigg| \
 = \mathcal{O}\left( \frac{ \sqrt{\epsilon^\gamma} \log(n_0 + n_1)^{1/2}}{ \epsilon^{d/4}\sqrt{n_0}} \right), \notag 
\EEA
for $j = 1, \dots , n_1$ with probability higher than $1 - \frac{n_1}{(n_0 + n_1)^3}$, where the $\sqrt{\epsilon^\gamma}$ factor comes from the fact that $\textup{Vol}(\mathcal{M} \setminus \mathcal{M}_{\epsilon^\gamma}) = \mathcal{O}\left(\epsilon^\gamma\right)$.  
Returning to the entries of $(L^\gamma_{\epsilon,n} - \tilde{L}^\gamma_{\epsilon,n})$, we have that $\frac{\textup{Vol}(\mathcal{M}_{\epsilon^\gamma})}{\epsilon n_1} ( \hat{k}_\epsilon (x_i,x_j) - \tilde{k}^\gamma_\epsilon (x_i,x_j))$ can be simplified to 
\BEA 
 \frac{\textup{Vol}(\mathcal{M}_{\epsilon^\gamma})}{\epsilon n_1} \left( \frac{ \epsilon^{-d/2} k_\epsilon(x_i,x_j) \left( \tilde{d}^{\gamma}(x_i) + \tilde{d}^\gamma(x_j) - d(x_i) - d(x_j) \right) }{2(d(x_i) + d(x_j))(\tilde{d}^\gamma(x_i) + \tilde{d}^\gamma(x_j))}\right). \notag
\EEA
The terms in the denominator are bounded away from $0$ by a constant, and the difference of $d, \tilde{d}^\gamma$ terms in the numerator are $\mathcal{O}\left( \frac{  \sqrt{\epsilon^\gamma} \log(n_0 + n_1)^{1/2} }{\epsilon^{d/4} \sqrt{n_0}}, \frac{ \log(n_0 + n_1)^{1/2} }{\epsilon^{d/4} \sqrt{n_1}} \right)$. Hence, the above is bounded by  
$$
\left(\frac{C_1 \log(n_0 + n_1)^{1/2}}{\epsilon \epsilon^{d/4} n_1^{1/2}} + \frac{C_2 \sqrt{\epsilon^\gamma} \log(n_0 + n_1)^{1/2}}{\epsilon \epsilon^{d/4} n_0^{1/2}} \right) \left(  \frac{\epsilon^{-d/2} \textup{Vol}(\mathcal{M}_{\epsilon^\gamma}) k_\epsilon(x_i,x_j)}{n_1} \right),
$$
for some constants $C_1,C_2>0$. The same calculation as in Lemma \ref{matrix-eigenvalues-supplement} shows that 
\BEA
\sup_{\|v\|_2 = 1} \| L_{\epsilon,n} v - \tilde{L}_{\epsilon,n} v \|_2 = \mathcal{O}\left(\frac{ \log(n_0 + n_1)^{1/2}}{ \epsilon^{d/2+1} n_1^{1/2}}, \frac{ \sqrt{\epsilon^\gamma} \log(n_0 + n_1)^{1/2}}{\epsilon^{d/2+1} n_0^{1/2}} \right) \notag
\EEA
The desired result then follows from Weyl's inequality.

\section{Convergence of Dirichlet Eigenvectors}\label{sm:conveigvec} \noindent

To see the convergence of Dirichlet eigenvectors, we require a result analogous to Lemma \ref{discrete norm lemma} for the truncated graph Laplacian. In this case, eigenvector error is measured only on those points sufficiently far away from the boundary, as this is where the truncated graph Laplacian is defined. Such a result is obtained by law of large numbers results applied to Corollary~\ref{L2boundcorollary}.
\begin{lemma}
\label{truncated discrete norm}
Suppose $n_1$ points are sampled uniformly from $\mathcal{M}_{\epsilon^\gamma}$, and $n_0$ points are sampled uniformly from $\mathcal{M} \setminus \mathcal{M}_{\epsilon^\gamma}$. Denote by $\| \cdot \|_{L^2(\mu_n, \gamma)}$ the norm induced by innerproduct on the points whose distance from the boundary are at least $\epsilon^\gamma$ defined as,
$$\langle R_n f, R_n \phi \rangle_{L^2(\mu_n,\gamma)} = \frac{1}{n} \sum_{j=1}^{n_1} f(x_j) \phi(x_j).$$ Let $f \in C^\infty(\mathcal{M})$, then with probability higher than $1 - \frac{2}{n_1^2} - \frac{2n_1}{(n_0 + n_1)^3}$, 
$$
\|\tilde{L}^\gamma_{\epsilon,n} R_n f - R_n \Delta f \|_{L^2(\mu_n,\gamma)} = \mathcal{O}\left(\frac{\sqrt{\log(n_0 + n_1)}}{\epsilon^{d/2+1} \sqrt{n_1}}, \frac{ \sqrt{\epsilon^\gamma} \sqrt{\log(n_0 + n_1)}}{\epsilon^{d/2 + 1}\sqrt{n_0}}, \epsilon \right). 
$$
\end{lemma}
\begin{proof}
Consider the following:
\BEA
\|\widetilde{L}^\gamma_{\epsilon,n} R_n f - R_n \Delta f\|_{L^2(\mu_{n},\gamma)} \leq \|\widetilde{L}^\gamma_{\epsilon,n} R_n f -  L^\gamma_{\epsilon,n} R_n  f\|_{L^2(\mu_n,\gamma)} 
+ \|L^\gamma_{\epsilon,n} R_n f - R_n L_\epsilon  f\|_{L^2(\mu_n,\gamma)} + \|R_n L_{\epsilon} f - R_n \Delta f\|_{L^2(\mu_n,\gamma)}\notag.
\EEA
The last two terms are shown to be small following the exact same argument in Lemma \ref{discrete norm lemma} on those $n_1$ points sampled uniformly away from the boundary. In this case, the mean of the random variable associated with the last term is $\mathcal{O}(\epsilon)$, by Corollary~\ref{L2boundcorollary}, and both concentrations hold simultaneously with probability higher than $1 - \frac{2}{n^2_1}$. An upper bound for the first term was shown in the proof of Lemma  \ref{truncated matrix eigenvalues}, which is $\mathcal{O}\left(\frac{\sqrt{\log(n_0 + n_1)}}{\epsilon^{d/2+1} \sqrt{n_1}}, \frac{ \sqrt{\epsilon^\gamma} \sqrt{\log(n_0 + n_1)}}{\epsilon^{d/2 + 1}\sqrt{n_0}} \right)$ with probability higher than $1 - \frac{2n_1}{(n_0 + n_1)^3}$. 
This completes the proof. 
\end{proof}
The convergence of Dirichlet eigenvector result now follows the exact same proof of Theorem~\ref{conveigvec}, with Lemma \ref{discrete norm lemma} replaced by Lemma \ref{truncated discrete norm}, and Theorem \ref{spectralconvneumann} replaced with Theorem \ref{spectralconvdiri uniform}, using the same reasoning as before to convert independent samples from $\mathcal{M}_{\epsilon^\gamma}$ and $\mathcal{M}\setminus \mathcal{M}_{\epsilon^\gamma}$ to a uniform setup. We state the end result here for convenience. 
\begin{theorem}
\label{conveigvec-dirichlet-supplement}
Let $\Delta$ denote the Dirichlet Laplacian, and  fix $0 < \gamma < 1/2$. For any $\ell$, there is a constant $c_\ell$ such that if $C'\left(\epsilon^{1/2} + \epsilon^{3 \gamma -1 } + \frac{\sqrt{ \log n}}{\epsilon^{d/2 + 1}\sqrt{n}} \right) < c_\ell$, then with probability higher than $1 - \frac{2k^2 + 4k + 12}{n^2} -  \frac{8(2 \sqrt{n} log(n) + 2n)}{n^3}$, for any normalized eigenvector $u$ of $\widetilde{L}^\gamma_{\epsilon,n}$ with eigenvalue $\widetilde{\lambda}_\ell^{\gamma,\epsilon,n}$, there is a normalized eigenfunction $f$ of $\Delta$ with eigenvalue $\lambda_\ell$ such that 
\begin{eqnarray}
\| R_n f - u \|_{L^2(\mu_n,\gamma)} = \mathcal{O}\left( \frac{\sqrt{\log n}}{\epsilon^{d/2 + 1}\sqrt{n}}, \epsilon \right)\label{eigvecderr}
\end{eqnarray}
as $\epsilon \to 0$ after $n \to \infty$. 
\end{theorem}

{\color{black}
\begin{remark}\label{remonDir} This result suggests that the eigenvector error rate is faster than that in the Neumann case as reported in Theorem~\ref{conveigvec}. This is not surprising because the $L^2$-norm is computed over data points that are $\epsilon^\gamma$ away from the boundary, where the dominating error is of order-$\epsilon$ as reported in Corollary~\ref{L2boundcorollary}. Based on the error in the eigenvalues (see Remark~\ref{valid parameter regimes}), $\epsilon^{1/2} \sim \left(n^{-1}\log n\right)^{\frac{1}{2d+6}}$, which is slower than $\epsilon$ obtained by balancing the two terms in \eqref{eigvecderr}. Plugging this rate into the two error terms in \eqref{eigvecderr}, the dominating (or slowest) rate is the first term, namely of order $(n^{-1}\log n)^{\frac{1}{2d+6}}$. This result justifies the empirical results observed in Figures~\ref{fig2}, \ref{fig3}, and \ref{fig4}, that the square of the $L^2$-errors of eigenvectors defined in \eqref{empiricalerror} is two times faster than the convergence rate of the eigenvalues.   
\end{remark}
}

\section{Generalization to Non-Uniform Sampling}\label{sec5} 

In this section, we generalize our results to non-uniform sampling data. In \cite{vaughn2024diffusion},\cite{v:2020}, it is shown that when the data is sampled non-uniformly (i.e., from some smooth density function $q(x)$ nonvanishing on $\mathcal{M}$ that is absolutely continuous w.r.t volume measure), defining
\BEA
K_{q,\epsilon} f (x) = \epsilon^{-d/2} \int_\mathcal{M} k_\epsilon (x,y) f(y) q(y) dV(y), \label{Kqe}
\EEA
then the unnormalized graph Laplacian construction converges to the weighted, weak form of the Laplacian, 
\BEA
\frac{2}{m_2 \epsilon^{d/2+1}} \int_\mathcal{M} \phi(x) \left((K_{q,\epsilon} 1)(x)f(x) - K_{q,\epsilon}f \right)q(x)dV(x) \ = \int_\mathcal{M}  (\nabla \phi \cdot \nabla f)  q^2 dV 
+ \mathcal{O}\left(\epsilon^{1/2} \right).\notag
\EEA

In this paper, we will consider a symmetrization of the normalized graph Laplacian formulation, generalizing Lemma~\ref{weak convergence}. Particularly, we perform normalizations on the matrix to eliminate the effect of sampling density. Based on the Theorem~4.7 in  \cite{vaughn2024diffusion}, the operator in Equation  \eqref{Kqe} can be expanded as follows:
\BEA
K_{q,\epsilon} f(x) &=& m_0^\partial(x) f(x)q(x) + \epsilon^{1/2}m_1^\partial(x)\Big( \frac{\partial (f(x)q(x))}{\partial \eta_x}+\frac{d-1}{2}H(x)f(x)q(x)\Big) \notag \\ &+& \epsilon \frac{m_2}{2} \left(\omega(x)f(x)q(x) + \Delta (f(x)q(x))+ \Big(\frac{m_2^\partial(x)}{m_2}-1\Big)\frac{\partial^2}{\partial\eta_x^2}(f(x)q(x)) \right) + \mathcal{O}(\epsilon^{3/2}). \notag
\EEA
Applying the above formula to the constant function $1$, we obtain 
\BEA
K_{q,\epsilon}   1(x) = m_0^\partial(x)q(x) + \epsilon^{1/2}m_1^\partial(x) \frac{d-1}{2}H(x)q(x) + \epsilon \frac{m_2}{2} \left(\omega(x)q(x)+{\Big(\frac{m_2^\partial(x)}{m_2}-1\Big)\frac{\partial^2}{\partial\eta_x^2}q(x)}\right) +\mathcal{O}(\epsilon^{3/2}).\notag
\EEA
Therefore, we have 
\BEA
q_{\epsilon}(x) &:=& \frac{K_{q,\epsilon}   1(x)}{m_0^\partial(x)} 
=  q(x) + \epsilon^{1/2}m_1^\partial(x) \frac{d-1}{2}H(x)q(x)(m_0^\partial(x))^{-1}  
\notag \\ &+& \epsilon \frac{m_2}{2} \left(\omega(x)q(x) +{\Big(\frac{m_2^\partial(x)}{m_2}-1\Big)\frac{\partial^2}{\partial\eta_x^2}q(x)}\right)(m_0^\partial(x))^{-1} +\mathcal{O}(\epsilon^{3/2}). \notag
\EEA
as an estimator of the density $q$.
Notice that 
\BEA
 f(x) q_\epsilon(x)^{-1} = f(x)q(x)^{-1} (1-\epsilon^{1/2}D(x)+\epsilon (-E(x)+D^2(x)) + \mathcal{O}(\epsilon^{3/2})) \notag
\EEA
where $D(x) =  m_1^\partial(x) \frac{d-1}{2}H(x)(m_0^\partial(x))^{-1} $ and $$E(x) =\frac{m_2}{2} \left(\omega(x)+{\Big(\frac{m_2^\partial(x)}{m_2}-1\Big)q^{-1}(x)\frac{\partial^2}{\partial\eta_x^2}q(x)}\right)(m_0^\partial(x))^{-1} $$. We therefore have that 
\BEA
K_{q,\epsilon}(f(x) q_\epsilon(x)^{-1} ) = K_\epsilon f(x) - \epsilon^{1/2}K_{\epsilon} (f(x)D(x))+\epsilon K_\epsilon(-f(x)E(x)+f(x)D^2(x)) + \mathcal{O}(\epsilon^{3/2}). \notag
\EEA
Plugging in $f(x) = 1(x)$, we obtain
\BEA
K_{q,\epsilon}(q_\epsilon(x)^{-1} ) = K_\epsilon 1(x) - \epsilon^{1/2}K_{\epsilon} (D(x)) + \epsilon K_\epsilon(-E(x)+D^2(x)) +\mathcal{O}(\epsilon^{3/2}) \notag. 
\EEA
Using the formula 
\BEA
\frac{A + B\epsilon^{1/2} + C\epsilon + \mathcal{O}(\epsilon^{3/2})}{H + F\epsilon^{1/2} + G\epsilon + \mathcal{O}(\epsilon^{3/2})} =  \frac{A}{H} + \frac{HB - AF}{H^2}\epsilon^{1/2} + \frac{CH - AG - F(HB - AF)}{H^3}\epsilon  + \mathcal{O}(\epsilon^{3/2}), \notag
\EEA
one obtains the following expression:
 \BEA 
   \frac{K_{q,\epsilon}(f(x) q_\epsilon(x)^{-1} )}{K_{q,\epsilon}(q_\epsilon(x)^{-1} )}  = \frac{K_\epsilon(f)}{K_\epsilon(1)} + \epsilon^{1/2}\left( \frac{-K_\epsilon 1(x)K_\epsilon(f(x)D(x))+ K_\epsilon f(x)K_\epsilon D(x)  }{K_\epsilon 1(x)^2} \right) + \frac{\epsilon}{K_\epsilon 1(x)^3} P(x), \label{expand0} \notag
 \EEA
 where 
 \BEA
P(x) &=&  -K_\epsilon 1(x) K_\epsilon(f(x)E(x))  + K_\epsilon 1(x) K_\epsilon(f(x)D^2(x))  + K_\epsilon f(x)K_\epsilon (E(x)) -K_\epsilon f(x)K_\epsilon (D^2(x)) \notag \\ &+& K_\epsilon(D(x))\left( -K_\epsilon 1(x) K_\epsilon (f(x)D(x)) + K_\epsilon (f(x)) K_\epsilon(D(x)) \right) = \mathcal{O}(\epsilon^{1/2}),\notag
\EEA
since the $\mathcal{O}\left(1\right)$ terms of $P$ cancel. Therefore, 
\BEA
\frac{\epsilon}{K_\epsilon 1(x)^3} P(x)  = \mathcal{O}\left( \epsilon^{3/2} \right). \notag
\EEA
Looking now at the terms
\BEA
\left( \frac{-K_\epsilon 1(x)K_\epsilon(f(x)D(x))+ K_\epsilon f(x)K_\epsilon D(x)  }{K_\epsilon 1(x)^2} \right), \notag
\EEA
we see again that the $\mathcal{O}\left( 1 \right)$ terms cancel. For the $\epsilon^{1/2}$ terms, we are left with 
\BEA
\frac{-K_\epsilon 1(x)K_\epsilon(f(x)D(x))+ K_\epsilon f(x)K_\epsilon D(x)  }{K_\epsilon 1(x)^2} = \frac{1}{K_\epsilon 1(x)^2} \left( \epsilon^{1/2} R + \mathcal{O}\left( \epsilon \right) \right). \notag
\EEA
where 
\BEA
R = m_0^\partial m_1^\partial \left(-\frac{\partial (Df)}{\partial \eta_x} + f \frac{\partial D}{\partial \eta_x} + D \frac{\partial f }{\partial \eta_x}\right). \notag
\EEA
Using product rule on the inside of the above, we see that this term cancels. Hence, 
\BEA
\epsilon^{1/2} \left( \frac{-K_\epsilon 1(x)K_\epsilon(f(x)D(x))+ K_\epsilon f(x)K_\epsilon D(x)  }{K_\epsilon 1(x)^2} \right) = \mathcal{O}\left( \epsilon^{3/2} \right).\notag
\EEA
Putting together the information above, we conclude
\BEA
\label{q pointwise}
\frac{K_{q,\epsilon}(f(x) q_\epsilon(x)^{-1} )}{K_{q,\epsilon}(q_\epsilon(x)^{-1} )} = \frac{K_\epsilon(f)}{K_\epsilon(1)} + \mathcal{O}(\epsilon^{3/2}).
\EEA
 We now define
\BEA
L_{\textrm{rw}, q, \epsilon} f (x) : = \frac{1}{\epsilon} \left(f(x) - \frac{K_{q,\epsilon}(f q_\epsilon^{-1} )(x)}{K_{q,\epsilon}(q_\epsilon^{-1} )(x)} \right), \label{intop_L_rwqe}
\EEA 
as well as 
\BEA
L_{q,\epsilon} f = \frac{1}{2} \left( L_{\textrm{rw}, q, \epsilon}f + L^*_{\textrm{rw}, q, \epsilon}f  \right),\label{intop_L_qe}
\EEA
where $L^*_{\textrm{rw}, q, \epsilon}$ denotes the adjoint of  $L_{\textrm{rw}, q, \epsilon}f $ with respect to $L^2(\mathcal{M})$. In particular, one can verify that,
\[L_{q,\epsilon}f(x) = \frac{1}{\epsilon} \Big(f(x) - \int_{\mathcal{M}}\hat{k}_{q,\epsilon}(x,y)f(y)dV(y)\Big),\]
where
\BEA
\hat{k}_{q,\epsilon}(x,y) = \frac{\epsilon^{-d/2}}{2}k_\epsilon (x,y)\left( \frac{ q(y) q^{-1}_\epsilon(y)}{K_{q,\epsilon}(q_\epsilon(x)^{-1} )} + \frac{q^{-1}_\epsilon(x)q(x)}{K_{q,\epsilon}(q_\epsilon(y)^{-1})} \right).\label{khat_qe}
 \EEA
Combining the above observations, we have the following result.
\begin{lemma}
\label{nonuniform weak}
for $f \in C^\infty(\mathcal{M})$, 
\BEA
L_{\textup{rw},q, \epsilon } f (x) = L_{\textup{rw},\epsilon} f (x) + \mathcal{O}\left(\epsilon^{1/2}\right), \notag
\EEA
as $\epsilon\to 0$. Moreover, for $f,\phi \in C^\infty(\mathcal{M})$, we have 
\BEA
\langle L_{\textup{rw},q,\epsilon}f, \phi \rangle_{L^2(\mathcal{M})} =  \langle L_{\textup{rw},\epsilon}f, \phi \rangle_{L^2(\mathcal{M})} + \mathcal{O}\left( \epsilon^{1/2} \right) = \int_\mathcal{M} \nabla f \cdot \nabla \phi dV + \mathcal{O}\left( \epsilon^{1/2} \right), \notag
\EEA
as well as 
\BEA
\langle L_{q,\epsilon}f, \phi \rangle_{L^2(\mathcal{M})} =  \langle L_{\epsilon}f, \phi \rangle_{L^2(\mathcal{M})} + \mathcal{O}\left( \epsilon^{1/2} \right) = \int_\mathcal{M} \nabla f \cdot \nabla \phi dV + \mathcal{O}\left( \epsilon^{1/2} \right), \notag 
\EEA
as $\epsilon\to 0$.
\end{lemma}
This result is analogous to Lemma~\ref{weak convergence}, which is needed to prove the spectral convergence.

\subsection{Discretization} \noindent 
In the previous discussion, we denoted $\mu_n$ as the empirical measure corresponding to uniformly sampled data (see Equation \eqref{empiricalmeasure}). In general, the data need not be restricted to uniformly sampled distribution. In the following discussion, we assume that $x_i\sim \mu$, where $d\mu = q dV$, and the same empirical measure $\mu_n$ in Equation \eqref{empiricalmeasure} corresponds to data sampled from $\mu$. 

For simplicity, we assume that $m_0^\partial$ is known. Methods of approximating $m^\partial_0$ can be found in \cite{bs:2017},\cite{vaughn2024diffusion}. 
Define 
\BEA
q_{\epsilon,n}(x_i) : = \frac{ \frac{1}{n} \sum_{j=1}^n \epsilon^{-d/2} k(x_i,x_j)}{m^\partial_0(x_i)}. \notag
\EEA
The following lemma is an immediate consequence of Hoeffding's inequality. 
\begin{lemma}
\label{q estimate} With probability $1 - \frac{2}{n^2}$, for all $1 \leq i \leq n$ and any fixed $\epsilon>0$, we have
\BEA
\label{qen to qe}
q_{\epsilon,n}(x_i) = q_\epsilon(x_i) + \mathcal{O}\left(\frac{ \sqrt{\log n}}{\epsilon^{d/2} \sqrt{n}} \right),
\EEA
as $n\to\infty$, and hence
\BEA
\label{qen to q}
q_{\epsilon,n}(x_i) = q(x_i) + \mathcal{O}\left(\frac{\sqrt{\log n}}{\epsilon^{d/2} \sqrt{n}} , \epsilon^{1/2} \right),
\EEA
as $\epsilon\to\infty$ after $n\to\infty$.
\end{lemma}
Define the matrix 
\BEA
 \qquad \quad \tilde{L}_{\textrm{rw},q,\epsilon,n} u(x_i) = \frac{1}{\epsilon} \left( u(x_i) - \frac{  \frac{1}{n} \sum_{j=1}^n \epsilon^{-d/2} k_\epsilon (x_i,x_j) u(x_j) q^{-1}_{\epsilon,n}(x_j) }{ \frac{1}{n} \sum_{k=1}^n \epsilon^{-d/2} k_\epsilon(x_i,x_k) q^{-1}_{\epsilon,n}(x_k) } \right),\label{matrix_L_rwqen}
\EEA
 as a Monte-Carlo approximation to Equation \eqref{intop_L_rwqe} with $x_i\sim \mu$. For the discussion below, we also consider a weighted empirical measure that approximates the uniform measure,
\[   \nu_{n} : = \frac{1}{n} \sum_{i=1}^n \frac{\delta_{x_i}}{q(x_i)}.\]
Since $q$ is unknown and will be empirically estimated by $q_{\epsilon,n}$, we will consider the empirical measure weighted as follows
\[   \nu_{\epsilon,n} : = \frac{1}{n} \sum_{i=1}^n \frac{\delta_{x_i}}{q_{\epsilon,n}(x_i)},\]
and let the operator $\tilde{L}_{\textrm{rw},q,\epsilon,n}$ acting between the innerproduct space $ L^2(\nu_{\epsilon,n}) \to L^2(\nu_{\epsilon,n})$, with
\BEA
\langle u,v \rangle_{L^2(\nu_{\epsilon,n})} : = \frac{1}{n} \sum_{i=1}^n \frac{u(x_i)v(x_i)}{q_{\epsilon,n}(x_i)}.\notag
\EEA
Following the uniformly sampled case, we also define the symmetric matrix,
\BEA
\tilde{L}_{q,\epsilon,n} = \frac{1}{2} \left( \tilde{L}_{\textrm{rw},q,\epsilon,n} + \tilde{L}^*_{\textrm{rw},q,\epsilon,n} \right),\label{mat_L_qen}
\EEA
where the adjoint $ \tilde{L}^*_{\textrm{rw},q,\epsilon,n}$ is taken in $L^2(\nu_{\epsilon,n})$.

\subsection{Convergence of Eigenvalues} \noindent
The spectral convergence, generalizing Theorem~\ref{spectralconvneumann} to non-uniformly sampled data, is given as follows:

\begin{theorem}\label{spectconvwithq}
Let $\mathcal{M}$ be a manifold without boundary or with Neumann boundary condition. Let $\lambda_i$ and $\tilde{\lambda}^{q,\epsilon,n}_{i}$ denote the $i$th eigenvalues of the Laplace-Beltrami operator, $\Delta$, and the matrix $\tilde{L}_{q,\epsilon,n}$, respectively. With probability greater than $1 - \frac{12}{n^2}$, for any $1 \leq i \leq n$ we have that 
 \begin{equation*}
 |\lambda_{i} - \tilde{\lambda}^{q,\epsilon,n}_{i}| = \mathcal{O}\left( \frac{\sqrt{\log n}}{\epsilon^{d/2+1}n^{1/2}},\epsilon^{1/2} \right),
\end{equation*}
as $\epsilon\to 0$ after $n\to \infty$.
\end{theorem}
The proof for the no boundary or Neumann case follows the same argument as that of Theorem~\ref{spectralconvneumann}. In the current case, we need to deduce several bounds that are analogous to Lemmas \ref{adapted-rbv-supplement}, \ref{matrix-eigenvalues-supplement}, and \ref{weak convergence}. Thus far, we already have Lemma~\ref{nonuniform weak} which is the analogous result to Lemma~\ref{weak convergence}. In particular, we need to deduce bounds by a min-max argument over the following identity,
\BEA
 \| \nabla f  \|^2_{L^2(\mathcal{M})} - \tilde{\lambda}^{q,\epsilon,n}_i = \underbrace{\| \nabla f  \|^2_{L^2(\mathcal{M})}  -  \langle L_{q,\epsilon} f, f \rangle_{L^2(\mathcal{M})}}_{\text{Lemma \ref{nonuniform weak}}}   +   \underbrace{\langle L_{q,\epsilon} f , f  \rangle_{L^2(\mathcal{M})}  - \tilde{\lambda}^{q,\epsilon,n}_i.}_{\text{discretization error}}  \label{genineq2}
\EEA
In the following sequence, we will deduce a spectral error bound induced by the discretization error. Here, the additional difficulty is that since the data is non-uniformly distributed with unknown density, $q$, we need to account for the error induced by the density estimate $q_{\epsilon,n}$ in Equation \eqref{mat_L_qen}. In particular, we would like to ensure that the $i$th eigenvalue of $L_{q,\epsilon}$ is close to $\tilde{\lambda}^{q,\epsilon,n}_i$. To achieve this, we will consider an intermediate matrix, which is a Monte-Carlo discretization of $L_{q,\epsilon}$ reweighted with the sampling density $q$, defined as,  
\BEA
L_{q,\epsilon,n} = \frac{2}{m_2 \epsilon}\left( I - K_{q,\epsilon,n} Q^{-1}\right), \label{Rosascomatrix}
\EEA 
where $(K_{q,\epsilon,n} )_{ij} = \hat{k}_{q,\epsilon}(x_i,x_j)$ is defined in Equation  \eqref{khat_qe} and $Q$ is the diagonal matrix $Q = \textrm{diag}(q(x_1), \dots q(x_n))$. It is an easy calculation to see that $L_{q,\epsilon,n}: L^2(\nu_n) \to L^2(\nu_n)$ is self-adjoint in $L^2(\nu_n)$. Let $\lambda_i^{q,\epsilon}$ and $\lambda_i^{q,\epsilon,n}$ denote the $i$-th eigenvalues of the integral operator $L_{q,\epsilon}$ and the matrix $L_{q,\epsilon,n}$, respectively. We will control the error induced by the discretization in Equation \eqref{genineq2} by the following spectral differences 
\[
\Big|\lambda_i^{q,\epsilon} - \tilde{\lambda}^{q,\epsilon,n}_i\Big| \leq \underbrace{ \Big|\lambda_i^{q,\epsilon} - \lambda^{q,\epsilon,n}_i \Big|}_{\text{see Lemma~\ref{rosascoq}}} + \underbrace{\Big| \lambda^{q,\epsilon,n}_i- \tilde{\lambda}^{q,\epsilon,n}_i\Big|}_{\text{see Lemma~\ref{finalmatrixdiff}}},
\]
where the upper bounds for the last two terms above will be deduced in Lemmas~\ref{rosascoq} and \ref{finalmatrixdiff} below, respectively.

Specifically, for the first spectral error difference above, we have an analogous result to Lemma~\ref{adapted-rbv-supplement} since the matrix $Q^{-1}$ correction in Equation \eqref{Rosascomatrix} yields a Monte-Carlo integration with uniform sampling, 
\begin{lemma} \label{rosascoq}
With probability greater than $1 - \frac{2}{n^2}$, 
\begin{equation*}
    \sup_{1\leq i \leq n } |\lambda^{q,\epsilon}_i - \lambda_i^{q,\epsilon,n}| = \mathcal{O} \Big(\frac{\sqrt{\log n}}{\epsilon^{d/2 + 1} \sqrt{n}}\Big),
\end{equation*}
as $n\to \infty$ for a fixed $\epsilon>0$.
\end{lemma}
To deduce an upper bound for $\Big| \lambda^{q,\epsilon,n}_i- \tilde{\lambda}^{q,\epsilon,n}_i\Big|$,

(or Lemma~\ref{finalmatrixdiff} below), we first deduce several weak convergence results.
\begin{lemma}
\label{rosasco matrix to operator weak}
With probability $1 -\frac{2}{n^2}$,
\BEA
\left| \langle L_{q,\epsilon,n} R_n f , R_n \phi \rangle_{L^2(\nu_n)} - \langle L_{q,\epsilon} f, \phi \rangle_{L^2(\mathcal{M})}  \right| = \mathcal{O}\left( \frac{\sqrt{\log n}}{\epsilon^{d/2+1}\sqrt{n}} \right). \notag
\EEA 
as $\epsilon \to 0$ after $n \to \infty$. 
\end{lemma}
\begin{proof}
This is an immediate application of Hoeffding's inequality. Indeed, with high probability, 
\BEA
\left| \frac{1}{n} \sum_{j=1}^n \hat{k}_{q,\epsilon}(x_i,x_j) \frac{f(x_j)}{q(x_j)} - \int_\mathcal{M} \hat{k}(x_i,y) f(y) dV(y) \right| = \mathcal{O}\left( \frac{\sqrt{\log n}}{\epsilon^{d/2} \sqrt{n}}\right). \notag
\EEA
First, we repeat this for all $i$. Second, we obtain a similar bound again using Hoeffding for the outside integral. This yields the result. 
\end{proof} 
\begin{lemma}
With probability $1 - \frac{8}{n^2}$ , we have 
\BEA
\left| \langle \tilde{L}_{\textup{rw},q,\epsilon,n} R_n f, R_n \phi \rangle_{L^2(\nu_{\epsilon,n})} - \langle L_{\textup{rw},q,\epsilon}f, \phi \rangle_{L^2(\mathcal{M})}  \right| = \mathcal{O}\left( \frac{\sqrt{\log n}}{\epsilon^{d/2+1} \sqrt{n}}, \epsilon^{1/2}  \right), \notag
\EEA
as $\epsilon\to 0$ after $n\to\infty$. Here, the matrix $\tilde{L}_{\textup{rw},q,\epsilon,n}$ and integral operator $L_{\textup{rw},q,\epsilon}$ are defined in Equations \eqref{matrix_L_rwqen} and \eqref{intop_L_rwqe}, respectively. 
\end{lemma}
\begin{remark}
We should point out that the extra factors of $\epsilon$ in the denominator are due to the error in the estimation of the sampling density. The same rate will apply to the next two lemmas.
\end{remark}
\begin{proof}
Notice that 
\BEA
\Bigg| \langle \tilde{L}_{\textrm{rw},q,\epsilon,n} R_n f, R_n \phi \rangle_{L^2(\nu_{\epsilon,n})} - \langle L_{\textrm{rw},q,\epsilon}f, \phi \rangle_{L^2(\mathcal{M})}  \Bigg| &\leq& \left| \langle \tilde{L}_{\textrm{rw},q,\epsilon,n} R_n f - R_n L_{\textrm{rw},q,\epsilon}f, R_n \phi \rangle_{L^2(\nu_{\epsilon,n})} \right| \notag \\
&+& \left| \langle R_n L_{\textrm{rw},q,\epsilon}f, R_n \phi \rangle_{L^2(\nu_{\epsilon,n})} - \langle L_{\textrm{rw},q,\epsilon}f, \phi \rangle_{L^2(\mathcal{M})}  \right|. \notag
\EEA
We begin by proving the first term is $\mathcal{O}\left( \frac{\sqrt{\log n}}{\epsilon^{d/2+1} \sqrt{n}} \right)$. With probability $1 - \frac{2}{n^2}$, 
\BEA
\frac{1}{q_{\epsilon,n}(x_i)} = \frac{1}{q_\epsilon (x_i)} + \mathcal{O}\left(\frac{\sqrt{\log n}}{\epsilon^{d/2} \sqrt{n}} \right) \label{error1/q}\notag
\EEA
for each $1 \leq i \leq n$, from Equation \eqref{qen to qe}. Hence, we can see that 
\BEA
\frac{1}{n}\sum_{k=1}^n \epsilon^{-d/2} k(x_i,x_k) q^{-1}_{\epsilon,n}(x_k) R_nf(x_k) - \frac{1}{n} \sum_{k=1}^n \epsilon^{-d/2} k(x_i,x_k) q^{-1}_{\epsilon}(x_k) R_n f(x_k) \notag 
\EEA
is $\mathcal{O}\left( \frac{\sqrt{\log n}}{\epsilon^{d/2} \sqrt{n}} \right)$, where the average of $\epsilon^{-d/2} k(x_i,x_k)f(x_k)$ is order-1. 
However, again from Hoeffding's inequality we have that with probability $1 - \frac{2}{n^2}$
\BEA
\Bigg|\frac{1}{n} \sum_{k=1}^n \epsilon^{-d/2} k(x_i,x_k) q^{-1}_{\epsilon}(x_k) R_n f(x_k) - \int_\mathcal{M} \epsilon^{-d/2} k(x_i,y)f(y) q^{-1}_\epsilon(y) q(y) dV(y) \Bigg|  = \mathcal{O}\left( \frac{\sqrt{\log n}}{\epsilon^{d/2} \sqrt{n}} \right). \notag
\EEA
This argument is also valid when $f = 1$. From these two bounds above, it follows that, with a total probability of $1 - \frac{4}{n^2}$, we have that 
\BEA
\left| \frac{1}{n}\sum_{k=1}^n \epsilon^{-d/2} k(x_i,x_k) q^{-1}_{\epsilon,n}(x_k) R_nf(x_k) - K_{q,\epsilon}(f q^{-1}_\epsilon)(x_i) \right| =  \mathcal{O}\left( \frac{\sqrt{\log n}}{\epsilon^{d/2} \sqrt{n}} \right), \notag
\EEA
and
\BEA
\left| \frac{1}{n}\sum_{k=1}^n \epsilon^{-d/2} k(x_i,x_k) q^{-1}_{\epsilon,n}(x_k) - K_{q,\epsilon}( q^{-1}_\epsilon)(x_i) \right| =  \mathcal{O}\left( \frac{\sqrt{\log n}}{\epsilon^{d/2} \sqrt{n}} \right), \notag
\EEA
for all $1 \leq i \leq n$. With similar considerations as before, Taylor expansions show that with probability $1- \frac{4}{n^2}$, 
\BEA
\left| \frac{\frac{1}{n}\sum_{k=1}^n \epsilon^{-d/2} k(x_i,x_k) q^{-1}_{\epsilon,n}(x_k) R_nf(x_k)}{\frac{1}{n}\sum_{k=1}^n \epsilon^{-d/2} k(x_i,x_k) q^{-1}_{\epsilon,n}(x_k)} - \frac{K_{q,\epsilon}(f q^{-1}_\epsilon)(x_i)}{K_{q,\epsilon}( q^{-1}_\epsilon)(x_i)} \right| = \mathcal{O}\left( \frac{\sqrt{\log n}}{\epsilon^{d/2} \sqrt{n}} \right). \notag
\EEA
Hence
\BEA
\left| \tilde{L}_{\textrm{rw},q,\epsilon,n}R_n f (x_i) - R_n L_{\textrm{rw},q,\epsilon} f (x_i)  \right| = \mathcal{O}\left(\frac{\sqrt{\log n}}{\epsilon^{d/2+1} \sqrt{n}} \right), \notag
\EEA
for all $1 \leq i \leq n$. Therefore, 
\BEA
\Bigg| \langle \tilde{L}_{\textrm{rw},q,\epsilon,n} R_n f -  R_nL_{\textrm{rw},q,\epsilon}f, R_n \phi \rangle_{L^2(\nu_{\epsilon,n})}\Bigg|  = \left|\frac{1}{n} \sum_{i=1}^n \frac{\phi(x_i) \left(  \tilde{L}_{\textrm{rw},q,\epsilon,n}R_n f (x_i) - R_n L_{\textrm{rw},q,\epsilon} f (x_i) \right) }{q_{\epsilon,n}(x_i)}  \right| = \mathcal{O}\left(\frac{\sqrt{\log n}}{\epsilon^{d/2+1} \sqrt{n}}  \right). \notag
\EEA
This shows that $\left| \langle \tilde{L}_{\textrm{rw},q,\epsilon,n} R_n f - R_n L_{\textrm{rw},q,\epsilon}f, R_n \phi \rangle_{L^2(\nu_{\epsilon,n})} \right|$  is $\mathcal{O}\left( \frac{\sqrt{\log n}}{\epsilon^{d/2+1} \sqrt{n}}\right)$. For the second term, notice that 
\BEA
\Bigg|  \frac{1}{n} \sum_{i=1}^n \frac{\phi(x_i) L_{\textrm{rw},q,\epsilon}f (x_i) }{q_{\epsilon,n}(x_i)} - \langle L_{\textrm{rw},q,\epsilon}f , \phi \rangle_{L^2(\mathcal{M})} \Bigg|    
&=&  \left|  \frac{1}{n} \sum_{i=1}^n \frac{\phi(x_i) L_{\textrm{rw},q,\epsilon}f (x_i) }{q_{\epsilon,n}(x_i)} - \langle L_{\textrm{rw},q,\epsilon}f , \phi / q \rangle_{L^2(\mu)} \right| \notag \\ &\leq&    \left|  \frac{1}{n} \sum_{i=1}^n \frac{\phi(x_i) L_{\textrm{rw},q,\epsilon}f (x_i) }{q(x_i)} - \langle L_{\textrm{rw},q,\epsilon}f , \phi / q \rangle_{L^2(\mu)} \right | 
\notag \\ &+& \mathcal{O}\left(\frac{\sqrt{\log n}}{\epsilon^{d/2} \sqrt{n}} , \epsilon^{1/2} \right) \notag \\ 
&=&  \mathcal{O}\left(\frac{\sqrt{\log n}}{\sqrt{n}} ,\frac{\sqrt{\log n}}{\epsilon^{d/2} \sqrt{n}} , \epsilon^{1/2} \right), \notag 
\EEA
with probability $1 - \frac{4}{n^2}$. Here, we denote by $L^2(\mu)$ as the set of square integrable (equivalence class) of functions with respect to the sampling measure $d\mu = qdV$. The second and third error bounds above follow directly from the estimate in Equation \eqref{qen to q} and the first error bound follows
directly from Hoeffding's argument, using rough estimates to bound $\left| \frac{\phi L_{\textrm{rw},q,\epsilon } f}{q_\epsilon} - \langle L_{\textrm{rw},q,\epsilon}f , \phi / q \rangle_{L^2(\mu)} \right|$. Inserting these two error bounds to the original equation, the proof is completed. 
\end{proof}
\indent Based on this result and Lemma~\ref{nonuniform weak}, one can immediately see the weak convergence of the symmetrized matrix 
$\tilde{L}_{q,\epsilon,n}$ to the integral operator $L_{q,\epsilon}$.
\begin{corollary}
\label{weak Lqen tilde to Lqe}
With probability $1 - \frac{8}{n^2}$, 
\BEA
\left|\langle \tilde{L}_{q,\epsilon,n} R_n f , R_n \phi \rangle_{L^2(\nu_{\epsilon,n})} - \langle L_{q, \epsilon} f , \phi \rangle_{L^2(\mathcal{M})} \right| = \mathcal{O}\left( \frac{\sqrt{\log n}}{\epsilon^{d/2+1} \sqrt{n}} , \epsilon^{1/2} \right), \notag
\EEA
as $\epsilon\to 0$ after $n\to\infty$,
where the integral operator $L_{q, \epsilon}$ is defined in Equation \eqref{intop_L_qe}, and the matrix $\tilde{L}_{q,\epsilon,n}$ is defined in  Equation \eqref{mat_L_qen}. 
\end{corollary}

Given the weak consistency results given by Corollary \ref{weak Lqen tilde to Lqe}  and Lemma \ref{rosasco matrix to operator weak}, we now deduce the spectral consistency between matrices $L_{q, \epsilon,n}$ and $\tilde{L}_{q,\epsilon,n}$.

\begin{lemma}\label{finalmatrixdiff}
With probability $1 - \frac{10}{n^2}$,
\BEA
\sup_{i \leq i \leq n} | \tilde{\lambda}^{q,\epsilon,n} - \lambda^{q,\epsilon,n} | = \mathcal{O}\left( \frac{\sqrt{\log n}}{\epsilon^{d/2+1} \sqrt{n}}, \epsilon^{1/2} \right), \notag
\EEA 
as $\epsilon\to 0$ after $n\to\infty$.
\end{lemma}
\begin{proof}
Notice that, combining Corollary \ref{weak Lqen tilde to Lqe}  and Lemma \ref{rosasco matrix to operator weak}, we have that with probability $1 - \frac{10}{n^2}$, for $\phi, f \in C^\infty(\mathcal{M})$, 
\BEA
\quad\quad\left| \langle \tilde{L}_{q,\epsilon,n} R_n f , R_n \phi \rangle_{L^2(\nu_{\epsilon,n})} - \langle L_{q,\epsilon,n} R_n f, R_n \phi \rangle_{L^2(\nu_n)}  \right| = \mathcal{O}\left( \frac{\sqrt{\log n}}{\epsilon^{d/2+1}\sqrt{n}} \right).\label{weakconvmatrices}
\EEA
By construction, the eigenvalues of both operators above satisfy the min-max principle in their respective inner product spaces. Hence, we perform a similar min-max argument to before, but between the two matrices. For completion, we include it in full detail here.  
Fix some $i$ dimensional vector subspace $S = \textrm{ span } \{v_1 , \dots v_i \}$. Notice that 
\BEA
\frac{\langle \tilde{L}_{q,\epsilon,n} v , v \rangle_{L^2(\nu_{\epsilon,n})}}{\|v\|^2_{L^2(\nu_{\epsilon,n})}} =  \frac{\langle \tilde{L}_{q,\epsilon,n} v , v \rangle_{L^2(\nu_{\epsilon,n})}}{\|v\|^2_{L^2(\nu_{\epsilon,n})}} - \frac{\langle L_{q,\epsilon,n} v , v \rangle_{L^2(\nu_n)}}{\| v \|^2_{L^2(\nu_n)}} + \frac{\langle L_{q,\epsilon,n} v , v \rangle_{L^2(\nu_n)}}{\| v \|^2_{L^2(\nu_n)}}, \notag
\EEA
for any $v \in S$. Maximizing both sides over all $v \in S$ with $L^2(\nu_n)$ norm $1$, subadditivity of the maximum shows that 
\BEA
\max_{v \in S} \frac{\langle \tilde{L}_{q,\epsilon,n} v , v \rangle_{L^2(\nu_{\epsilon,n})}}{\|v\|^2_{L^2(\nu_{\epsilon,n})}}   \leq \max_{v \in S} \left(\frac{\langle \tilde{L}_{q,\epsilon,n} v , v \rangle_{L^2(\nu_{\epsilon,n})}}{\|v\|^2_{L^2(\nu_{\epsilon,n})}} - \frac{\langle L_{q,\epsilon,n} v , v \rangle_{L^2(\nu_n)}}{\| v \|^2_{L^2(\nu_n)}} \right)  +  \max_{v \in S} \frac{\langle L_{q,\epsilon,n} v , v \rangle_{L^2(\nu_n)}}{\| v \|^2_{L^2(\nu_n)}}.\notag
\EEA
Denote by $S_{\textrm{min}}$ the specific vector subspace of dimension $i$ on which the quantity
\BEA
\max_{v \in S} \frac{\langle L_{q,\epsilon,n} v , v \rangle_{L^2(\nu_n)}}{\| v \|^2_{L^2(\nu_n)}},\notag
\EEA
is minimized. We now have 
\BEA
\max_{v \in S_{\textrm{min}}} \frac{\langle \tilde{L}_{q,\epsilon,n} v , v \rangle_{L^2(\nu_{\epsilon,n})}}{\|v\|^2_{L^2(\nu_{\epsilon,n})}}   \leq \max_{v \in S_{\textrm{min}}} \left(\frac{\langle \tilde{L}_{q,\epsilon,n} v , v \rangle_{L^2(\nu_{\epsilon,n})}}{\|v\|^2_{L^2(\nu_{\epsilon,n})}} - \frac{\langle L_{q,\epsilon,n} v , v \rangle_{L^2(\nu_n)}}{\| v \|^2_{L^2(\nu_n)}} \right) +  \lambda^{q,\epsilon,n}_i.\notag
\EEA
Again, certainly the left-hand side is bounded below by the minimum over all $i$ dimensional linear subspaces $S$. Hence, 
\BEA
\tilde{\lambda}^{q,\epsilon,n}_i - \lambda_i^{q,\epsilon,n} \leq \max_{v \in S_{\textrm{min}}} \left(\frac{\langle \tilde{L}_{q,\epsilon,n} v , v \rangle_{L^2(\nu_{\epsilon,n})}}{\|v\|^2_{L^2(\nu_{\epsilon,n})}} - \frac{\langle L_{q,\epsilon,n} v , v \rangle_{L^2(\nu_n)}}{\| v \|^2_{L^2(\nu_n)}} \right). \notag
\EEA
The exact same argument can be applied to conclude that 
\BEA
\lambda^{q,\epsilon,n}_i - \tilde{\lambda}_i^{q,\epsilon,n} \leq \max_{v \in S'_{\textrm{min}}} \left(-\frac{\langle \tilde{L}_{q,\epsilon,n} v , v \rangle_{L^2(\nu_{\epsilon,n})}}{\|v\|^2_{L^2(\nu_{\epsilon,n})}} + \frac{\langle L_{q,\epsilon,n} v , v \rangle_{L^2(\nu_n)}}{\| v \|^2_{L^2(\nu_n)}} \right). \notag
\EEA
Hence, it suffices to prove a bound for 
\BEA
\label{weak finite normalized bound}
\left|\frac{\langle \tilde{L}_{q,\epsilon,n} v , v \rangle_{L^2(\nu_{\epsilon,n})}}{\|v\|^2_{L^2(\nu_{\epsilon,n})}} - \frac{\langle L_{q,\epsilon,n} v , v \rangle_{L^2(\nu_n)}}{\| v \|^2_{L^2(\nu_n)}} \right|
\EEA
for fixed $v$. Without loss of generality, assume $\|v\|^2_{L^2(\nu_n)} = 1$. Since $q_{\epsilon,n}$ estimates $q$, it follows immediately that 
\BEA
\|v\|^2_{L^2(\nu_{\epsilon,n})} = 1 + \mathcal{O}\left( \frac{\sqrt{\log n}}{\sqrt{n} \epsilon^{d/2}} , \epsilon^{1/2} \right).  \notag 
\EEA
Plugging this to Equation \eqref{weak finite normalized bound}, we obtain  
\BEA
\Bigg|\frac{\langle \tilde{L}_{q,\epsilon,n} v , v \rangle_{L^2(\nu_{\epsilon,n})}}{1 + \mathcal{O}\left( \frac{\sqrt{\log n}}{\sqrt{n} \epsilon^{d/2}}, \epsilon^{1/2} \right)} - \langle L_{q,\epsilon,n} v , v \rangle_{L^2(\nu_n)} \Bigg| &\leq&  \left|\langle \tilde{L}_{q,\epsilon,n} v , v \rangle_{L^2(\nu_{\epsilon,n})} - \langle L_{q,\epsilon,n} v , v \rangle_{L^2(\nu_n)} \right| \notag \\ &+& \langle \tilde{L}_{q,\epsilon,n} v , v \rangle_{L^2(\nu_{\epsilon,n})}\mathcal{O}\left( \frac{\sqrt{\log n}}{\sqrt{n} \epsilon^{d/2}}, \epsilon^{1/2} \right) \label{ineqlast}
\EEA
But it follows from Corollary \ref{weak Lqen tilde to Lqe} that
$$
\langle \tilde{L}_{q,\epsilon,n} v , v \rangle_{L^2(\nu_{\epsilon,n})} = \mathcal{O}\left( 1 , \frac{\sqrt{\log n}}{\epsilon^{d/2+1}\sqrt{n}} \right).
$$ 
Using the previously mentioned bound in Equation  \eqref{weakconvmatrices} on the first term of Equation \eqref{ineqlast}, using $f \in C^\infty(\mathcal{M})$ satisfying $R_n f = v$, we conclude the result.  
\end{proof}
\indent With these lemmas, we can prove  Theorem~\ref{spectconvwithq} following the argument of Theorem~\ref{spectralconvneumann} on Equation \eqref{genineq2}. This concludes our discussion on the spectral convergence.
\begin{remark}
When considering the Dirichlet case, note that now $n_0$, the number of points inside $\mathcal{M} \setminus \mathcal{M}_{\epsilon^\gamma}$, becomes a binomial random variable with success rate 
$$
\int_{\mathcal{M} \setminus \mathcal{M}_{\epsilon^\gamma}} q(y) dV(y). 
$$
Similarly for $n_1$. Since there are constants $q_{\textup{min}}$ and  $q_{\textup{max}}$ such that 
$$
q_{\textup{min}} \leq q \leq q_{\textup{max}}, 
$$
it follows that 
$$
q_{\textup{min}} \textup{Vol}(\mathcal{M} \setminus \mathcal{M}_{\epsilon^\gamma}) \leq \int_{\mathcal{M} \setminus \mathcal{M}_{\epsilon^\gamma}} q(y) dV(y) \leq \textup{Vol}(\mathcal{M} \setminus \mathcal{M}_{\epsilon^\gamma}) q_{\textup{max}}.
$$
Hence, the analysis for the Dirichlet follows the same arguments as before, with the primary difference being that constants involving $q_{\textup{min}}$ and $q_{\textup{max}}$ are absorbed into the big-oh notation in the convergence of eigenvalues. 
\end{remark}

\subsection{Convergence of Eigenvectors} \noindent 
Recall that the error analysis of the eigenvectors (see Theorem~\ref{conveigvec}) relies on the $L^2$ convergence result in Lemma~\ref{discrete norm lemma}, which is deduced from Lemma \ref{continuous norm lemma} and concentration inequalities. 
The following two lemmas are analogues of Lemma ~\ref{continuous norm lemma} and Lemma \ref{discrete norm lemma} for non-uniformly sampled data, respectively. The first follows immediately from Lemma \ref{continuous norm lemma} along with Equation \eqref{q pointwise}
\begin{lemma}
Let $0 < \gamma < 1/2$, and $\epsilon>0$ such that $\epsilon^\gamma$ is less than the radius of the normal collar, $r_C$. Let $f \in C^\infty(\mathcal{M})$ be a Neumann eigenfunction of $\Delta$. Then 
\BEA
\|L_{q,\epsilon} f - \Delta f \|_{L^2(\mathcal{M})} = \mathcal{O}\left(\epsilon^{\gamma/2}\right). \notag
\EEA
\end{lemma}

The next lemma follows using the same argument as in Lemma \ref{discrete norm lemma}, but correcting by the estimate for the sampling density $q$. This is accomplished by using the innerproduct $L^2(\nu_{\epsilon,n})$. 
\begin{lemma}
Fix $0 < \gamma < 1/2$, and let $f \in C^\infty(\mathcal{M})$ be a Neumann eigenfunction of $\Delta$. Then with probability higher than $ 1 - \frac{6}{n^2}$,   
\BEA
\|\widetilde{L}_{q,\epsilon,n} R_n f - R_n \Delta f\|_{L^2(\nu_{\epsilon,n})} = \mathcal{O}\left( \frac{\sqrt{\log n}}{\epsilon^{d/2+1}\sqrt{n}}, \epsilon^{\gamma/2} \right). \notag
\EEA
\end{lemma} 
\begin{proof}
Note first that using the same argument in Lemma \ref{discrete norm lemma}, with $L^2(\mu_n)$ norms replaced by the $L^2(\nu_n)$ norms, along with the previous lemma, one sees that 
\BEA
\| \tilde{L}_{q, \epsilon,n} R_n f - R_n \Delta f \|_{L^2(\nu_n)} = \mathcal{O}\left(\frac{\sqrt{ \log (n) }}{\epsilon^{d/2+1} \sqrt{n}}, \epsilon^{\gamma/2} \right). \notag
\EEA
To see the rate for $L^2(\nu_{\epsilon,n})$ norm, note that 
\BEA
\| \tilde{L}_{q, \epsilon,n} R_n f - R_n \Delta f \|^2_{L^2(\nu_{\epsilon,n})} = \| \tilde{L}_{q, \epsilon,n} R_n f - R_n \Delta f \|^2_{L^2(\nu_{\epsilon,n})}  
- \| \tilde{L}_{q, \epsilon,n} R_n f - R_n \Delta f \|^2_{L^2(\nu_{n})} 
+ \mathcal{O}\left(\frac{ \log (n)}{\epsilon^{d+2}n}, \epsilon^{\gamma} \right).  \notag
\EEA
Expanding the above norms and using Lemma \ref{q estimate}, one sees that the error coming from the difference of norms above is the product of $\mathcal{O}\left( \frac{1}{\epsilon^{d/2} \sqrt{n}}\right)$ and $\mathcal{O}\left(\frac{ \log (n)}{\epsilon^{d+2}n}, \epsilon^{\gamma}\right)$. Clearly the second term dominates. Taking a square root yields the result.  
\end{proof}
\noindent We now have all necessary results for the nonuniform case, as the proofs for uniform sampling generalize to this setting immediately. Without repeating the proof, we state the main result.

\begin{theorem}\label{conveigvecwithq}
Let $\Delta$ denote the Neumann Laplacian for manifold with boundary. For any $\ell$, there is a constant $c_\ell$ such that if $C' \left( \frac{\sqrt{\log n}}{\epsilon^{d/2+1}n^{1/2}} + \epsilon^{1/2}  \right) < c_\ell$, then with probability higher than $1-  \frac{2k^2 + 6k + 24}{n^2}$,  for any normalized eigenvector $u$ of $\tilde{L}_{\epsilon,n}$ with eigenvalue $\tilde{\lambda}^{\epsilon,n}_\ell$, there is a normalized eigenfunction $f$ of $\Delta$ with eigenvalue $\lambda_\ell$ such that
\begin{equation*}
    \| R_n f - u \|_{L^2(\nu_{\epsilon,n})} =   \mathcal{O}\left(\frac{ \sqrt{\log (n)} }{\epsilon^{d/2+1}\sqrt{n}},\epsilon^{\gamma/2} \right),
\end{equation*}
as $\epsilon\to 0$ after $n\to\infty$, where $k$ is the geometric multiplicity of eigenvalue $\lambda_\ell$, and $0<\gamma<1/2$. 
\end{theorem}

\begin{remark}
For closed manifold (compact with no boundary), the same result can be attained with $\epsilon^{\gamma/2}$ replaced by $\epsilon$ so we have a significantly improved estimate, extending the same conclusion in Remark~\ref{improvedrate} to non-uniformly distributed data.
For the Dirichlet eigenvalues, note that by the same reasoning as before, $L_{q,\epsilon}f$ is close in $L^2(\mathcal{M}_{\epsilon^\gamma})$ to $\Delta f$. Using the same argument as Lemma \ref{truncated discrete norm} but with discrete norms which correct for the sampling density, one can obtain an analogue for Lemma \ref{truncated discrete norm} for the nonuniform case. This allows the same argument as before to be used to obtain an analogue of Theorem~\ref{conveigvec-dirichlet-supplement} for the nonuniform case.
\end{remark}


\end{document}